\documentclass[a4paper,10pt
]{article}%

\usepackage[hidelinks]{hyperref}
\hypersetup{
  final,
  pdftitle={Equivariant Dendroidal Segal Spaces},
  pdfauthor={Bonventre, P. and Pereira, L. A.},
  linktoc=page
}

\usepackage[open=false]{bookmark}

\usepackage{amsmath, amsfonts, amssymb, amsthm}%

\usepackage[all,cmtip]{xy}%
\newdir{ >}{{}*!/-5pt/@{>}}%
\usepackage{array}%
\usepackage[dvipsnames]{xcolor}%
\usepackage[normalem]{ulem}%
\usepackage{dsfont}%
\usepackage{bbm}%
\usepackage{paralist}%

\usepackage{relsize}%

\usepackage{geometry}

\usepackage[latin1]{inputenc}%
\usepackage{MnSymbol}
\usepackage{stmaryrd}
\usepackage{upgreek}
\usepackage{mathrsfs}

\usepackage{enumerate}%
\usepackage[inline]{enumitem}%

\numberwithin{equation}{section} 
\numberwithin{figure}{section}

\newtheorem{theorem}[equation]{Theorem}%
\newtheorem*{theorem*}{Theorem}%
\newtheorem{lemma}[equation]{Lemma}%
\newtheorem{proposition}[equation]{Proposition}%
\newtheorem{corollary}[equation]{Corollary}%
\newtheorem*{conjecture*}{Conjecture}%
\newtheorem{innercustomgeneric}{\customgenericname}
\providecommand{\customgenericname}{}
\newcommand{\newcustomtheorem}[2]{%
  \newenvironment{#1}[1]
  {%
   \renewcommand\customgenericname{#2}%
   \renewcommand\theinnercustomgeneric{##1}%
   \innercustomgeneric
  }
  {\endinnercustomgeneric}
}

\newcustomtheorem{customthm}{Theorem}
\newcustomtheorem{customcor}{Corollary}

\theoremstyle{definition} 
\newtheorem{definition}[equation]{Definition}%
\newtheorem*{definition*}{Definition}%
\newtheorem{example}[equation]{Example}%
\newtheorem{remark}[equation]{Remark}%
\newtheorem{notation}[equation]{Notation}%

\usepackage{mathtools}
\mathtoolsset{showonlyrefs}
\newcommand{\set}[1]{\left\{#1\right\}}%
\newcommand{\C}{\ensuremath{\mathcal{C}}}%
\renewcommand{\P}{\ensuremath{\mathcal{P}}}%
\renewcommand{\O}{\ensuremath{\mathcal{O}}}%

\newcommand{\sSet}{\ensuremath{\mathsf{sSet}}}%
\newcommand{\Cat}{\mathsf{Cat}}
\newcommand{\Op}{\mathsf{Op}}%
\newcommand{\sOp}{\ensuremath{\mathsf{sOp}}}%
\newcommand{\dSet}{\mathsf{dSet}}

\DeclareMathOperator{\colim}{colim}%

\usepackage{todonotes}
\usepackage{multirow}

\usepackage{tikz}%
\usetikzlibrary{matrix,arrows,decorations.pathmorphing,
cd,patterns,calc}
\tikzset{%
  treenode/.style = {shape=rectangle, rounded corners,%
                     draw, align=center,%
                     top color=white, bottom color=blue!20},%
  root/.style     = {treenode, font=\Large, bottom color=red!30},%
  env/.style      = {treenode, font=\ttfamily\normalsize},%
  dummy/.style    = {circle,draw,inner sep=0pt,minimum size=2mm}%
}%

\usetikzlibrary[decorations.pathreplacing]

\usepackage{ifdraft}
\ifdraft{
  \color[RGB]{63,63,63}
  
}
{
}

\author{Peter Bonventre, Lu\'is A. Pereira}%
\title{Equivariant dendroidal Segal spaces and $G$-$\infty$-operads}%
\date{\today}

\begin{document}	\maketitle%

\begin{abstract}
	We introduce the analogues of the notions of complete Segal space and of Segal category in the context of equivariant operads with norm maps, and build model categories with these as the fibrant objects. We then show that these model categories are Quillen equivalent to each other and to the model category for 
	$G$-$\infty$-operads built in a previous paper.

	Moreover, we establish variants of these results for the Blumberg-Hill indexing systems. 
      
	In an appendix, we discuss Reedy categories in the equivariant context.
\end{abstract}

\tableofcontents

%

\section{Introduction}

This paper follows \cite{Per18} and \cite{BP17} and is the third piece of a larger project aimed at understanding the homotopy theory of 
\textit{equivariant operads with norm maps}.
Here, norm maps are a new piece of structure that must be considered when dealing with equivariant operads
(see Remark \ref{NORMMAP REM} for a brief definition of norm maps or the introductions to 
\cite{Per18},\cite{BP17} for a more extensive discussion).
The need to understand norm maps
was made clear by Hill, Hopkins and Ravenel, who used  them in the context of equivariant ring spectra
as part of their solution of the Kervaire invariant one problem \cite{HHR16}.

The starting point of this project was the discovery by the authors of,
for each finite group $G$,
a category $\Omega_G$ of $G$-trees whose objects diagrammatically encode compositions of norm maps 
(in a $G$-equivariant operad)
and whose arrows encode the necessary compatibilities between such compositions.
Our categories $\Omega_G$ are a somewhat non-obvious  generalization of the dendroidal category $\Omega$
of Cisinski-Moerdijk-Weiss, 
and indeed all the key combinatorial concepts in their work,
such as faces, degeneracies, boundaries and horns, generalize to $G$-trees \cite[\S 5,\S 6]{Per18}.
As such, it is natural to attempt to generalize the 
Cisinski-Moerdijk program \cite{CM11},\cite{CM13a},\cite{CM13b} to the equivariant context. 

Recall that the main result of their program is the existence of a Quillen equivalence
\[
	W_{!} \colon \mathsf{dSet} 
		\rightleftarrows
	\mathsf{sOp}  \colon N_{hc} 
\]
where $\mathsf{dSet} = \mathsf{Set}^{\Omega^{op}} $
is the category of presheaves on $\Omega$, 
which are called \textit{dendroidal sets},
and 
$\mathsf{sOp}$ is the category of \textit{simplicial colored operads} (also referred to as simplicial multicategories).
Their program was carried out in three main steps:
\begin{inparaenum}
	\item[(i)] \cite{CM11} established the existence of the \textit{$\infty$-operad} model structure on $\mathsf{dSet}$ 
	(with some of the key combinatorial analysis based on Moerdijk and Weiss' previous work in \cite{MW09});
	\item[(ii)] \cite{CM13a} established auxiliary model structures on the categories $\mathsf{sdSet}$ and $\mathsf{PreOp}$
              of dendroidal spaces and pre-operads, with fibrant objects the \textit{complete dendroidal Segal spaces} and \textit{Segal operads},
              and showed that all three of $\mathsf{dSet}$, $\mathsf{sdSet}$ and $\mathsf{PreOp}$ are Quillen equivalent;
	\item[(iii)] lastly, \cite{CM13b} established the existence of the model structure on $\mathsf{sOp}$ as well as the Quillen equivalence between $\mathsf{sOp}$ and $\mathsf{PreOp}$, finishing the proof of the main result of the program\footnote{Recall that by using the inclusions of simplicial categories and simplicial sets into simplicial operads and dendroidal sets (cf. the introduction to \cite{CM13b}), the Cisinski-Moerdijk program recovers and generalizes the Bergner-Joyal-Lurie-Rezk-Tierney program studying the various models for $(\infty,1)$-categories. A survey of these models can be found at \cite{Ber10}.}.
\end{inparaenum}

From the perspective of the Cisinski-Moerdijk program, 
\cite{Per18} is then the equivariant analogue of the first step \cite{CM11} (as well as \cite{MW09}), 
while the present paper provides the equivariant analogue of the
second step \cite{CM13a}.
More explicitly, in \cite{Per18}, and inspired by the category $\Omega_G$ of $G$-trees,
the second author equipped the category
$\mathsf{dSet}^G$ of $G$-equivariant dendroidal sets with a model structure whose fibrant objects 
are ``equivariant operads with norm maps up to homotopy'',
called $G$-$\infty$-operads.
Further, it was shown therein \cite[Prop. 6.15]{Per18} that whenever a $G$-operad
$\O \in \mathsf{sOp}^G$ is
suitably fibrant the homotopy coherent nerve
$N_{hc}(\O)$ is such a 
$G$-$\infty$-operad (rather than just an ``$\infty$-operad with a $G$-action'').
In the present paper our main results,
Theorems \ref{INC0AGJ THM}, \ref{PREOPMOD THM}, \ref{ANOQUEQUIV THM},
are then the existence of suitable model structures on the categories 
$\mathsf{sdSet}^G$ and $\mathsf{PreOp}^G$
of $G$-dendroidal spaces and $G$-pre-operads,
with fibrant objects the \textit{complete equivariant dendroidal Segal spaces} and \textit{equivariant Segal operads},
as well as the existence of Quillen equivalences
between all three of 
$\mathsf{dSet}^G$, $\mathsf{sdSet}^G$ and $\mathsf{PreOp}^G$.

Table \ref{TABLE} provides a summary of the parallel narratives for categories, operads and equivariant operads,
 listing the different model categories as well as the terminology for the fibrant objects.
\begin{table}[htbp]
      \label{TABLE}
      \centering
      \resizebox{\columnwidth}{!}{%
        \begin{tabular}{|c|c|c|}
          \hline
          ``categories up to htpy'' & ``operads up to htpy'' & ``equivariant operads up to htpy''
          \\ \hline
          simplicial sets $\sSet$ & dendroidal sets $\dSet$ & equivariant dendroidal sets $\dSet^G$
          \\
          Joyal model structure & model str. from \cite{CM11} & model structure from \cite{Per18}
          \\
          $\infty$-categories & $\infty$-operads & $G$-$\infty$-operads
          \\ \hline
          bisimplicial sets $\mathsf{ssSet}$ & simp. dend. sets $\mathsf{sdSet}$ & equiv. simp. dend. sets $\mathsf{sdSet}^G$
          \\
          Rezk model structure & model str. from \cite{CM13a} & model structure from \S \ref{CEDSS SEC}
          \\
          complete Segal spaces & complete dend. Segal spaces & complete equiv. dend. Segal spaces
          \\ \hline
          Segal precategories $\mathsf{SeCat}$ & Segal preoperads $\mathsf{PreOp}$ & equiv. Segal preoperads $\mathsf{PreOp}^G$
          \\
          Hirschowitz-Simpson & model str. from \cite{CM13a} & model structure from \S \ref{PREOP SEC}
          \\
          Segal categories & Segal operads & equiv. Segal operads
          \\ \hline
          simplicial categories $\mathsf{sCat}$ & simplicial operads $\sOp$ & equiv. simplicial operads $\sOp^G$
          \\
          Bergner model structure & model str. from \cite{CM13b} & model structure forthcoming
          \\ \hline
        \end{tabular}
      }
      \caption{A summary of models for $\infty$-categories, $\infty$-operads, and $G$-$\infty$-operads.}
\end{table}

It is worth noting that, much as was the case of the work in \cite{Per18}, our results are not formal consequences of their non-equivariant analogues, due to the nature of norm maps\footnote{As a point of contrast, we note that the lack of norms in the categorical case causes the equivariant generalization of this latter program to indeed be formal; see \cite{Ste16,Ber17}.}.
Indeed, in \cite{BP17}, the second piece of our project,
the authors introduced the notion of 
\textit{genuine equivariant operads},
which are new algebraic objects motivated by the combinatorics of norm maps as encoded by the category $\Omega_G$ of $G$-trees.
And while a priori the work in \cite{BP17} is largely perpendicular to the Cisinski-Moerdijk program
(the main result \cite[Thm. III]{BP17} is what one might call the 
``operadic Elmendorf-Piacenza theorem'', which is an equivariant phenomenon),
some of the new technical hurdles in this paper versus \cite{CM13a} come from the need to work with 
(colored) genuine equivariant operads, 
which we repackage in \S \ref{GENEQOP SEC} via an independent (but equivalent) perspective to that of \cite{BP17}.

\vskip 10pt

The organization of the paper is as follows.

\S \ref{PREL SEC} mostly recalls the necessary notions concerning the category $\Omega_G$ of $G$-trees and the category $\mathsf{dSet}^G$ of $G$-dendroidal sets 
that were introduced in \cite{Per18}.
However, some new notions and results can be found throughout, most notably the notion of \textit{orbital face} of a $G$-tree in Definition \ref{ORBFACE DEF}
and the associated notion of \textit{orbital horn} in 
\S \ref{EQDENDSETS SEC}.

The main goal of \S \ref{EQINNERAN SEC} is to establish
Proposition \ref{HYPER PROP},
which roughly states that Segal core inclusions, horn inclusions and orbital horn inclusions can in some circumstances be used interchangeably.
The bulk of the work takes place in 
\S \ref{CHAREDGE SEC} where Lemma \ref{CHAREDGE LEM},
a powerful technical result we call the
\textit{characteristic edge lemma}, is established.
\S \ref{HYPERSAT SEC} then shows Proposition \ref{HYPER PROP} via a string of easy applications of 
Lemma \ref{CHAREDGE LEM}.
Lastly, \S \ref{GENEQOP SEC} recasts the genuine equivariant operads of \cite{BP17} in a different perspective more suitable for our purposes in \S \ref{EDSS_SEC}.

\S \ref{QUIEQ SEC} establishes the desired Quillen equivalences between 
$\mathsf{dSet}^G$, $\mathsf{sdSet}^G$, $\mathsf{PreOp}^G$
via largely abstract methods.
Our approach is inspired by 
\cite[Thm. 6.6]{CM13a}, which observes that the Rezk/complete model structure on $\mathsf{sdSet}$ can be built via two distinct localization procedures.
In fact, we will prefer to use the common localization perspective to define the equivariant Rezk/complete model structure on $\mathsf{sdSet}^G$ (cf. Definition \ref{JOINREED DEF} and Remark \ref{JOINREED REM}),
and then ``work backward'' (cf. Remark \ref{RECOVDEF REM}) to obtain the analogues of the definitions in \cite{CM13a} and of \cite[Thm. 6.6]{CM13a}.
%
As such, in \S \ref{JOINBOUS SEC} we first discuss an abstract setting for such common
localizations,
which is then applied in \S \ref{CEDSS SEC} to obtain
the Quillen equivalence
$\mathsf{dSet}^G \rightleftarrows \mathsf{sdSet}^G$
in Theorem \ref{INC0AGJ THM}.
\S \ref{PREOP SEC} then uses purely formal techniques to induce the model structure on
$\mathsf{PreOp}^G$ from the model structure on
$\mathsf{sdSet}^G$
and to establish the Quillen equivalence
$\mathsf{PreOp}^G \rightleftarrows \mathsf{sdSet}^G$
in Theorem \ref{ANOQUEQUIV THM}.

In our last main section \S \ref{EDSS_SEC},
motivated by the fact that in our desired model structure on simplicial $G$-operads $\mathsf{sOp}^G$
(to be described in a follow-up paper)
the weak equivalences are Dwyer-Kan equivalences
(i.e. characterized by fully faithfulness and essential surjectivity), 
we establish Theorem \ref{COMPIFFDK THM}, which gives a Dwyer-Kan type description of the weak equivalences 
between the fibrant objects in either of $\mathsf{sdSet}^G$, $\mathsf{PreOp}^G$.

\S \ref{INDEX SEC}, which is transversal to the rest of the paper, generalizes all our main results by replacing the category $\Omega_G$ of $G$-trees
with certain special subcategories
$\Omega_{\mathcal{F}} \subseteq \Omega_G$
which (almost exactly) correspond to the 
\textit{indexing systems} first identified
by Blumberg and Hill in \cite{BH15}.

Lastly, Appendix \ref{EQREED AP} discusses an equivariant variant of the \textit{generalized Reedy categories}
of \cite{BM11} which plays an essential role in \S \ref{CEDSS SEC}
when describing the model structure on $\mathsf{sdSet}^G$.
The key to this appendix is the \textit{Reedy-admissibility} condition in 
Definition \ref{GENRED DEF}(iv),
which is a fairly non-obvious equivariant generalization of one of the generalized Reedy axioms in \cite{BM11}.

\subsection{Acknowledgments}

The authors would like to thank the anonymous referee for their many helpful suggestions and comments, 
including the use of Table \ref{TABLE}.

\section{Preliminaries}\label{PREL SEC}

\subsection{The category of trees $\Omega$}

We start by recalling the key features of the category $\Omega$ of trees that will be used throughout.
Our official model for $\Omega$ will be Weiss' algebraic model of \textit{broad posets} as discussed in \cite[\S 5]{Per18},
hence we first recall some key notation and terminology.
Given a tree diagram $T$ such as
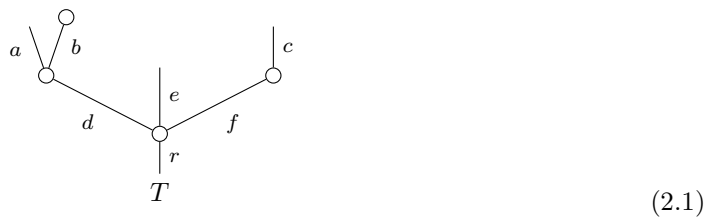
\begin{equation}\label{FIRSTTREE EQ}
	\begin{tikzpicture}[auto,grow=up,
	level distance = 2.2em,
	every node/.style={font=\footnotesize,minimum size=1.5mm}]%
	\tikzstyle{level 2}=[sibling distance=4.25em]%
	\tikzstyle{level 3}=[sibling distance=1.5em]%
		\node at (0,0)[font=\normalsize]{$T$}%
			child{node [dummy] {}%
				child{node [dummy] {}%
					child{node {}%
					edge from parent node [swap] {$c$}}%
				edge from parent node [swap] {$f$}}%
				child[level distance = 2.5em]{
				edge from parent node [swap] {$e$}}%
				child{node [dummy] {}%
					child{node[dummy] {}%
					edge from parent node [very near end,swap] {$b$}}%
					child{node {}%
					edge from parent node [very near end] {$\phantom{b}a$}}%
				edge from parent node {$d$}}%
			edge from parent node [swap] {$r$}};%
	\end{tikzpicture}%
\end{equation}
and for each edge $t$ of $T$ topped by a vertex $\circ$, 
we write $t^{\uparrow}$ to denote the tuple of edges immediately above $t$.
In our example, 
$r^{\uparrow}=def$, 
$d^{\uparrow} = ab$,
$f^{\uparrow} = c$ and
$b^{\uparrow} = \epsilon$, 
where $\epsilon$ is the empty tuple.
Edges $t$ for which:
\begin{inparaenum}
\item[(i)] $t^{\uparrow} \neq \epsilon$, such as $r,d,f$, are called \textit{nodes};
\item[(ii)] $t^{\uparrow} = \epsilon$, such as $b$, are called \textit{stumps};
\item[(iii)] $t^{\uparrow}$ is undefined, such as $a,c,e$, are called \textit{leaves}.
\end{inparaenum}
Each vertex of $T$ is then encoded symbolically as 
$t^{\uparrow} \leq t$, which we call a \textit{generating broad relation}. 
This notation is meant to suggest a form of transitivity: for example, the generating relations
$ab \leq d$ and $def \leq r$
generate, via \textit{broad transitivity},
a relation $abef \leq r$
(we note that this is essentially compact notation for the operations and composition in the colored operad $\Omega(T)$ generated by $T$
\cite[\S 3]{MW07},\cite[Rem. 4.4, Ex. 4.6]{Per18}). The other broad relations obtained by broad transitivity are 
$aef \leq r$,
$dec \leq r$,
$abec \leq r$,
$aec \leq r$,
$a \leq d$.
The set of edges of $T$ together with these broad relations
(as well as identity relations $t \leq t$) form the 
\textit{broad poset} associated to the tree, which is again denoted $T$.

Given a broad relation $t_0 \cdots t_n \leq t$,
we further write $t_i \leq_d t$.
Pictorially, this says that the edge $t_i$ is above $t$,
and it is thus clear that $\leq_d$ defines a partial order on edges of $T$.
Trees always have a single $\leq_d$-maximal edge, called the \textit{root}. Edges other than the root or the leaves are called \textit{inner edges}. In our example $r$ is the root, $a,e,c$ are leaves, and $b,d,f$ are inner edges. 

We denote the sets of edges, inner edges, vertices of $T$ by 
$\boldsymbol{E}(T)$, $\boldsymbol{E}^{\mathsf{i}}(T)$,
$\boldsymbol{V}(T)$.

The Cisinski-Moerdijk-Weiss category $\Omega$ of trees then has as objects the tree diagrams as in \eqref{FIRSTTREE EQ}
and as maps $\varphi \colon T \to S$ the monotone maps of broad posets
(meaning that if $t_1 \cdots t_k \leq t$ then
$\varphi(t_1) \cdots \varphi(t_k) \leq \varphi(t)$).
In fact, in \cite{Wei12} Weiss characterized 
those broad posets associated to trees (see \cite[Defs. 5.1 and 5.9]{Per18}),
so that one is free to work intrinsically with broad posets.

Moreover, our discussion will be somewhat simplified by the assumption that $\Omega$
contains exactly one representative of each \textit{planarized} tree.
Informally, this means that trees $T \in \Omega$
come with a preferred planar representation,
though this can also be formalized in purely algebraic terms, see \cite[\S 3.1]{BP17}.
For our purposes, the main consequence is that any map 
$S \to T$ in $\Omega$ has a (strictly) unique factorization
$S \xrightarrow{\simeq} S' \to T$ as an isomorphism followed by a \textit{planar map} \cite[Prop. 3.23]{BP17}. 
Informally, $S'$ is obtained from $S$
by ``pulling back'' the planarization of $T$.

We now recall the key classes of maps of $\Omega$.
A map $\varphi \colon S \to T$ which is injective on edges is called a \textit{face map}
while a map that is surjective on edges and preserves leaves is called a \textit{degeneracy map}
(the extra requirement ensures that leaves of $S$ do not become stumps of $T$).
Moreover, a face map $\varphi$ is further called an \textit{inner face map}
if $\varphi(r_S) = r_T$ and 
$\varphi(\underline{l}_S) = \underline{l}_T$ 
(where $r_{(-)}$ denotes the root edge and $\underline{l}_{(-)}$ the leaf tuple)
and called an \textit{outer face map} if it does not admit a factorization as a non-isomorphism inner face map followed by a face map.

The following result is \cite[Cor. 3.32]{BP17}, with the additional planar statement following from the unique factorization of maps in $\Omega$ as isomorphisms followed by planar maps. 

\begin{proposition}\label{UNIQUEFACT PROP}
	A map $\varphi \colon S \to T$ in $\Omega$ has a factorization, unique up to unique isomorphisms,
\[
	S \xrightarrow{\varphi^{-}}
	U \xrightarrow{\varphi^{i}}
	V \xrightarrow{\varphi^{o}}
	T	
\]
as a degeneracy followed by an inner face map followed by an outer face map.

Moreover, there is a (strictly) unique factorization with $\varphi^i,\varphi^o$ planar. 
\end{proposition}

We next recall an explicit characterization and notation for planar inner/outer faces
(planar degeneracies are characterized by edge multiplicities, see \cite[Prop. 3.47(ii)]{BP17}).
For any subset $E \subseteq \boldsymbol{E}^{\mathsf{i}}(T)$, there is a planar inner face
$T-E$ which removes the inner edges in $E$ but keeps all broad relations involving edges not in $E$
(this is the hardest class of maps to visualize pictorially, as the vertices adjacent to each $e \in E$ are combined via broad transitivity/composition).
For each broad relation
$t_1 \cdots t_k = \underline{t} \leq t$ in $T$,
there is a planar outer face
$T_{\underline{t} \leq t}$
such that
$r_{T_{\underline{t} \leq t}} = t$ and
$\underline{l}_{T_{\underline{t} \leq t}} = \underline{t}$
(in fact, by Proposition \ref{UNIQUEFACT PROP} this is the largest such face).
Moreover, the edges $s$ of $T_{\underline{t} \leq t}$ are the edges of $T$ such that
$s \leq_d t$ and $\forall_{i} s \not <_d t_i$ while the vertices are the $s^{\uparrow} \leq s$ such that 
$s \leq_d t$ and $\forall_{i} s \not \leq_d t_i$ 
(pictorially, $T_{\underline{t} \leq t}$ removes those sections of $T$ not above $t$ and above some $t_i$).

\begin{remark}\label{INNFULL REM}
	Inner faces $T-E \hookrightarrow T$ are always full, i.e. $T-E$ contains all broad relations of $T$ between those edges in 
$\boldsymbol{E}(T-E) = \boldsymbol{E}(T) \setminus E$.
	By contrast, whenever $T$ has stumps some of its outer faces $T_{\underline{t} \leq t}$ are not full,
	the main example being given by the maximal outer faces
	that ``remove stumps'' \cite[Not. 5.41]{Per18}.
\end{remark}

\begin{remark}\label{DEGREE REM}
	Following \cite[Ex. 2.8]{BM11}, one has a degree function 
	$|-|\colon \Omega \to \mathbb{N}$ given by $|T|=|\boldsymbol{V}(T)|$
	such that non isomorphism face maps (resp. degeneracies) strictly increase (decrease) $|-|$.
	As such, the subcategory of face maps is denoted $\Omega^+$ while that of degeneracies is denoted $\Omega^-$.
\end{remark}

We now collect a couple of useful lemmas concerning faces.

\begin{lemma}\label{INNINT LEM}
	Consider a diagram of planar faces in $\Omega$
	(implicitly regarded as inclusion maps)
\[
\begin{tikzcd}[column sep =3em]
	V \ar[hookrightarrow]{r}{\text{out}} 
	\ar[hookrightarrow]{d}[swap]{\text{inn}} &
	U \ar[hookrightarrow]{d}
\\
	\bar{V} \ar[hookrightarrow]{r}{\text{out}}&
	\bar{U}
\end{tikzcd}
\]
	such that the horizontal maps are outer face maps and the left vertical map is an inner face map.

Then $\boldsymbol{E}^{\mathsf{i}}(V) = 
\boldsymbol{E}^{\mathsf{i}}(U) \cap 
\boldsymbol{E}^{\mathsf{i}} (\bar{V})$.
\end{lemma}

\begin{proof}
	Write $r$ and $\underline{l}=l_1\cdots l_n$
	for the root and leaf tuple of $V$ or, equivalently, of $\bar{V}$.
	Since the horizontal maps are outer, an edge
	$e \in \boldsymbol{E}^{\mathsf{i}}(U)$ 
	(resp. $e \in \boldsymbol{E}^{\mathsf{i}}(\bar{U})$)
	is also in $\boldsymbol{E}^{\mathsf{i}}(V)$ (resp. in $E^{\mathsf{i}}(\bar{V})$) iff
	$e <_d r$ and $\forall_i e \not \leq_d l_i$.
	But then 
	$\boldsymbol{E}^{\mathsf{i}}(V) =
	\boldsymbol{E}^{\mathsf{i}}(U) \cap 
	\boldsymbol{E}^{\mathsf{i}}(V) =
	\boldsymbol{E}^{\mathsf{i}}(U) \cap
	\boldsymbol{E}^{\mathsf{i}}(\bar{V})$. 
\end{proof}

\begin{lemma}\label{CUPCAP LEM}
	Let $\{U_i \hookrightarrow T\}$ be a collection of planar outer faces of $T$ with a common root $t$. Then there are planar outer faces
	$U^{\cup} \hookrightarrow T$, $U^{\cap} \hookrightarrow T$,
	also with root $t$, such that
\begin{equation}\label{CUPCAP EQ}
	\boldsymbol{E}(U^{\cup}) = \bigcup_i \boldsymbol{E}(U_i), \quad
	\boldsymbol{V}(U^{\cup}) = \bigcup_i \boldsymbol{V}(U_i), \qquad
	\boldsymbol{E}(U^{\cap}) = \bigcap_i \boldsymbol{E}(U_i), \quad
	\boldsymbol{V}(U^{\cap}) = \bigcap_i \boldsymbol{V}(U_i).
\end{equation}
Moreover, these are the smallest (resp. largest) outer faces
containing (contained in) all $U_i$.
\end{lemma}

\begin{remark}
	More generally, $U^{\cup}$ and $U^{\cap}$
	can be defined whenever the $U_i$ have a common edge.
\end{remark}

\begin{proof}
	Since edges and vertices are simply elements and generating broad relations of the broad poset of a tree,  
	\eqref{CUPCAP EQ} generates pre-broad posets 
	(cf. \cite[Rem. 5.2]{Per18}) $U^{\cup}$ and $U^{\cap}$.
	It now suffices to check that 
	$U^{\cup}$ and $U^{\cap}$ are trees,
	i.e. that they satisfy the axioms in 
	\cite[Defs. 5.1, 5.3, 5.9]{Per18}.
	Antisymmetry and simplicity are inherited from $T$, the root axiom follows since the $U_i$ have a common root (in the $U^{\cap}$ case note that if $s$ is in $U^{\cap}$, then so is any $s'$ such that
	$s \leq_d s' \leq_d t$),
	and the nodal axiom is obvious from \eqref{CUPCAP EQ} (which, a posteriori, is correct as an identity on sets of edge and vertex sets).
\end{proof}

\begin{notation}
      \label{DELTAOMEGA NOT}
We will write $\iota \colon \Delta \to \Omega$ for the standard inclusion, which sends $[n]$ to the linear tree with $n+1$ edges (recall that all notions in this section generalize the usual notions for $\Delta$).
Additionally, we will as usual write $\eta =  \iota([0])$ for the ``stick tree'' containing a single edge and no vertices, and note that $\iota$ induces an identification between $\Delta$ and the overcategory $\Omega \downarrow \eta$.
\end{notation}

\subsection{The category of $G$-trees $\Omega_G$}

We next recall the category $\Omega_G$ of $G$-trees introduced in \cite[\S 5.3]{Per18}. We start with an explicit and representative example of a $G$-tree (for more examples, see \cite[\S 4.3]{Per18}).
Letting $G = \{ \pm 1, \pm i, \pm j, \pm k\}$ denote the group of quaternionic units 
and $G \geq H \geq K \geq L$ denote the subgroups %
$H = \langle j \rangle$, %
$K = \langle -1 \rangle$, %
$L = \{1\}$,
there is a $G$-tree $T$ with 
\textit{expanded representation}
given by the two trees on the left below and
\textit{orbital representation}
given by the (single) tree on the right.
\begin{equation}\label{TWOREP EQ}
	\begin{tikzpicture}[auto,grow=up, level distance = 2.2em,
	every node/.style={font=\scriptsize,inner sep = 2pt}]%
		\tikzstyle{level 2}=[sibling distance=7em]%
		\tikzstyle{level 3}=[sibling distance=2.25em]%
			\node at (4.75,0){}%
				child{node [dummy] {}%
					child{node [dummy] {}%
						child{node {}%
						edge from parent node [swap,very near end] {$-k a$}}%
						child[level distance = 2.4em]{node {}%
						edge from parent node [swap,near end] {$k b$}}%
						child{node {}%
						edge from parent node [very near end] {$k a$}}%
					edge from parent node [swap] {$k c$}}%
					child{node [dummy] {}%
						child{node {}%
						edge from parent node [swap,very near end] {$-i a$}}%
						child[level distance = 2.4em]{node {}%
						edge from parent node [swap,near end] {$i b$}}%
						child{node {}%
						edge from parent node [very near end] {$i a$}}%
					edge from parent node  {$i c$}}%
				edge from parent node [swap] {$i d$}};%
			\node at (0,0){}%
				child{node [dummy] {}%
					child{node [dummy] {}%
						child{node {}%
						edge from parent node [swap,very near end] {$-j a$}}%
						child[level distance = 2.4em]{node {}%
						edge from parent node [swap,near end] {$j b$}}%
						child{node {}%
						edge from parent node [very near end] {$j a$}}%
					edge from parent node [swap] {$j c$}}%
					child{node [dummy] {}%
						child{node {}%
						edge from parent node [swap,very near end] {$-a\phantom{j}$}}%
						child[level distance = 2.4em]{node {}%
						edge from parent node [swap,near end] {$b\phantom{j}$}}%
						child{node {}%
						edge from parent node [very near end] {$\phantom{-j}a$}}%
					edge from parent node  {$\phantom{j}c$}}%
				edge from parent node [swap] {$d$}};%
		\begin{scope}[every node/.style={font=\footnotesize}]%
			\node at (9.35,0){}%
				child{node [dummy] {}%
					child{node [dummy] {}%
						child{node {}%
						edge from parent node [swap,very near end] {$(G/K) \cdot b$}}%
						child{node {}%
						edge from parent node [very near end] {$(G/L) \cdot a$}}%
					edge from parent node [right] {$(G/K) \cdot c$}}%
				edge from parent node [right] {$(G/H) \cdot d$}};%
		\end{scope}%
		\draw[decorate,decoration={brace,amplitude=2.5pt}] (4.85,0) -- (-0.1,0) node[midway,inner sep=4pt,font=\normalsize]{$T$}; %
		\node at (9.35,-0.15) [font=\normalsize] {$T$};
	\end{tikzpicture}%
\end{equation}%
Note that the edge labels on the expanded representation encode the action of $G$ so that the edges 
$a,b,c,d$ have isotropy $L,K,K,H$.

Formally, the definition of $\Omega_G$ \cite[Def. 5.44]{Per18} is given as follows.
Given a non-equivariant forest diagram $F$ 
(i.e. a finite collection of tree diagrams side by side),
there is
an associated broad poset just as before, and one thus obtains a category $\Phi$ of forests and broad monotone maps.
Letting $\Phi^G$ denote the category of $G$-objects in $\Phi$, referred to as $G$-forests,
the category $\Omega_G \subset \Phi^G$ of $G$-trees
is defined as the full subcategory of those $G$-forests such that the $G$-action is transitive on tree components.

We note that any $G$-tree $T$ can be written as
an induction $T \simeq G \cdot_H T_{\**}$, where $T_{\**}$ is some fixed tree component, $H\leq G$ is the subgroup sending $T_{\**}$ to itself,
and we regard $T_{\**} \in \Omega^H$, i.e., as a tree with a $H$-action (where we caution that $\Omega^G \subsetneq \Omega_G$).
For example, in \eqref{TWOREP EQ} it is 
$T \simeq G \cdot_H T_d$ for $H\leq G$, $T \in \Omega_G$ as defined therein and
$T_d \in \Omega^H$ the tree component containing $d$. 

Moreover, we similarly assume that $G$-trees (and forests in general) are planarized, meaning that they come with a total order of the tree components, each of which is  planarized.

If $T\in \Omega_G$ has tree components $T_1,\cdots, T_k$, we write
$\boldsymbol{E}(T) = \amalg_i \boldsymbol{E}(T_i)$, 
$\boldsymbol{E}^{\mathsf{i}}(T) = 
\amalg_i \boldsymbol{E}^{\mathsf{i}}(T_i)$,
$\boldsymbol{V}(T) = \amalg_i \boldsymbol{V}(T_i)$
for its sets of edges, inner edges and vertices, as well as
$\boldsymbol{E}_G(T) = \boldsymbol{E}(T)/G$,
$\boldsymbol{E}^{\mathsf{i}}_G(T) = \boldsymbol{E}^{\mathsf{i}}(T)/G$,
$\boldsymbol{V}_G(T) = \boldsymbol{V}(T)/G$ for its sets of 
\textit{edge orbits},
\textit{inner edge orbits} and
\textit{$G$-vertices}.

Before discussing face maps in the equivariant context, it is worth commenting on the complementary roles of the expanded and orbital representations.
On the one hand, the $G$-broad posets associated to $G$-trees are diagrammatically represented by the expanded representation,
so that the arrows of $\Omega_{G}$ are best understood from that perspective.
On the other hand, the diagrams encoding compositions of norm maps of an equivariant operad $\mathcal{O}$
are given by the orbital representations of  
$G$-trees (see Example \ref{STRICTLIFT EX} and Remark \ref{NORMMAP REM}, or alternatively 
\cite[Ex. 4.9]{Per18}, \cite[(1.10)]{BP17}).
As a result, different aspects of our discussion are guided by different representations, and this will require us to discuss the different notions of face/boundary/horn suggested by the two representations.
We start by recalling the notion of face discussed in \cite{Per18}, which is motivated by the expanded representation.

\begin{definition}
	Let $T \in \Omega_G$ be a $G$-tree with non-equivariant tree components 
	$T_1, T_2,\cdots,T_k$.
	
	A \textit{face} of $T$ 
	is an underlying face map
	$U \hookrightarrow T_i$ in $\Omega$ for some $1 \leq i \leq k$.
	Further, we abbreviate faces of $T$ as
	$U \hookrightarrow T$,
	and call them \textit{planar/outer} faces
	whenever so is the map $U \hookrightarrow T_i$.
\end{definition}

\begin{notation}
	Given $T \in \Omega_G$, we write $\mathsf{Face}(T)$ for the
	$G$-poset of \textit{planar faces} $U \hookrightarrow T$.
	We note that the $G$-action is given by the unique factorization of the composite
	$U \hookrightarrow T \xrightarrow{g} T$
	as $U \simeq gU \hookrightarrow T$ such that 
	$gU \hookrightarrow T$ is planar.
\begin{equation}\label{FACEGACT EQ}
\begin{tikzcd}
	U \ar[hookrightarrow]{r} \ar{d}[swap]{\simeq} &
	T \ar{d}{g}
\\
	gU \ar[hookrightarrow]{r} & T
\end{tikzcd}
\end{equation}
Alternatively, planar faces $U \hookrightarrow T$ can be viewed as sub-broad posets $U \subseteq T$,
identifying this $G$-action with the natural action on subsets. However, we prefer the planar face framework since it is more readily related to the presheaves $\Omega[T]$ discussed in the next section
(see Remark \ref{FACEGACT REM}).
\end{notation}

\begin{notation}\label{BARUT NOT}
	Given $T\in \Omega_G$ and a face
	$\varphi \colon U \hookrightarrow T$ we write $\bar{U}^T$, or just $\bar{U}$ when no confusion should arise, for the planar face in the unique factorization
	$U \overset{\varphi^i}{\hookrightarrow} \bar{U}^T \overset{\varphi^o}{\hookrightarrow} T$
	(cf. Proposition \ref{UNIQUEFACT PROP}) with $\varphi^o$ planar.
	We will call $\bar{U}^T$ the \textit{outer closure of $U$}, due to it being the smallest planar outer face of $T$ such that $U \hookrightarrow \bar{U}^T$ (to see this, let $V \hookrightarrow T$ be a planar outer face such that $U \hookrightarrow V$; then Proposition \ref{UNIQUEFACT PROP} applied to 
	$U \hookrightarrow \bar{U}^V \hookrightarrow V \hookrightarrow T$ implies it must be $\bar{U}^V = \bar{U}^T$, showing 
	$\bar{U}^T \hookrightarrow V$).
\end{notation}

\begin{remark}\label{PLFUNCTOR REM}
	Recalling the notation $\Omega^+ \subset \Omega$
	for the subcategory of face maps,
	we write $\Omega^+ \downarrow T$ for the category of all faces of $T \in \Omega_G$.
	By pulling back the planarization of $T$ one then obtains a \textit{planarization functor}
	\[
		\Omega^+ \downarrow T \xrightarrow{pl} \mathsf{Face}(T)
	\]
which respects the $G$-actions on the two categories.
	Note, however, that the inclusion 
	$\mathsf{Face}(T) \subset \Omega^+ \downarrow T$ (which is a section of $pl$) does not respect the $G$-actions, as displayed in (\ref{FACEGACT EQ}).
\end{remark}

We now introduce the notion of face of a $G$-tree that is suggested by the orbital representation.

\begin{definition}\label{ORBFACE DEF}
	Let $T \in \Omega_G$ be a $G$-tree.
	An \textit{orbital face} of $T$ is a map 
	$S \hookrightarrow T$ in $\Omega_G$ which is injective on edges. Further, an orbital face is called
	\textit{inner/outer} if any (and thus all) of its  component maps is and \textit{planar} if it is a planar map of forests \cite[Def. 3.21]{BP17}.
\end{definition}

\begin{example}\label{ORBFACE EX}
The following are three planar orbital faces of the $G$-tree $T$ in \eqref{TWOREP EQ},
with $R_1 \hookrightarrow T$,  
$R_2 \hookrightarrow T$ orbital outer faces and 
$S \hookrightarrow T$ an orbital inner face.
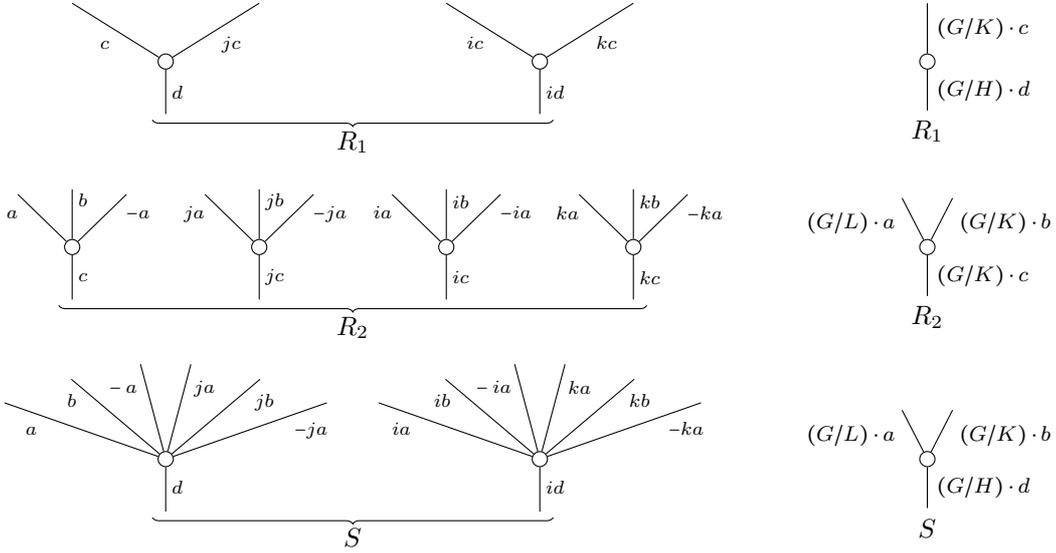
\begin{equation}
	\begin{tikzpicture}[auto,grow=up, level distance = 2.2em,
	every node/.style={font=\scriptsize,inner sep = 2pt}]%
	\begin{scope}
		\tikzstyle{level 2}=[sibling distance=7em]%
			\node at (17.5em,0){}%
				child{node [dummy] {}%
					child{
					edge from parent node [swap] {$k c$}}%
					child{
					edge from parent node  {$i c$}}%
				edge from parent node [swap] {$i d$}};%
			\node at (3.5em,0){}%
				child{node [dummy] {}%
					child{
					edge from parent node [swap] {$j c$}}%
					child{
					edge from parent node  {$\phantom{j}c$}}%
				edge from parent node [swap] {$d$}};%
		\begin{scope}[every node/.style={font=\footnotesize}]%
			\node at (32em,0){}%
				child{node [dummy] {}%
					child{
					edge from parent node [right] {$(G/K) \cdot c$}}%
				edge from parent node [right] {$(G/H) \cdot d$}};%
		\end{scope}%
		\draw[decorate,decoration={brace,amplitude=2.5pt}] (18em,0) -- (3em,0) node[midway,inner sep=4pt,font=\normalsize]{$R_1$}; %
		\node at (32em,-0.15) [font=\normalsize] {$R_1$};
	\end{scope}
	\begin{scope}[yshift = -7em]
		\begin{scope}
		\tikzstyle{level 2}=[sibling distance=2.25em]%
			\node at (0,0){}%
				child{node [dummy] {}%
					child{node {}%
					edge from parent node [swap,very near end] {$-a\phantom{j}$}}%
					child[level distance = 2.4em]{node {}%
					edge from parent node [swap,near end] {$b\phantom{j}$}}%
					child{node {}%
					edge from parent node [very near end] {$\phantom{-j}a$}}%
				edge from parent node [swap] {$c\phantom{j}$}};%
			\node at (7em,0){}%
				child{node [dummy] {}%
					child{node {}%
					edge from parent node [swap,very near end] {$-j a$}}%
					child[level distance = 2.4em]{node {}%
					edge from parent node [swap,near end] {$j b$}}%
					child{node {}%
					edge from parent node [very near end] {$j a$}}%
				edge from parent node [swap] {$j c$}};%
			\node at (14em,0){}%
				child{node [dummy] {}%
					child{node {}%
					edge from parent node [swap,very near end] {$-i a$}}%
					child[level distance = 2.4em]{node {}%
					edge from parent node [swap,near end] {$i b$}}%
					child{node {}%
					edge from parent node [very near end] {$i a$}}%
				edge from parent node [swap] {$i c$}};%
			\node at (21em,0){}%
				child{node [dummy] {}%
					child{node {}%
					edge from parent node [swap,very near end] {$-k a$}}%
					child[level distance = 2.4em]{node {}%
					edge from parent node [swap,near end] {$k b$}}%
					child{node {}%
					edge from parent node [very near end] {$k a$}}%
				edge from parent node [swap] {$k c$}};%
		\end{scope}
		\begin{scope}[every node/.style={font=\footnotesize}]%
		\tikzstyle{level 2}=[sibling distance=2.25em]
			\node at (32em,0){}%
				child{node [dummy] {}%
					child{node {}%
					edge from parent node [swap,very near end] {$(G/K) \cdot b$}}%
					child{node {}%
					edge from parent node [very near end] {$(G/L) \cdot a$}}%
				edge from parent node [right] {$(G/K) \cdot c$}};%
		\end{scope}%
		\draw[decorate,decoration={brace,amplitude=2.5pt}] (21.5em,0) -- (-0.5em,0) node[midway,inner sep=4pt,font=\normalsize]{$R_2$}; %
		\node at (32em,-0.15) [font=\normalsize] {$R_2$};
	\end{scope}
	\begin{scope}[yshift = -15em]
		\begin{scope}
		\tikzstyle{level 2}=[sibling distance=3em]%
			\node at (3.5em,0){}%
				child{node [dummy] {}%
					child[sibling distance=2.5em,level distance=2.2em]{node {}%
					edge from parent node [swap, near end] {$-ja\phantom{j}$}}%
					child[sibling distance=2.5em,level distance=3.2em]{node {}%
					edge from parent node [very near end,swap] {$j b\phantom{j}$}}%
					child[sibling distance=2em,level distance=3.8em]{node {}%
					edge from parent node [swap,very near end] {$j a$}}%
					child[sibling distance=2em,level distance=3.8em]{node {}%
					edge from parent node [very near end] {$\phantom{j}-a$}}%
					child[sibling distance=2.5em,level distance=3.2em]{node {}%
					edge from parent node [very near end] {$\phantom{j}b$}}%
					child[sibling distance=2.5em,level distance=2.2em]{node {}%
					edge from parent node [near end] {$\phantom{-j}a$}}%
				edge from parent node [swap] {$d\phantom{j}$}};%
			\node at (17.5em,0){}%
				child{node [dummy] {}%
					child[sibling distance=2.5em,level distance=2.2em]{node {}%
					edge from parent node [swap, near end] {$-ka\phantom{j}$}}%
					child[sibling distance=2.5em,level distance=3.2em]{node {}%
					edge from parent node [very near end,swap] {$k b\phantom{j}$}}%
					child[sibling distance=2em,level distance=3.8em]{node {}%
					edge from parent node [swap,very near end] {$k a$}}%
					child[sibling distance=2em,level distance=3.8em]{node {}%
					edge from parent node [very near end] {$\phantom{j}-ia$}}%
					child[sibling distance=2.5em,level distance=3.2em]{node {}%
					edge from parent node [very near end] {$\phantom{j}i b$}}%
					child[sibling distance=2.5em,level distance=2.2em]{node {}%
					edge from parent node [near end] {$\phantom{-j}i a$}}%
				edge from parent node [swap] {$i d\phantom{j}$}};%
		\end{scope}
		\begin{scope}[every node/.style={font=\footnotesize}]%
		\tikzstyle{level 2}=[sibling distance=2.25em]%
			\node at (32em,0){}%
				child{node [dummy] {}%
					child{node {}%
					edge from parent node [swap,very near end] {$(G/K) \cdot b$}}%
					child{node {}%
					edge from parent node [very near end] {$(G/L) \cdot a$}}%
				edge from parent node [right] {$(G/H) \cdot d$}};%
		\end{scope}%
		\draw[decorate,decoration={brace,amplitude=2.5pt}] (18em,0) -- (3em,0) node[midway,inner sep=4pt,font=\normalsize]{$S$}; %
		\node at (32em,-0.15) [font=\normalsize] {$S$};
	\end{scope}
	\end{tikzpicture}%
\end{equation}%
These examples illustrate our motivation for the term 
``orbital face'': the tree diagrams in the orbital representations of $R_1,R_2,S$ look like faces of the tree in the orbital representation of $T$.

Adapting the notation for (non-equivariant) inner faces, we write
$S = T-Gc = T-\{c,jc,ic,kc\}$ and analogously throughout the paper.
We will need no analogous notation for orbital outer faces.
\end{example}

\begin{notation}\label{TREEDIFNOT NOT}
	In the remainder of the paper we sometimes need to consider (non-equivariant) faces and orbital faces simultaneously.
	As such, we reserve the letters $U,V,W$ for trees in $\Omega$
	and the letters $R,S,T$ for $G$-trees in $\Omega_G$.
\end{notation}

One has the following orbital face analogue of Proposition \ref{UNIQUEFACT PROP}.

\begin{proposition}\label{INNOUTORB PROP}
	Any orbital face $\varphi \colon S \hookrightarrow T$ in $\Omega_G$ has a factorization 
	$S \overset{\varphi^i}{\hookrightarrow} R \overset{\varphi^o}{\hookrightarrow} T$, unique up to unique isomorphism, as an orbital inner face followed by an orbital outer face.	
	
	Moreover, there is a (strictly) unique factorization with $\varphi^o$ planar. 
\end{proposition}

\begin{proof}
It suffices to prove the planar statement.
Writing $T = \amalg_{i \in I} T_i$ and $S = \amalg_{j\in J} S_j$ for the tree component decompositions,
$\varphi \colon S \hookrightarrow T$ is described by a map 
$\varphi \colon J \to I$ together with planar face maps
$S_j \hookrightarrow T_{\varphi(j)}$.
Abbreviating $\bar{S}_j = \bar{S}_j^{T_{\varphi(j)}}$ (cf. Notation \ref{BARUT NOT}), the uniqueness in Proposition \ref{UNIQUEFACT PROP}
shows that the maps 
$S_j \hookrightarrow \bar{S}_j \hookrightarrow T_{\varphi(j)}$ must be the components of the desired factorization
$S \hookrightarrow R \hookrightarrow T$, so that it remains only to show that 
$R = \amalg_j \bar{S}_j$
admits a unique compatible $G$-action and that the natural map
$R \to T$ is injective.
To obtain the $G$-action on $R$, we again apply 
Proposition \ref{UNIQUEFACT PROP}
to obtain unique dashed arrows as in the following diagram.
\[
\begin{tikzcd}
	S_j \ar[hookrightarrow]{r} \ar{d}{g} & 
	\bar{S}_j \ar[hookrightarrow]{r} \ar[dashed]{d}{g} &
	T_{\varphi(j)} \ar{d}{g}
\\
	S_{gj} \ar[hookrightarrow]{r} & 
	\bar{S}_{gj} \ar[hookrightarrow]{r} & T_{g\varphi(j)}
\end{tikzcd}
\]
That the $G$-action is associative and unital, meaning that the composite
$\bar{S}_j \xrightarrow{g} \bar{S}_{gj} \xrightarrow{\bar{g}} \bar{S}_{\bar{g}gj}$ equals 
$\bar{S}_j \xrightarrow{\bar{g}g} \bar{S}_{\bar{g}gj}$
and that 
$\bar{S}_j \xrightarrow{e} \bar{S}_{ej}$
is the identity $\bar{S}_j = \bar{S}_{ej}$, 
follows from the analogous properties for $S,T$ and the uniqueness in Proposition \ref{UNIQUEFACT PROP}. 
Lastly, the remaining claim that $R \to T$ is injective follows since each map $\bar{S}_j \hookrightarrow T$ is injective together with the fact that the roots $r_{\bar{S}_j} = r_{S_j}$
are $\leq_d$-incomparable, so that edges of $T$ in different $\bar{S}_j$ are also $\leq_d$-incomparable \cite[Cor. 5.25]{Per18}.
\end{proof}

The argument at the end of the previous proof has the following two consequences.

\begin{remark}\label{ROOTISOMONO REM}
	If $U \in \mathsf{Face}(T)$ has isotropy $H$,
	the induced map 
	$G \cdot_H U \to T$ is injective on edges iff
	$H$ is also the isotropy of the root $r_U$.
\end{remark}

\begin{remark}\label{INNFULLORB REM}
	Orbital inner faces $S \hookrightarrow T$ are full (cf. Remark \ref{INNFULL REM}), i.e. all broad relations in $T$ between edges of $S$ are also in $S$. 
\end{remark}

We next discuss the interactions between (non-equivariant) faces and orbital faces.

\begin{proposition}\label{MINGFACT PROP}
	Let $T \in \Omega_G$.
	For any (non-equivariant) face $U \hookrightarrow T$ 
	there is a smallest planar orbital face $GU \hookrightarrow T$
	such that $U \hookrightarrow GU$. 
\end{proposition}

\begin{proof} Without loss of generality, we assume $U$ is planar throughout.

Suppose first that $U=\bar{U}^T$ is outer and write $H\leq G$ for
the isotropy of its root $r_U$.
By Lemma \ref{CUPCAP LEM} there exists a smallest planar outer face 
containing all $h U \hookrightarrow T$ for $h \in H$,
which we denote by $HU$.
Moreover, $HU$ inherits the $H$-action from $T$ 
(by either its construction or its characterization).
The natural map
$G \cdot_H HU \to T$
is then injective on edges
(Remark \ref{ROOTISOMONO REM})
and we thus let $GU$ be $G \cdot_H HU$
with the planar structure induced from $T$. The claim that $GU$ is the smallest planar orbital face such that $U \hookrightarrow GU$ is clear.

	Before tackling the general case, we collect some key observations.
	Firstly, if $U$ is outer then so is the (non-equivariant) face $HU$ and the orbital face $GU$.
	Secondly, the root tuple of 
	$GU$ is $G\cdot_H r_U$.
	Lastly, we need to characterize the leaf tuple of $GU$. We call a leaf $l$ of $U$ \textit{orbital} if 
all the edges in $Hl \cap \boldsymbol{E}(U)$ are leaves of $U$, 
	and claim that the leaves of $GU$ are the tuple $\underline{l}$ formed by the $G$-orbits of the orbital leaves of $U$. 
	Indeed, a leaf $l$ of $U$ is also a leaf of $HU$ iff 
	$\forall_{h \in H}$($l \in \boldsymbol{E}(hU)$ implies that $l$ is a leaf of $hU$) iff
	$\forall_{h \in H}$($h^{-1} l \in \boldsymbol{E}(U)$ implies that $l$ is a leaf of $U$).

	In the general case, we define $GU$ as the orbital inner face of $G \bar{U}$ that removes all edge orbits not represented in $U$ 	(that all such edge orbits are inner follows from the description of the roots and leaves of $G\bar{U}$ in the previous paragraph). 
	For the remaining claim that $G U$ is the smallest planar orbital face with $U \hookrightarrow GU$, let $U \hookrightarrow S$ be any such face, and write $S \hookrightarrow \bar{S} \hookrightarrow T$ for the planar orbital factorization given by Proposition \ref{INNOUTORB PROP}.
	Then by the outer case established before it is	
	$GU \hookrightarrow G \bar{U} \hookrightarrow \bar{S}$ and, since all edges of $GU$ (which are the orbits of the edges of $U$) are in $S$, it follows that $GU \hookrightarrow S \hookrightarrow \bar{S}$ due to the inner face inclusion 
	$S \hookrightarrow \bar{S}$ being full (Remark \ref{INNFULLORB REM}).
\end{proof}

\begin{example}
Much of the complexity in the previous proof is needed to handle  the scenario of non outer faces $U \hookrightarrow T$
of $G$-trees $T$ which have stumps,
which is easily the subtlest case, as illustrated by the following
example (where $G = \mathbb{Z}_{/2}=\{\pm1\}$).
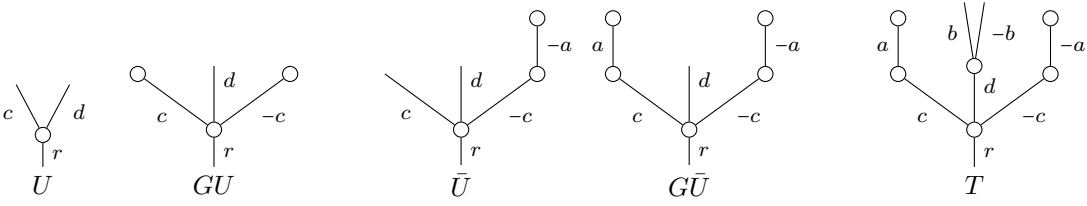
\[%
	\begin{tikzpicture}[auto,grow=up, every node/.style = {font=\footnotesize}]%
	\begin{scope}[level distance = 1.9em]
	\tikzstyle{level 2}=[sibling distance=2em]%
	\tikzstyle{level 3}=[sibling distance=0.75em]%
		\node at (0,0) [font=\normalsize]{$U$}%
			child{node [dummy] {}%
				child{
				edge from parent node [swap, near end] {$d$}}%
				child{
				edge from parent node [near end] {$\phantom{d}c$}}%
			edge from parent node [swap] {$r$}};%
	\end{scope}
	\begin{scope}[level distance = 2.1em]
	\tikzstyle{level 2}=[sibling distance=2.85em]%
	\tikzstyle{level 3}=[sibling distance=0.75em]%
		\node at (2.25,0) [font=\normalsize]{$GU$}%
			child{node [dummy] {}%
				child{node [dummy] {}%
				edge from parent node [swap] {$-c$}}%
				child[level distance = 2.4em]{
				edge from parent node [near end,swap] {$d$}}%
				child{node [dummy] {}%
				edge from parent node {$\phantom{-}c$}}%
			edge from parent node [swap] {$r$}};%
		\node at (5.5,0) [font=\normalsize]{$\bar{U}$}%
			child{node [dummy] {}%
				child{node [dummy] {}%
					child{node [dummy] {}
					edge from parent node [swap] {$-a$}}%
				edge from parent node [swap] {$-c$}}%
				child[level distance = 2.4em]{
				edge from parent node [near end,swap] {$d$}}%
				child{
				edge from parent node {$\phantom{-}c$}}%
			edge from parent node [swap] {$r$}};%
		\node at (8.5,0) [font=\normalsize]{$G\bar{U}$}%
			child{node [dummy] {}%
				child{node [dummy] {}%
					child{node [dummy] {}
					edge from parent node [swap] {$-a$}}%
				edge from parent node [swap] {$-c$}}%
				child[level distance = 2.4em]{
				edge from parent node [near end,swap] {$d$}}%
				child{node [dummy] {}%
					child{node [dummy] {}
					edge from parent node {$\phantom{-}a$}}%
				edge from parent node {$\phantom{-}c$}}%
			edge from parent node [swap] {$r$}};%
		\node at (12.25,0) [font=\normalsize]{$T$}%
			child{node [dummy] {}%
				child{node [dummy] {}%
					child{node [dummy] {}
					edge from parent node [swap] {$-a$}}%
				edge from parent node [swap] {$-c$}}%
				child[level distance = 2.4em]{node [dummy] {}%
					child{
					edge from parent node [near end,swap]{$-b$}}%
					child{
					edge from parent node [near end]{$\phantom{-}b$}}%
				edge from parent node [near end,swap] {$d$}}%
				child{node [dummy] {}%
					child{node [dummy] {}
					edge from parent node {$\phantom{-}a$}}%
				edge from parent node {$\phantom{-}c$}}%
			edge from parent node [swap] {$r$}};%
	\end{scope}
	\end{tikzpicture}%
\]%

\end{example}

\begin{remark}\label{GOUT REM}
If $U \hookrightarrow T$ is outer, the characterization of the roots and leaves of $GU$ in the previous proof shows that
the inner edges of $GU$ are the $G$-orbits of the inner edges of $U$.
\end{remark}

\begin{remark}\label{GINNER REM}
	For any inner face $V-e$ 
	of a face $V \hookrightarrow T$ one has 
	that $G(V-e)$ is either $GV - Ge$ or $GV$.
	Indeed, the latter holds iff $V-e$ contains either an inner edge or a leaf of the form $ge$.
\end{remark}

\begin{remark}
Writing $\mathsf{Face}_o(T)$ for the poset of planar orbital faces, Proposition \ref{MINGFACT PROP} gives a $G$-equivariant functor (note that $G$ acts trivially on $\mathsf{Face}_o(T)$)
\[
	\mathsf{Face}(T) \xrightarrow{G(-)} \mathsf{Face}_o(T),
\qquad
	U \mapsto GU.
\]
Moreover, there is a natural inclusion
$\mathsf{Face}_o(T) \subseteq \mathsf{Face}(T)/G$ (sending an orbital face $S$ to the class of components $[S_{\**}]$)
whose left adjoint is the induced functor 
$\mathsf{Face}(T)/G \to \mathsf{Face}_o(T)$.
\end{remark}

\begin{remark}\label{ORB_FACE_REM}
	Following the intuition in Example \ref{ORBFACE EX}, there is an isomorphism of posets
	\begin{equation}
	\mathsf{Face}_o(T) \xrightarrow{\simeq} \mathsf{Face}(T/G), \qquad S \mapsto S/G,
	\end{equation}
	where $T/G$, which is formally defined below, can be informally thought of as the underlying tree in the orbital representation of $T$. 
	
	However, we should first caution that though this claim is intuitive, some care is needed.
	For example, the broad poset of $T/G$ is in general \textit{not} the quotient of the broad poset of $T$,
	as that may fail the simplicity axiom in 
	\cite[Def. 5.9]{Per18}.
	In fact, the assignment $T \mapsto T/G$ is \textit{not}
	a functor $\Omega_G \to \Omega$, as shown by
	the following (for $G = \mathbb{Z}_{/2} = \{\pm 1\}$),
	since no dashed arrow exists.
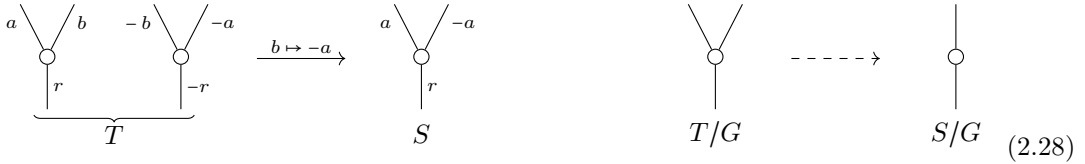
\begin{equation}\label{QUOTMAP EQ}
	\begin{tikzpicture}[auto,grow=up, level distance = 2.2em,
	every node/.style={font=\scriptsize,inner sep = 2pt}]%
		\tikzstyle{level 2}=[sibling distance=2.25em]%
		\begin{scope}
			\node at (0,0){}%
				child{node [dummy] {}%
					child{node {}%
					edge from parent node [swap,very near end] {$b\phantom{j}$}}%
					child{node {}%
					edge from parent node [very near end] {$\phantom{-j}a$}}%
				edge from parent node [swap] {$r\phantom{j}$}};%
			\node (T2) at (5em,0){}%
				child{node [dummy] {}%
					child{node {}%
					edge from parent node [swap,very near end] {$-a\phantom{j}$}}%
					child{node {}%
					edge from parent node [very near end] {$\phantom{j} -b$}}%
				edge from parent node [swap] {$-r$}};%
			\node (S) at (14em,0){}%
				child{node [dummy] {}%
					child{node {}%
					edge from parent node [swap,very near end] {$-a\phantom{j}$}}%
					child{node {}%
					edge from parent node [very near end] {$\phantom{j} a$}}%
				edge from parent node [swap] {$r$}};%
			\node at ($(S)+(0,-0.25)$) [font=\normalsize] {$S$};
		\draw[decorate,decoration={brace,amplitude=2.5pt}] (5.5em,0) -- (-0.5em,0) node[midway,inner sep=4pt,font=\normalsize]{$T$}; %
		\draw[->]
	($(T2) + (1,.75)$) -- ($(S)+(-1,.75)$) node [midway,above] {$b \mapsto -a$};
		\end{scope}%
		\begin{scope}[xshift=25em]
			\node (TG) at (0,0){}%
				child{node [dummy] {}%
					child{node {}%
					edge from parent node {}}%
					child{node {}%
					edge from parent node {}}%
				edge from parent node {}};%
			\node at ($(TG) + (0,-0.25)$) [font=\normalsize] {$T/G$};
			\node (SG) at (9em,0){}%
				child{node [dummy] {}%
					child{node {}%
					edge from parent node {}}%
				edge from parent node {}};%
			\node at ($(SG) + (0,-0.25)$) [font=\normalsize] {$S/G$};
			\draw[->,dashed]
			($(TG) + (1,.75)$) -- ($(SG)+(-1,.75)$);
		\end{scope}
	\end{tikzpicture}%
\end{equation}%
We now outline the formal construction of $T/G$,
starting with some preliminary notation.

Given $\underline{e},\underline{f}$ tuples of edges of $T$, 
write $\underline{f} \leq \underline{e}$
if $\underline{e} = e_1 e_2 \cdots e_k$
and there is a tuple decomposition
$\underline{f} = 
\underline{f}_1 \underline{f}_2 \cdots
\underline{f}_k$
such that $\underline{f}_i \leq e_i$.
When the $e_i$ are $\leq_d$-incomparable,
\cite[Prop. 5.30]{Per18} says that such decomposition is unique, so that $\underline{e},\underline{f}$ consist of distinct edges and we can regard 
$\underline{e},\underline{f}$ as subsets $\underline{e},\underline{f} \subseteq \boldsymbol{E}(T)$.
	
We now say that a relation
$\underline{f} \leq \underline{e}$	
is an \textit{orbital relation} if
$\underline{e} \subseteq \boldsymbol{E}(T)$
is a transitive $G$-subset and $\underline{f} \subseteq \boldsymbol{E}(T)$ is a $G$-subset. 
Reinterpreting the orbital relations of $T$ 
as broad relations on the set 
$\boldsymbol{E}_{G}(T) = \boldsymbol{E}(T)/G$ of edge orbits,
one readily checks that this defines a 
dendroidally ordered set \cite[Def. 5.9]{Per18},
i.e. a tree, that we denote $T/G$.
Note that one hence has a functor
$(-)/G \colon \Omega_G^+ \to \Omega$,
where $\Omega_G^+$ is the subcategory of orbital face maps,
and planarizations of the $T/G$ are chosen arbitrarily.

Lastly, we observe that, in analogy to the non-equivariant case,
the orbital outer faces of $T$ are indexed by orbital relations.
\end{remark}

\subsection{Equivariant dendroidal sets}\label{EQDENDSETS SEC}

Recall \cite[\S 5.4]{Per18} that the category of 
\textit{$G$-equivariant dendroidal sets}
is the presheaf category 
$\mathsf{dSet}^G = \mathsf{Set}^{\Omega^{op} \times G}$.
Given $T \in \Omega_G$ with non-equivariant tree components $T_1,\cdots,T_k$,
we extend the usual notation for representable functors 
to obtain $\Omega[T] \in \mathsf{dSet}^G$ via
\[
	\Omega[T] = \Omega[T_1] \amalg \cdots \amalg \Omega[T_k]
\]
regarded as a $G$-object in $\mathsf{dSet}$.
One further defines \textit{boundaries} (in the union formula we regard the injections $\Omega[U] \to \Omega[T]$ as inclusions; the equivalence between the colimit and union formulas follows from Proposition \ref{UNIQUEFACT PROP})
\[
	\partial \Omega[T] = 
	\mathop{\colim}_{U \in \mathsf{Face}(T),U \neq T_i}
	\Omega[U] =
	\bigcup_{U \in \mathsf{Face}(T),U \neq T_i}
	\Omega[U]
\]
and, for $\emptyset \neq E \subseteq \boldsymbol{E}^{\mathsf{i}}(T)$ a
non-empty $G$-subset of inner edges 
(we abbreviate $E_i = E \cap \boldsymbol{E}^{\mathsf{i}}(T_i)$), \textit{$G$-inner horns}
\[
	\Lambda^{E}[T] = 
	\mathop{\colim}_{U \in 
	\mathsf{Face}(T),
	(T_i - E_i) \not \hookrightarrow U}
	\Omega[U] =
	\bigcup_{U \in 
	\mathsf{Face}(T),
	(T_i - E_i) \not \hookrightarrow U}
	\Omega[U]
\]
which, informally, are the subcomplexes of $\Omega[T]$ that remove the inner faces $T_i-D$ for $D \subseteq E_i$.

Lastly, letting $\mathsf{Face}_{sc}(T)$ denote those outer faces of $T$ with no inner edges (these are either single edges $t$ or generated by single vertices $t^{\uparrow} \leq t$), we define the 
\textit{Segal core of $T$}
\[
	Sc[T] 
= 
	\mathop{\colim}_{U \in 
	\mathsf{Face_{sc}}(T)}
	\Omega[U] 
=
	\bigcup_{U \in 
	\mathsf{Face}_{sc}(T)}
	\Omega[U].
\]

Note that if $T \simeq G \cdot_H T_{\**}$ for some $T_{\**} \in \Omega^H$ then 
\begin{equation}\label{T_DECOMP_EQ}
	\Omega[T] \simeq G \cdot_H \Omega[T_{\**}], 
\quad
	\partial \Omega[T] \simeq G \cdot_H \partial \Omega[T_{\**}], \quad
	\Lambda^{E}[T] \simeq G \cdot_H \Lambda^{E_{\**}}[T_{\**}],
\quad
	Sc[T] \simeq G \cdot_H Sc[T_{\**}].
\end{equation}
As a cautionary note, we point out that though representable functors $\Omega[T]$ are defined for $T \in \Omega_G$,
evaluations $X(U)$ of $X \in \mathsf{dSet}^G$
are defined only for $U \in \Omega$ (cf. Notation \ref{TREEDIFNOT NOT}).

\begin{remark}\label{FACEGACT REM}
	For $T \in \Omega_G$, a planar face $\varphi_U \colon U \hookrightarrow T$
	can also be regarded as a dendrex $\varphi_U \in \Omega[T](U)$.
	However, the $G$-isotropy $H$ of $U \in \mathsf{Face}(T)$ must not be confused with the $G$-isotropy of $\varphi_U$.
	Instead, $\Omega[T](U)$ has a larger $G \times \mathsf{Aut}(U)^{op}$-action,	
	where $\mathsf{Aut}(U)^{op}$ acts by precomposition,
	and the $G \times \mathsf{Aut}(U)^{op}$-isotropy of $\varphi_U$
	is a subgroup 
	$\Gamma \leq G \times \mathsf{Aut}(U)$
	consisting of elements
	$(h,\phi^{-1}(h))$
	where $h\in H$ and
	$\phi\colon H \to \mathsf{Aut}(U)$ is a homomorphism.
	Indeed, noting that \eqref{FACEGACT EQ} implies there is an identity $\Omega[gU] = g \Omega[U]$ as subpresheaves of $\Omega[T]$, it follows that the $G$-isotropy $H$ of $U \in \mathsf{Face}(T)$ coincides with the $G$-isotropy of the subpresheaf $\Omega[U]\subseteq \Omega[T]$ so that, by Yoneda, 
	$U \in \Omega$ has a canonical $H$-action 
	$\phi \colon H \to \mathsf{Aut}(U)$
	(more explicitly, $\phi(h)$ is the left isomorphism in \eqref{FACEGACT EQ}).
	We abuse notation by writing 
	$U \in \Omega^H \subseteq \Omega_H$ to denote this.  
\end{remark}


Recall that a class of maps is called \textit{saturated} if it is closed under pushouts, transfinite composition and retracts.

The saturation of the boundary inclusions 
$\partial \Omega[T] \to \Omega[T]$
is the class of \textit{$G$-normal monomorphisms},
i.e. those monomorphisms $X \to Y$ in $\mathsf{dSet}^G$ such that
$Y(U) \setminus X(U)$ has an $\mathsf{Aut}(U)$-free action for all $U \in \Omega$ (non-equivariantly, this is \cite[Prop. 1.5]{CM11}, and also holds equivariantly by \cite[Rem. 6.7]{Per18}; alternatively, it can be shown using \cite[Props. 6.5(ii) and 5.62]{Per18}).
Moreover, since one can forget the $G$-action when verifying this condition, we will usually call these simply \textit{normal monomorphisms}.

The saturation of the $G$-inner horn inclusions 
$\Lambda^E[T] \to \Omega[T]$
is called the class of \textit{$G$-inner anodyne maps}, 
while those $X \in \mathsf{dSet}^G$
with the right lifting property against all $G$-inner horn inclusions are called \textit{$G$-$\infty$-operads}.

We can now recall the statement of \cite[Thm 2.1]{Per18}, which was the main result therein.

\begin{theorem}\label{DSETGMODEL THM}
	There is a model structure on $\mathsf{dSet^G}$
	such that the cofibrations are the normal monomorphisms and the fibrant objects are the $G$-$\infty$-operads.
\end{theorem}

\begin{remark}
The definition of $G$-$\infty$-operads just given is a priori distinct from the original definition \cite[Def. 6.12]{Per18} which used only 
\textit{generating $G$-inner horn inclusions}, i.e. 
those inclusions $\Lambda^{Ge}[T] \to \Omega[T]$ with $E=Ge$ an inner edge orbit.
The definition herein has the technical advantages of being naturally compatible with restricting the $G$-action and of allowing for a simpler proof of Lemma \ref{CHAREDGE LEM}, 
which is our main tool for showing that maps
are $G$-inner anodyne.
The equivalence between the two definitions follows from 
\cite[Prop. 6.17]{Per18},
although we will also independently recover this fact
from Lemma \ref{CHAREDGE LEM} as Corollary \ref{REGGENHORN COR}.
\end{remark}

In addition to the $G$-inner horns defined above, we now introduce a new kind of horn that, much like orbital faces,
is naturally suggested by the orbital representation of $G$-trees.
Given $\emptyset \neq E \subseteq \boldsymbol{E}^{\mathsf{i}}(T)$ a $G$-equivariant set of inner edges, we define the associated 
\textit{orbital $G$-inner horn} by
\[
	\Lambda^{E}_o[T] = 
	\mathop{\colim}_{S \in 
	\mathsf{Face}_o(T),
	(T - E) \not \hookrightarrow S}
	\Omega[S] =
	\bigcup_{S \in 
	\mathsf{Face}_o(T),
	(T - E) \not \hookrightarrow S}
	\Omega[S]
\]
where we note that the equivalence between the colimit and union formulas now follows from Proposition \ref{MINGFACT PROP}.

\begin{remark}\label{ORB_HORN_REM}
One can strengthen the identification 
$\mathsf{Face}_o(T) \simeq \mathsf{Face}(T/G)$
in Remark \ref{ORB_FACE_REM}.

Say a subcomplex $A \subseteq \Omega[T]$ is \textit{orbital}
if it is the union of orbital faces $\Omega[S]$, $S\in \mathsf{Face}_o(T)$. 
Equivalently, by Proposition \ref{MINGFACT PROP} this means that for $U \in \mathsf{Face}(T)$
one has $\Omega[U] \subseteq A$ iff $\Omega[GU] \subseteq A$. There is then a natural bijection 
of posets (under inclusion)
\begin{equation}
	\set{
	\text{orbital subcomplexes $\bigcup_i\Omega[S_i]$ of $\Omega[T]$}}
		\leftrightarrow
	\set{\text{subcomplexes $\bigcup_i\Omega[S_i/G]$ of $\Omega[T/G]$}}.
\end{equation}
In particular, note that $\Lambda^{Ge}_o[T]$ corresponds to $\Lambda^{[e]}[T/G]$
and $Sc[T]$ corresponds to $Sc[T/G]$.
\end{remark}

\begin{example}\label{HORNEX EX}
	Let $G = \mathbb{Z}_{/2} = \{\pm 1\}$,
	and consider the tree $T \in \Omega^G \subset \Omega_G$ at the top below.
	The following depicts the poset of planar faces of $T$ \textit{not in} $\Lambda_o^{Gb}[T]$.
	By contrast, $\Lambda^{Gb}[T]$ lacks only the boxed faces (which are precisely those faces pictured below that are \textit{inner} faces of $T$).
        \noeqref{HORN_EX_FIG}
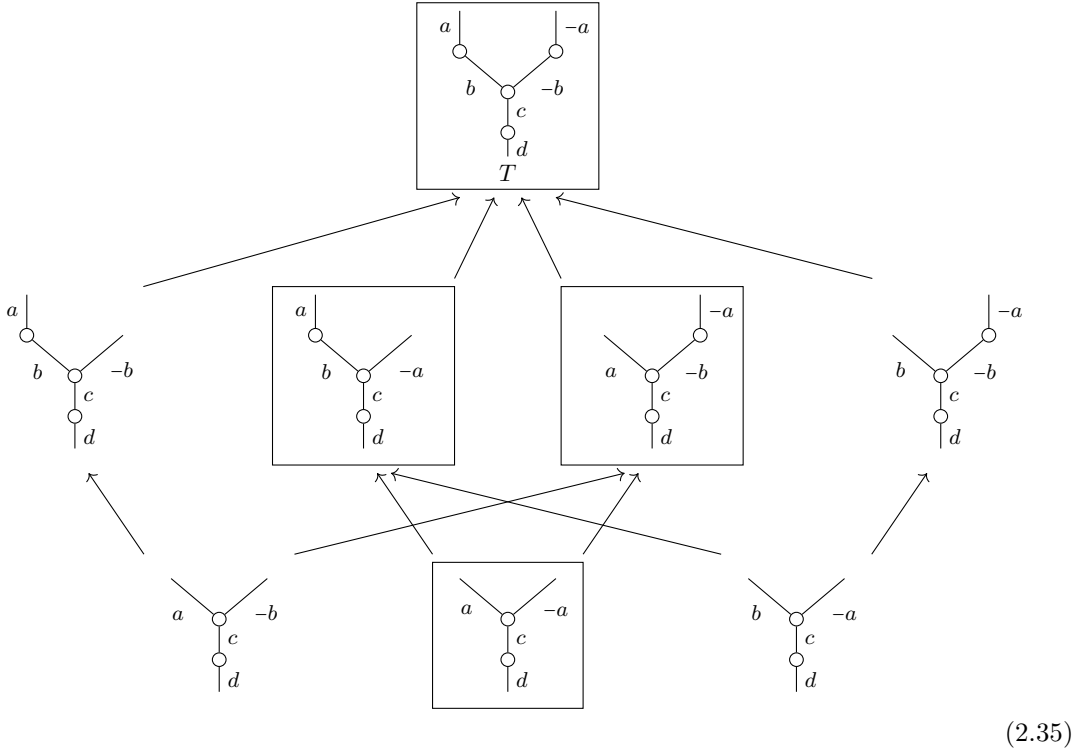
	      \begin{equation}\label{HORN_EX_FIG}
            \newlength{\mywidth}
            \setlength{\mywidth}{1.7em}
            \newlength{\rh}
            \setlength{\rh}{6\mywidth}
            \begin{tikzpicture}[grow= up, level distance = \mywidth, auto,
                  y=2\mywidth,
                  every node/.style = {font=\small, transform shape},
                  scale=.9]
                  \tikzstyle{level 2}=[sibling distance=4em]
                  \tikzstyle{level 3}=[sibling distance=4em]
                  \node (T) at (0,\rh+\mywidth) [font=\normalsize]{$T$}
                  child{node [dummy] {}
                    child{node [dummy] {}
                      child{node [dummy] {}
                        child{edge from parent node [swap] {$-a$}}
                        edge from parent node [swap] {$-b$}}
                      child{node [dummy] {}
                        child{edge from parent node [] {$a$}}
                        edge from parent node {$b$}}
                      edge from parent node [swap] {$c$}}
                    edge from parent node [swap] {$d$}};
                  \node (Tb) at (6em,0) {}
                  child{node [dummy] {}
                    child{node [dummy] {}
                      child{node [dummy] {}
                        child{edge from parent node [swap] {$-a$}}
                        edge from parent node [swap] {$-b$}}
                      child{
                        edge from parent node [] {$\phantom{-b}a$}
                      }
                      edge from parent node [swap] {$c$}}
                    edge from parent node [swap] {$d$}};
                  \node (Tbb) at (-6em,0) {}
                  child{node [dummy] {}
                    child{node [dummy] {}
                      child{edge from parent node [swap] {$-a\phantom{b}$}}
                      child{node [dummy] {}
                        child{edge from parent node [] {$a$}}
                        edge from parent node [] {$b$}}
                      edge from parent node [swap] {$c$}}
                    edge from parent node [swap] {$d$}};
                  \node (Paa) at (-18em,0) {}
                  child{node [dummy] {}
                    child{node [dummy] {}
                      child{
                        edge from parent node [swap] {$-b$}}
                      child{node [dummy] {}
                        child{edge from parent node [] {$a$}}
                        edge from parent node [] {$b$}}
                      edge from parent node [swap] {$c$}}
                    edge from parent node [swap] {$d$}};
                  \node (Pa) at (18em,0) {}
                  child{node [dummy] {}
                    child{node [dummy] {}
                      child{node [dummy] {}
                        child{
                          edge from parent node [swap] {$-a$}}
                        edge from parent node [swap] {$-b$}}
                      child{
                        edge from parent node [] {$b$}}
                      edge from parent node [swap] {$c$}}
                    edge from parent node [swap] {$d$}};
                  \draw[->]
                  (Paa)++(1,2.1) -- ($(T)+(-.7,-.3)$); 
                  \draw[->]
                  (Tbb)++(1.33,2.2) -- ($(T)+(-.2,-.3)$); 
                  \draw[->]
                  (Pa)++(-1,2.2) -- ($(T)+(.7,-.3)$); 
                  \draw[->]
                  (Tb)++(-1.33,2.2) -- ($(T)+(.2,-.3)$); 
                  \node (TGb) at (0,-\rh) {}
                  child{node [dummy] {}
                    child{node [dummy] {}
                      child{
                        edge from parent node [swap] {$-a$}}
                      child{
                        edge from parent node [] {$a$}}
                      edge from parent node [swap] {$c$}}
                    edge from parent node [swap] {$d$}};
                  \node (Pab) at (-12em,-\rh) {}
                  child{node [dummy] {}
                    child{node [dummy] {}
                      child{edge from parent node [swap] {$-b$}}
                      child{edge from parent node [] {$\phantom{-b}a$}}
                      edge from parent node [swap] {$c$}
                    }
                    edge from parent node [swap] {$d$}};
                  \node (Pabb) at (12em,-\rh) {}
                  child{node [dummy] {}
                    child{node [dummy] {}
                      child{edge from parent node [swap] {$-a\phantom{b}$}}
                      child{edge from parent node [] {$b$}}
                      edge from parent node [swap] {$c$}
                    }
                    edge from parent node [swap] {$d$}};
                  \draw[->]
                  (Pab)++(-1.1,1.8) -- ($(Paa)+(.2,-.2)$); 
                  \draw[->]
                  (Pab)++(1.1,1.8) -- ($(Tb)+(-.4,-.2)$); 
                  \draw[->]
                  (Pabb)++(-1.1,1.8) -- ($(Tbb)+(.4,-.2)$); 
                  \draw[->]
                  (Pabb)++(1.1,1.8) -- ($(Pa)+(-.2,-.2)$); 
                  \draw[->]
                  (TGb)++(-1.1,1.8) -- ($(Tbb)+(.2,-.2)$); 
                  \draw[->]
                  (TGb)++(1.1,1.8) -- ($(Tb)+(-.2,-.2)$); 
                  \draw
                  (TGb)++(-1.1,-.1) rectangle ++(2.2,1.8);
                  \draw
                  (Tbb)++(-1.33,-.1) rectangle ++(2.66,2.2);
                  \draw
                  (Tb)++(-1.33,-.1) rectangle ++(2.66,2.2);
                  \draw
                  (T)++(-1.33,-.2) rectangle ++(2.66,2.3);
            \end{tikzpicture}
      \end{equation}           
\end{example}

\section{Equivariant inner anodyne maps}\label{EQINNERAN SEC}

Much as in \cite[\S 2]{CM13a}, 
we need to show that the inclusions $Sc[T] \to \Omega[T]$, $T\in \Omega_G$ are $G$-inner anodyne.
In addition, parts of the equivariant dendroidal story are 
naturally described in terms of 
orbital $G$-inner horns $\Lambda^E_o[T]$ 
(rather than $G$-inner horns $\Lambda^E[T]$),
and one must hence also show that the inclusions
$\Lambda^E_o[T] \to \Omega[T]$
are $G$-inner anodyne.

In practice, the proofs of such results are long as well as somewhat repetitive, since they share many technical arguments.
In fact, dealing with the case of orbital horn inclusions requires using many of the arguments in the long proof of \cite[Thm 7.1]{Per18}.

As such, we split our technical analysis into two parts.
In \S \ref{CHAREDGE SEC} we prove Lemma \ref{CHAREDGE LEM} which  we call the \textit{characteristic edge lemma} and
which abstractly identifies sufficient conditions for a map to be $G$-inner anodyne
(see Remark \ref{RECOVER REM} for a comparison with previous results in the literature).
Then in \S \ref{HYPERSAT SEC} we deduce that the desired maps
are $G$-inner anodyne by applying Lemma \ref{CHAREDGE LEM},
and further establish Proposition \ref{HYPER PROP} which,
informally,
says that Segal core inclusions, $G$-inner horn inclusions and orbital $G$-inner horn inclusions can be used interchangeably in some contexts.

Lastly, \S \ref{GENEQOP SEC} briefly discusses \textit{colored genuine equivariant operads}
(which in the single color case were first introduced in \cite{BP17}), which play an important role in \S \ref{HMPTYGEN SEC}.

\subsection{The characteristic edge lemma} \label{CHAREDGE SEC}

\begin{definition}\label{CHAREDGE DEF}
	Let $T \in \Omega_G$, $A \subseteq \Omega[T]$ a subdendroidal set, and $\{U_i\}_{i \in I} \subseteq \mathsf{Face}(T)$ a subset.

	Given a set $\Xi^i$ of inner edges of $U_i$ and a subface $V \hookrightarrow U_i$, denote $\Xi_V^i = \Xi^i \cap \boldsymbol{E}^{\mathsf{i}}(V)$.

	Suppose further that the indexing set $I$ is a 
	finite $G$-poset. For each $i \in I$ denote
\[
	A_{<i} = A \cup \bigcup_{j<i} \Omega[U_j]
\]
	We say that $\{\Xi^i \subseteq \boldsymbol{E}^{\mathsf{i}}(U_i)\}$
	is a \textit{characteristic inner edge collection of $\{U_i\}$ with respect to $A$} if:
	\begin{itemize}
	\item[(Ch0)] $A$, $\{U_i\}$ and $\{\Xi^i\}$ are all $G$-equivariant, i.e. $g A = A$, $g U_i = U_{gi}$, $g \Xi^i = \Xi^{gi}$ as appropriate; 
	\item[(Ch1)] for all $i$, any outer face $V \simeq \bar{V}^{U_i}$
		of $U_i$ such that $\Xi_{V}^i = \emptyset$
		is contained in $A_{<i}$;
	\item[(Ch2)] for all $i$, any face
		$V \hookrightarrow U_i$ such that $(V-\Xi_V^i) \in A$
		is contained in $A_{<i}$;
	\item[(Ch3)] for all $j \not \geq i$, 
		all faces $V \hookrightarrow U_i$ such that 
		$(V-\Xi^i_V) \hookrightarrow U_j$
		are contained in $A_{<i}$.
	\end{itemize}
\end{definition}

\begin{remark}\label{XIIII REM}
If $g i \neq i$, then $i,g i$ are incomparable in $I$. Indeed, if $i<gi$ then $i<gi<g^2i<\cdots$ would violate antisymmetry, and likewise if $i>gi$.
Hence, (Ch3) applies when $j=gi$ for $gi\neq i$.

In particular, we assume throughout that if
$gi \neq i$ then $U_{gi} \neq U_i$,
or else $U_i$ would be in $A_{<i}$.
\end{remark}

\begin{remark}\label{SOMEMAIN REM}
In some of the main examples (see Propositions \ref{ORB_HORN_PROP} and \ref{REG_HORN_PROP}), there exists a $G$-equivariant set 
$\Xi$ of inner edges of $T$ such that $\Xi^i = \Xi \cap \boldsymbol{E}^{\mathsf{i}}(U_i)$.
	
We caution that, for fixed $A$ and $\{U_i\}$, our characteristic conditions are \textit{not} monotone on such $\Xi$ since increasing $\Xi$ makes (Ch1) more permissive while making (Ch2),(Ch3) more restrictive.
\end{remark}

\begin{lemma}\label{CHAREDGE LEM}
If $\{\Xi^i\}_{i \in I}$ is a characteristic inner edge collection of $\{U_i\}_{i\in I}$ with respect to $A$, then the map
	\begin{equation}\label{CHARLEM EQ}
		A \to A \cup \bigcup_{i \in I} \Omega[U_i]
	\end{equation}
is $G$-inner anodyne. 
In fact, \eqref{CHARLEM EQ} is built cellularly\footnote{Recall that a map $f$ is built cellularly from a class of maps $I$ if it can be obtained as a (possibly transfinite) composition of pushouts of maps in $I$ (but without retracts). See \cite[Def. 2.1.9]{Hov99},
\cite[Example 12.6.12]{Ri14}.
}
from $G$-inner horn inclusions $\Lambda^E[S] \to \Omega[S]$, $S \in \Omega_G$.
\end{lemma}

Recall that a subset $S \subseteq \mathcal{P}$ of a poset $\mathcal{P}$ is called \textit{convex}\footnote{Here our terminology follows that in \cite[\S 0]{Go91}. Such subsets are sometimes also called \textit{downward closed}.} if $s \in S$ and 
$p<s$ implies $p \in S$.

\begin{proof}
We start with the case of $I \simeq G/H$ transitive so that, abbreviating $U = U_{[e]}$, $\{U_i\}$ is the set of conjugates $gU$. 
We likewise abbreviate $\Xi = \Xi^{[e]}$ and
$\Xi_V = \Xi_V^{[e]}$ for $V \hookrightarrow U$.
Moreover, in this case one has $A_{<[g]}=A$ in (Ch1),(Ch2),(Ch3) and that $H$ is also the isotropy of $U$ in $\mathsf{Face}(T)$.

We write $\mathsf{Face}_{\Xi}^{lex}(U)$
for the $H$-poset of planar faces $V \hookrightarrow U$
such that $\Xi_V \neq \emptyset$ and $\Xi_V = \Xi_{\bar{V}}$
ordered as follows: 
$V \leq V'$ if either
	\begin{inparaenum}
		\item[(i)] $\bar{V} \hookrightarrow \bar{V}'$ and 
		$\bar{V} \neq \bar{V}'$ or
		\item[(ii)] $\bar{V} = \bar{V}'$ and
		$V \hookrightarrow V'$
	\end{inparaenum}
(alternatively, this is the lexicographic order of pairs $(\bar{V},V))$.
We note that here and in the remainder of the proof all outer closures are implicitly taken in $U$ (rather than $T$), i.e. 
$\bar{V}=\bar{V}^U$.

For any $H$-equivariant convex subset $C$ of $\mathsf{Face}_{\Xi}^{lex}(U)$ we write (for the identity $\Omega[gV]=g\Omega[V]$, see Remark \ref{FACEGACT REM})
\[
A_C = 
A \cup \bigcup_{g\in G,V \in C} \Omega[gV]= 
A \cup \bigcup_{g\in G,V \in C} g \Omega[V].
\]
It now suffices to show that whenever
$C \subseteq C'$
the map 
$A_C \to A_{C'}$ is built cellularly from 
$G$-inner horn inclusions
(indeed, setting $C=\emptyset$ and 
$C'=\mathsf{Face}_{\Xi}^{lex}(U)$ recovers \eqref{CHARLEM EQ}
when $I \simeq G/H$).

Without loss of generality we can assume that $C'$ is obtained from $C$ by adding the $H$-orbit of a single $W \hookrightarrow U$. Further, we may assume $W \not \in A_C$ or else $A_C=A_{C'}$.
Letting $K \leq H$ denote the isotropy of 
$W$ in $\mathsf{Face}_{\Xi}^{lex}(U)$ and regarding $W \in \Omega^{K}\subseteq \Omega_{K}$, we claim there is a pushout diagram
\begin{equation}\label{FIRPUSH EQ}
\begin{tikzcd}
	G \cdot_{K} \Lambda^{\Xi_{W}}[W] \ar{d}[swap]{\iota} \ar{r}
	\arrow[dr, phantom, "\ulcorner", very near start]&
	A_C \ar{d}
\\
	G \cdot_{K} \Omega[W] \ar{r}&
	A_{C'}
\end{tikzcd}
\end{equation}
where we note that the inner edge set $\Xi_{W}$ is $K$-equivariant
since $\Xi_W = \Xi \cap \boldsymbol{E}^{\mathsf{i}}(W)$ and $\Xi$ is $H$-equivariant by (Ch0).
The pushout \eqref{FIRPUSH EQ} will follow once we establish the following claims:
\begin{enumerate}
\item[(a)] all proper outer faces $V$ of $W$ are in $A_C$;
\item[(b)] an inner face $W - D$ of $W$ is in $A_C$ iff 
$D \not \subseteq \Xi_{W}$;
\item[(c)] 
the $G$-isotropy (i.e. the isotropy in $\mathsf{Face}(T)$)
of faces 
$W - D$, $D \subseteq \Xi_{W}$ is contained in $K$.
\end{enumerate}
To see why these claims suffice, write $P$ for the pushout of  \eqref{FIRPUSH EQ} with $A_{C'}$ removed. Surjectivity of $P \to A_{C'}$ is immediate from the assumptions on $C,C'$.
For injectivity, since $A_C \to A_{C'}$ is certainly injective, it suffices to check that the restriction of the bottom horizontal map in \eqref{FIRPUSH EQ} to the complement of the image of the left vertical map 
$\iota$ satisfies:
\begin{inparaenum}
\item[(i)] it is injective and;
\item[(ii)] its image does not overlap with the image of $A_C$. 
\end{inparaenum} 
Condition (ii) follows from (a),(b). As for (i), injectivity is automatic when further restricting to the intersection with an individual conjugate $\Omega[gW] = g\Omega[W]$, with (c) ensuring injectivity across different conjugate intersections.

To check (a), writing $\bar{V}=\bar{V}^U$ for the corresponding outer face of $U$, one has
	\[
	\Xi_V = \Xi \cap \boldsymbol{E}^{\mathsf{i}} (V) 
	= \Xi \cap \boldsymbol{E}^{\mathsf{i}}(W) \cap \boldsymbol{E}^{\mathsf{i}}(\bar{V})
	= \Xi \cap \boldsymbol{E}^{\mathsf{i}}(\bar{W}) \cap \boldsymbol{E}^{\mathsf{i}}(\bar{V})
	= \Xi \cap \boldsymbol{E}^{\mathsf{i}}(\bar{V})
	= \Xi_{\bar{V}}
	\]
where the second step follows from Lemma \ref{INNINT LEM}
(applied to $V \hookrightarrow W \hookrightarrow U$, 
$V \hookrightarrow \bar{V} \hookrightarrow U$)
and the third since by definition of
$\mathsf{Face}_{\Xi}^{lex}(U)$ it is $\Xi_{W} = \Xi_{\bar{W}}$.
Thus either $\Xi_V= \Xi_{\bar{V}} = \emptyset$ so that $\bar{V}\in A$ by (Ch1),
or $\Xi_V = \Xi_{\bar{V}} \neq \emptyset$
so that $V \in \mathsf{Face}_{\Xi}^{lex}(U)$ with $V<W$, and thus $V\in C$. In either case one has $V \in A_C$.

We now check the ``if'' direction of (b).
If $D \not \subseteq \Xi_{W}$
then $W' = W - (D \setminus \Xi_{W})$
is in $\mathsf{Face}_{\Xi}^{lex}(U)$
(since $\bar{W}' = \bar{W}$ and thus
$\Xi_{W'} = \Xi_{W}  = \Xi_{\bar{W}} = \Xi_{\bar{W}'}$)
and $W'<W$, and thus $W' \in A_C$.

For the ``only if'' direction of (b), 
note first that it suffices to consider $D = \Xi_W$.
The assumption $W \not \in A_C$ together with (Ch2) imply that
$W'=W-\Xi_{W}$ is not in $A$, and thus it remains to show that 
$W'$ is not a face of any $gV$ with $g\in G$, $V \in C$.
Suppose otherwise, i.e. $W' \hookrightarrow gV$.
If it were $g \not \in H$, 
then it would be $W' \hookrightarrow gV \hookrightarrow g U \neq U$, and (Ch3) would imply $W\in A$. Thus we need only consider $g\in H$, and since $C$ is $H$-equivariant, we can set $g=e$.
It now suffices to show that if $W' \hookrightarrow V$
then it must be $W \leq V$ in $\mathsf{Face}_{\Xi}^{lex}(U)$,
since by convexity of $C$ this would contradict $W \not \in C$.
Since $W' \hookrightarrow V$ implies 
$\bar{W} = \bar{W}' \hookrightarrow \bar{V}$,
the condition $W \leq V$ is automatic from the definition of $\leq$ unless $\bar{W} = \bar{V}$.
In this latter case, by definition of 
$\mathsf{Face}_{\Xi}^{lex}(U)$ the face $V$ must contain as inner edges all edges in 
$\Xi_V = \Xi_{\bar{V}} = \Xi_{\bar{W}} = \Xi_{W}$,
so that not only $W - \Xi_{W} = W' \hookrightarrow V$ but also $W \hookrightarrow V$. But then it is $W \leq V$ in either case, establishing the desired contradiction. 

We now show (c).
If $g(W-D) = W-D$ then $g(W - \Xi_W) \hookrightarrow U$,
and thus $W - \Xi_W \hookrightarrow g^{-1}U$,
so that by (Ch3) it must be $g \in H$ or else it would be $W \in A$.
Now suppose $h(W-D)=W-D$ with $h\in H$.
Since $\Xi$ is $H$-equivariant (by (Ch0)) and
$\Xi_{W-D} = \Xi_{W} \setminus D$ (due to $D \subseteq \Xi_{W}$) it follows that 
$h(W-\Xi_W) = W-\Xi_W$,
so that we may assume $D = \Xi_W$.
Now note that
$hW$, $h(W-\Xi_W)=W-\Xi_W$, $W$
are all faces of $U$ with a common outer closure $\bar{W}$.
Hence
$h\Xi_{W} = \Xi_{hW} \subseteq \Xi_{\bar{W}} = \Xi_{W}$, where the last step follows since
$W \in \mathsf{Face}_{\Xi}^{lex}(U)$, and by cardinality reasons it must in fact be $h \Xi_{W} = \Xi_{W}$. But then $hW,W$
have the same outer closure and the same inner edges, and thus 
$hW=W$, establishing (c).

Lastly, we address the case of general $I$.
For each $G$-equivariant convex subset $J$ of $I$, set
\[
	A_J = 
	A \cup \bigcup_{j \in J} \Omega[U_j].
\]
As before, it suffices to check that for all convex subsets
$J \subseteq J'$
the map $A_J \to A_{J'}$ is built cellularly from $G$-inner horn inclusions,
and again we can assume that $J'$ is obtained from $J$ by adding a single $G$-orbit $Gj$ of $I$.
By the $I$ transitive case, it now suffices to check that
$\{\Xi^{gj}\}_{gj \in Gj}$ is also a characteristic inner edge collection of $\{U_{g j}\}_{g j \in Gj}$ with respect to $A_J$.
(Ch0) is clear, and since by $G$-equivariance and convexity it is $A_{<gj} \subseteq A_J$,
the new (Ch1),(Ch2),(Ch3)
conditions follow from the original conditions
(for (Ch2), note that if 
$(V - \Xi^i_V) \in A_J$ but $(V - \Xi^i_V) \nin A$, then convexity of $J$ implies $(V - \Xi^i_V) \hookrightarrow U_j$ for some $j \not \geq i$, so that in this case the new (Ch2) follows from the original (Ch3)).
\end{proof}

\begin{remark}\label{CHAREDGE2 REM}
The requirement $A \subseteq \Omega[T]$ in Definition \ref{CHAREDGE DEF} can be relaxed.
Consider an inclusion $A \subseteq B$ with 
$B \in \mathsf{dSet}^G$,
a set of non-degenerate dendrices
$\{b_i \in B(U_i)\}_{i \in I}$
and a collection of edges
$\{\Xi^i \subseteq \boldsymbol{E}^{\mathsf{i}}(U_i)\}_{i \in I}$, 
with $I$ a finite $G$-poset.
Letting
$A_{<i} = A \cup \bigcup_{j<i} b_j\left(\Omega[U_j]\right)$,
suppose then that:
\begin{itemize}
	\item[(Ch0.1)] the maps $b_i \colon \Omega[U_i] \to B$ are injective away from $b_i^{-1}(A_{<i})$;
	\item[(Ch0.2)] $A$, $\{U_i\}$ ,$\{b_i\}$ and $\{\Xi^i\}$
	are all $G$-equivariant in the sense that:
	\begin{inparaenum}
		\item[(i)] $gA = A$;
		\item[(ii)] there are
		isomorphisms
		$U_i \xrightarrow{g} U_{g_i}$
		such that the composite $U_i \xrightarrow{g} U_{g_i} \xrightarrow{\bar{g}} U_{\bar{g} g_i}$ is $U_i \xrightarrow{\bar{g}g} U_{\bar{g} g_i}$ and 
		$U_i \xrightarrow{e} U_{ei}$ is the identity $U_i = U_{ei}$;
		\item[(iii)] the composites
		$\Omega[U_i] \xrightarrow{b_i} B \xrightarrow{g} B$
		and
		$\Omega[U_i]  \xrightarrow{g} \Omega[U_{gi}] \xrightarrow{b_{gi}} B$ coincide;
		\item[(iv)] $g \left(\Xi^i\right)=\Xi^{gi}$.
	\end{inparaenum}
\end{itemize}
The original conditions (Ch1),(Ch2),(Ch3) can now be reinterpreted by, 
for each face $\varphi_V \colon V \hookrightarrow U_i$, 
reinterpreting expressions such as ``$V$ contained in $A_{<i}$'', 
$(V-\Xi^i_V) \hookrightarrow U_j$
as $b_i(\varphi_V) \in A_{<i}$, $b_i(\varphi_{V-\Xi^i_V}) \in b_j(\Omega[U_j])$.

The proof of Lemma \ref{CHAREDGE LEM} now carries through mostly unchanged to show that the inclusion
\begin{equation}
	A \to A \cup \bigcup_{i \in I} b_i(\Omega[U_i])
\end{equation}
is $G$-inner anodyne (again built cellularly from $G$-inner horn inclusions).
Indeed, writing 
$A_C = 
A \cup \bigcup_{g\in G,V \in C} g b_{[e]}(\Omega[V])$,
the obvious analogues of (a),(b) in the proof
show that one can form the analogous diagram \eqref{FIRPUSH EQ} and that
$A_{C'} \setminus A_{C}$ is generated by the dendrices
$g b_i(\varphi_{W-D})$ for $D \subseteq \Xi_W$. 
That \eqref{FIRPUSH EQ} is indeed a pushout then follows from (Ch0.1), which ensures that all dendrices attached by each conjugate inclusion
$g\left(\Lambda^{\Xi_W}[W] \to \Omega[W]\right)$
are distinct, together with the analogue of (c), 
which ensures that all such conjugate inclusions
attach different dendrices.
As a technical note, precisely reformulating (c) requires accounting for the isotropy issue discussed in Remark \ref{FACEGACT REM}, and is thus slightly cumbersome. However, just as in first line of the proof of (c), it is immediate from (Ch3) that
$\Lambda^{\Xi_W}[W] \to \Omega[W]$
and $g\left(\Lambda^{\Xi_W}[W] \to \Omega[W]\right)$
could only possibly attach common dendrices if $g\in H$, so that the bulk of the argument in the proof of (c) concerns the $H$-isotropy in $\mathsf{Face}(W)$, 
and thus carries through with no noteworthy changes.
\end{remark}


\begin{remark} \label{RECOVER REM}
Lemma \ref{CHAREDGE LEM} readily recovers several arguments in the literature:
\begin{itemize}
\item[(i)] In \cite[\S 10]{Rez01} (also \cite[\S 6.2]{Rez10}), Rezk introduces the notion of \textit{covers}, which in our language are the simplicial subsets
$Sc[n] \subseteq A \subseteq \Delta[n]$
such that if $V \hookrightarrow [n]$ is in $A$ then so is the outer closure $\bar{V}^{[n]}$
(in words, $A$ is generated by outer faces).
Similarly, in the proof of \cite[Prop. 2.4]{CM13a}
Cisinski and Moerdijk use subcomplexes
that can be regarded as
\textit{dendroidal covers},
i.e. subcomplexes
$Sc[T] \subseteq A \subseteq \Omega[T]$
such that if $V$ is in $A$ then so is $\bar{V}^{T}$
(note that this definition extends unchanged to the equivariant context).
Lastly, the subcomplexes 
$\Omega[T] \cup_l \Omega[S] \subseteq \Omega[T \circ_l S]$
in the grafting result \cite[Lemma 5.2]{MW09},
and likewise for the equivariant analogue \cite[Prop. 6.19]{Per18}, are also dendroidal covers.

Lemma \ref{CHAREDGE LEM} implies
that any inclusion $A \to A'$ of $G$-equivariant (dendroidal) covers for $T\in \Omega_G$
is $G$-inner anodyne. 
Indeed, let $I=\mathsf{Face}_{A'}^{out}(T)$ be the $G$-poset of planar outer faces $V \hookrightarrow T$ contained in $A'$, ordered by inclusion, 
$\Xi = \boldsymbol{E}^{\mathsf{i}}(T)$ and $U_V = V$.
(Ch0) is clear, (Ch1) follows since 
$Sc(T) \subseteq A$, (Ch2) follows since $A$ is a cover and
(Ch3) follows since the $U_i$ are outer.

Alternatively, one can also use $I=\mathsf{Face}_{A',o}^{out}(T)$
for the $G$-trivial set of orbital outer faces 
$GV \hookrightarrow T$,
together with an \textit{arbitrary} total order (see Remark \ref{TWOPROOF REM} for a similar example).

Lastly, note that in the special case 
$\{U_i\}=\{T\}$, $\Xi=\boldsymbol{E}^{\mathsf{i}}(T)$,
(Ch1) says precisely $Sc[T] \subseteq A$.

\item[(ii)] In \cite[Lemma 9.7]{MW09}, Moerdijk and Weiss introduced a \textit{characteristic edge} condition that can be regarded as a special case of our characteristic edge collection condition as generalized in Remark \ref{CHAREDGE2 REM}, and which served as one of our main inspirations.

Therein, they work in the case of $B= \Omega[T] \otimes \Omega[S]$
a tensor product of (non-equivariant) representable dendroidal sets, in which case (Ch0.1) is easily verified (and (Ch0.2) is automatic).
In our notation, they then require that $I \simeq \**$ (so that (Ch3) is also automatic), 
the dendrex 
$b_{\**} \in \left(\Omega[T] \otimes \Omega[S]\right)(U_{\**})$ encodes a special type of subtree $U_{\**}$ of $\Omega[T] \otimes \Omega[S]$, which they call an \textit{initial segment},
and they further require that $\Xi^{\**} = \{\xi\}$ is a singleton, called the \textit{characteristic edge}.
Moreover, they then demand that $A$ should contain all outer faces of the subtree $U_{\**}$, from which (Ch1) follows, 
as well as the key characteristic condition 
\cite[Lemma 9.7]{MW09}(ii),
which coincides with (Ch2) in this specific setting.

Similarly, in \cite[Lemma 7.39]{Per18} the second author introduced a \textit{characteristic edge orbit} condition that generalizes that in \cite{MW09} to the equivariant context 
by letting $I \simeq G/H$
and the $\Xi^{[g]}=\Xi \cap \boldsymbol{E}^{\mathsf{i}}(U_{[g]})$ be determined by a $G$-edge orbit $\Xi \simeq G \xi$ (cf. Remark \ref{SOMEMAIN REM}).

However, both of the lemmas in \cite{MW09} and \cite{Per18}
have the drawback of needing to be used iteratively
(so that much effort therein is spent showing that this can be done) while Lemma \ref{CHAREDGE LEM} is designed so that a single use suffices for the natural applications.
Indeed, conditions (Ch1) and (Ch3), the first of which relaxes the requirement in \cite{MW09},\cite{Per18} that $A$ should contain all outer faces of the $U_i$, essentially provide abstract conditions under which the original characteristic edge arguments of \cite{MW09},\cite{Per18} can be iterated.
\end{itemize}
\end{remark}

\begin{example}\label{THM71 EX}
	As indicated above, Lemma \ref{CHAREDGE LEM} can be used to
	reorganize and streamline the rather long proofs of \cite[Thms 7.1 and 7.2]{Per18}. We illustrate this in the hardest case, that of \cite[Thm. 7.1(i)]{Per18}, which states that if
	$S,T \in \Omega_G$ are open (i.e. have no stumps) and
	$G \xi$ is an inner edge orbit of $T$ the maps
\begin{equation}\label{THM71 EQ}
	\partial \Omega[S] \otimes \Omega[T]
		\coprod_{\partial \Omega[S] \otimes \Lambda^{G\xi}[T]}
	\Omega[S] \otimes \Lambda^{G\xi}[T]
\to
	 \Omega[S] \otimes \Omega[T]
\end{equation}
are $G$-inner anodyne.

Given $S,T \in \Omega_G$, it is possible \cite[\S 7.1]{Per18} to define a $G$-equivariant broad poset $S \otimes T$
so that 
$\left(\Omega[S] \otimes \Omega[T]\right)(V) = hom(V,S \otimes T)$
where the hom-set is taken in broad posets.
Intuitively $S \otimes T$ is an object with edge set $\boldsymbol{E}(S) \times \boldsymbol{E}(T)$
and where each edge $(s,t)$ of $S \otimes T$
may, depending on whether $s \in S$, $t \in T$ are leaves or not,
admit two \textit{distinct} vertices: a $S$-vertex 
$(s,t)^{\uparrow S} = s^{\uparrow} \times t \leq (s,t)$
and a $T$-vertex
$(s,t)^{\uparrow T} = s \times t^{\uparrow} \leq (s,t)$.

To recover \cite[Thm. 7.1(i)]{Per18} from Lemma \ref{CHAREDGE LEM}, we first let $I = \mathsf{Max}(S \otimes T)$
be the $G$-poset of maximal subtrees $U \hookrightarrow S \otimes T$ 
(these are called \textit{percolation schemes} in \cite[\S 9]{MW09}), ordered lexicographically \cite[Def. 7.29]{Per18}.
As an example, let $G=\mathbb{Z}_{/2} = \{\pm 1\}$
and consider the $\mathbb{Z}_{/2}$-trees
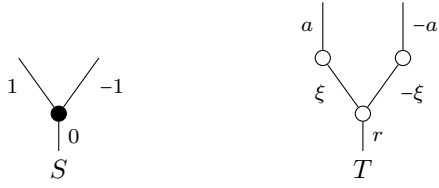
\[%
	\begin{tikzpicture}[auto,grow=up, every node/.style = {font=\footnotesize}]%
	\begin{scope}[level distance = 2.1em]
	\tikzstyle{level 2}=[sibling distance=3em]%
	\tikzstyle{level 3}=[sibling distance=2.5em]%
		\node at (0,0) [font=\normalsize] {$S$}%
			child{node [dummy,fill=black] {}%
				child{
				edge from parent node [near end,swap] {$-1$}}%
				child{
				edge from parent node [near end]{$\phantom{-}1$}}%
			edge from parent node [swap] {$0$}};%
		\node at (4,0) [font=\normalsize] {$T$}%
			child{node [dummy] {}%
				child{node [dummy] {}
					child{
					edge from parent node[swap]{$-a\phantom{-}$}}
				edge from parent node [near end,swap] {$-\xi\phantom{-}$}}%
				child{node [dummy] {}
					child{
					edge from parent node{$\phantom{-}a$}}
				edge from parent node [near end]{$\phantom{-}\xi$}}%
			edge from parent node [swap] {$r$}};%
	\end{scope}
	\end{tikzpicture}%
\]%
We depict the $\mathbb{Z}_{/2}$-poset 
$\mathsf{Max}(S \otimes T)$
in Figure \ref{FIGURE} (note that $(s,t)$ is abbreviated as $t_s$).
In words, the maximal subtrees are built by starting with the
``double root'' $r_0$ and iteratively choosing between the available $S$ and $T$ vertices (along all upward paths) until the 
``double leaves'' $a_1,a_{-1},-a_1,-a_{-1}$ are reached.
The generating relations $U \leq U'$ in a generic
$\mathsf{Max}(S \otimes T)$
occur whenever $U$ contains an outer face $V$ shaped as on the left below and, by ``replacing'' $V$ with $V'$ as on the right, one obtains $U'$.
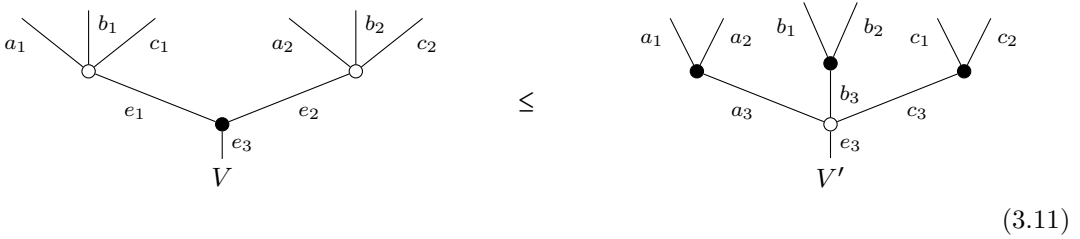
\begin{equation}\label{GENLEXREL EQ} 
	\begin{tikzpicture}[auto,grow=up, level distance = 2em,
	every node/.style={font=\footnotesize},
	dummy/.style = {circle,draw,inner sep=0pt,
	minimum size=1.75mm}]
	\begin{scope}
	\tikzstyle{level 2}=[sibling distance=10em]%
	\tikzstyle{level 3}=[sibling distance=2.5em]%
		\node at (0,0) [font=\normalsize]{$V$}
			child{node [dummy,fill=black] {}
				child{node [dummy] {}
					child{
					edge from parent node [swap,near end] {$c_2$}}
					child[level distance = 2.3em]{
					edge from parent node [swap, near end] {$b_2$}}
					child{
					edge from parent node [near end] {$a_2$}}
				edge from parent node [swap] {$e_2$}}
				child{node [dummy] {}
					child{
					edge from parent node [swap,near end] {$c_1$}}
					child[level distance = 2.3em]{
					edge from parent node [swap, near end] {$b_1$}}
					child{
					edge from parent node [near end] {$a_1$}}				
				edge from parent node {$e_1$}}
			edge from parent node [swap] {$e_3$}};
	\end{scope}
	\begin{scope}
	\tikzstyle{level 2}=[sibling distance=5em]%
	\tikzstyle{level 3}=[sibling distance=2em]%
		\node at (8,0) [font=\normalsize] {$V'$}
			child{node [dummy] {}
				child{node [dummy,fill=black] {}
					child{
					edge from parent node [swap,very near end] {$c_2$}}
					child{
					edge from parent node [very near end] {$c_1$}}
				edge from parent node [swap] {$c_3$}}
				child[level distance = 2.3em]{node [dummy,fill=black] {}
					child{
					edge from parent node [swap,very near end] {$b_2$}}
					child{
					edge from parent node [very near end] {$b_1$}}
				edge from parent node [swap] {$b_3$}}
				child{node [dummy,fill=black] {}
					child{
					edge from parent node [very near end,swap] {$a_2$}}
					child{
					edge from parent node [very near end] {$a_1$}}				
				edge from parent node {$a_3$}}
			edge from parent node [swap] {$e_3$}};
		\node at (4,1) [font=\large] {$\leq$};
	\end{scope}
	\end{tikzpicture}
\end{equation}
The claim that $\leq$ is indeed a partial order 
(at least if one of $S,T$ is open) is 
\cite[Prop. 7.31]{Per18}.
As an aside, we note that   
$V,V'$ above have a common inner face
$V-\{e_1,e_2\} = V'-\{a_3,b_3,c_3\}$,
which encodes an (universal!) example of a 
Boardman-Vogt relation (see \cite[\S 5.1]{MW07}).
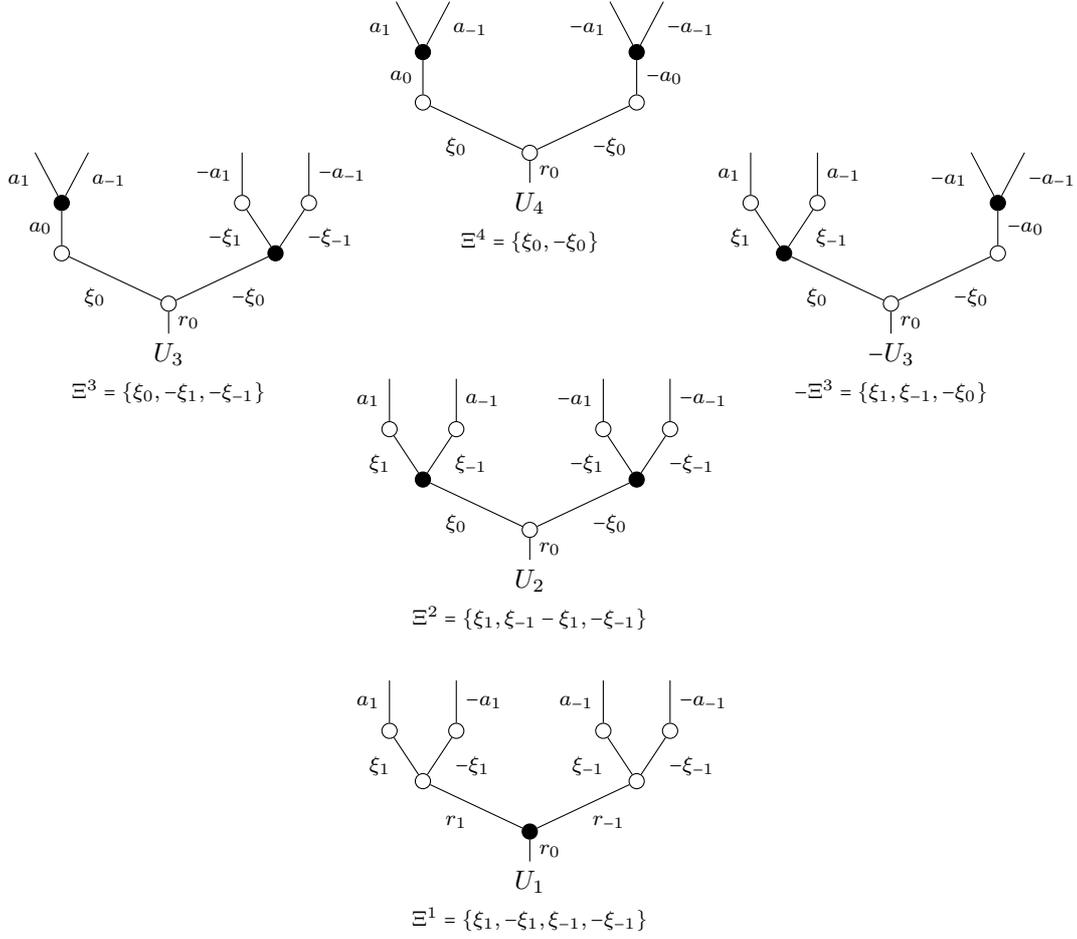
\begin{figure}[ht]
\[%
	\begin{tikzpicture}[auto,grow=up, every node/.style = {font=\footnotesize}]%
	\begin{scope}[level distance = 1.9em]
	\tikzstyle{level 2}=[sibling distance=8em]%
	\tikzstyle{level 3}=[sibling distance=2.5em]%
	\tikzstyle{level 4}=[sibling distance=2em]%
		\node at (0,-0.5) {$\Xi^1=\{\xi_1,-\xi_1,\xi_{-1},-\xi_{-1}\}$};%
		\node at (0,0) [font=\normalsize] {$U_1$}%
			child{node [dummy,fill=black] {}%
				child{node [dummy] {}%
					child{node [dummy] {}
						child{
						edge from parent node[swap]{$-a_{-1}$}}
					edge from parent node [near end,swap] {$-\xi_{-1}$}}%
					child{node [dummy] {}
						child{
						edge from parent node{$\phantom{-}a_{-1}$}}
					edge from parent node [near end]{$\phantom{-}\xi_{-1}$}}%
				edge from parent node [swap] {$r_{-1}$}}%
				child{node [dummy] {}%
					child{node [dummy] {}
						child{
						edge from parent node[swap]{$-a_1$}}
					edge from parent node [near end,swap] {$-\xi_1$}}%
					child{node [dummy] {}
						child{
						edge from parent node{$\phantom{-}a_1$}}
					edge from parent node [near end]{$\phantom{-}\xi_1$}}%
				edge from parent node {$r_{1}$}}%
			edge from parent node [swap] {$r_0$}};
		\node at (0,3.5) {$\Xi^2=\{\xi_1,\xi_{-1}-\xi_1,-\xi_{-1}\}$};%
		\node at (0,4) [font=\normalsize] {$U_2$}%
			child{node [dummy] {}%
				child{node [dummy,fill=black] {}%
					child{node [dummy] {}
						child{
						edge from parent node[swap]{$-a_{-1}$}}
					edge from parent node [near end,swap] {$-\xi_{-1}$}}%
					child{node [dummy] {}
						child{
						edge from parent node{$-a_{1}$}}
					edge from parent node [near end]{$-\xi_{1}$}}%
				edge from parent node [swap] {$-\xi_0$}}%
				child{node [dummy,fill=black] {}%
					child{node [dummy] {}
						child{
						edge from parent node[swap]{$a_{-1}\phantom{-}$}}
					edge from parent node [near end,swap] {$\xi_{-1}\phantom{-}$}}%
					child{node [dummy] {}
						child{
						edge from parent node{$\phantom{-}a_1$}}
					edge from parent node [near end]{$\phantom{-}\xi_1$}}%
				edge from parent node {$\xi_0$}}%
			edge from parent node [swap] {$r_0$}};
		\node at (-4.75,6.5) {$\Xi^3=\{\xi_0,-\xi_1,-\xi_{-1}\}$};%
		\node at (-4.75,7) [font=\normalsize] {$U_3$}%
			child{node [dummy] {}%
				child{node [dummy,fill=black] {}%
					child{node [dummy] {}
						child{
						edge from parent node[swap]{$-a_{-1}$}}
					edge from parent node [near end,swap] {$-\xi_{-1}$}}%
					child{node [dummy] {}
						child{
						edge from parent node{$-a_{1}$}}
					edge from parent node [near end]{$-\xi_{1}$}}%
				edge from parent node [swap] {$-\xi_0$}}%
				child{node [dummy] {}%
					child{node [dummy,fill=black] {}
						child{
						edge from parent node [swap,near end]{$a_{-1}\phantom{-}$}}
						child{
						edge from parent node [near end] {$\phantom{-}a_{1}$}}
					edge from parent node {$a_0$}}%
				edge from parent node {$\xi_0$}}%
			edge from parent node [swap] {$r_0$}};
		\node at (4.75,6.5) {$-\Xi^3=\{\xi_1,\xi_{-1},-\xi_0\}$};%
		\node at (4.75,7) [font=\normalsize] {$-U_3$}%
			child{node [dummy] {}%
				child{node [dummy] {}%
					child{node [dummy,fill=black] {}
						child{
						edge from parent node [swap,near end]{$-a_{-1}$}}
						child{
						edge from parent node [near end] {$-a_{1}$}}
					edge from parent node [swap] {$-a_0$}}%
				edge from parent node [swap] {$-\xi_0$}}%
				child{node [dummy,fill=black] {}%
					child{node [dummy] {}
						child{
						edge from parent node[swap]{$a_{-1}\phantom{-}$}}
					edge from parent node [near end,swap] {$\xi_{-1}\phantom{-}$}}%
					child{node [dummy] {}
						child{
						edge from parent node{$\phantom{-}a_1$}}
					edge from parent node [near end]{$\phantom{-}\xi_1$}}%
				edge from parent node {$\xi_0$}}%
			edge from parent node [swap] {$r_0$}};
		\node at (0,8.5) {$\Xi^4=\{\xi_0,-\xi_0\}$};%
		\node at (0,9) [font=\normalsize] {$U_4$}%
			child{node [dummy] {}%
				child{node [dummy] {}%
					child{node [dummy,fill=black] {}
						child{
						edge from parent node [swap,near end]{$-a_{-1}$}}
						child{
						edge from parent node [near end] {$-a_{1}$}}
					edge from parent node [swap] {$-a_0$}}%
				edge from parent node [swap] {$-\xi_0$}}%
				child{node [dummy] {}%
					child{node [dummy,fill=black] {}
						child{
						edge from parent node [swap,near end]{$a_{-1}\phantom{-}$}}
						child{
						edge from parent node [near end] {$\phantom{-}a_{1}$}}
					edge from parent node {$a_0$}}%
				edge from parent node {$\xi_0$}}%
			edge from parent node [swap] {$r_0$}};
	\end{scope}
	\end{tikzpicture}%
\]%
\caption{The $\mathbb{Z}_{/2}$-poset $\mathsf{Max}(S \otimes T)$ and characteristic edges $\Xi^i$}
\label{FIGURE}
\end{figure}

Returning to the task of proving that \eqref{THM71 EQ} is $G$-inner anodyne, we define $\Xi^U$, 
for each maximal subtree $U \hookrightarrow S \otimes T$,
to be the set of inner edges of $U$ of the form
$(g \xi)_s$ \textit{such that}
the vertex $(g \xi)_s^{\uparrow U} \leq (g \xi)_s$ in $U$
is a $T$-vertex (see Figure \ref{FIGURE}).
We now verify (Ch1),(Ch2),(Ch3).
We recall that, since $S,T$ are assumed open,
\cite[Lemma 7.19]{Per18} guarantees that,
for faces
$S' \hookrightarrow S$, $T' \hookrightarrow T$,
a factorization
$V \hookrightarrow S' \otimes T' \hookrightarrow S \otimes T$
exists iff the edges of $V$ are in 
$\boldsymbol{E}(S') \times \boldsymbol{E}(T')$.

For (Ch1), note first that there is an equivariant grafting decomposition
$T = T_{\not < G\xi} \amalg_{G \xi} T^{\leq G\xi}$, 
where $T_{\not < G\xi}$ contains the edges $t \in T$ such that
$\forall_{g \in G} t \not < g\xi$ (pictorially, this is a lower equivariant outer face of $T$) 
while $T^{\leq G\xi}$
contains the edges $t \in T$ such that
$\exists_{g \in G} t \leq g\xi$ (an upper equivariant outer face of $T$).
But one now readily checks that if
$V \hookrightarrow U$ is an \textit{outer} face such that
$\Xi^U_V = \emptyset$, then either 
$V \hookrightarrow S \otimes T_{\not < G\xi}$ or 
$V \hookrightarrow S \otimes T^{\leq G\xi}$, and thus $V \in A$.

For (Ch3), suppose $U_j \not \geq U_i$, 
$V \hookrightarrow U_i$ and
$(V - \Xi^{U_i}_V) \hookrightarrow U_j$.
Then it follows from \cite[Lemma 7.37]{Per18}
that there exists a generating relation $U_k < U_i$
such that $(V - \Xi^{U_i}_V) \hookrightarrow U_k$
(indeed, \cite[Lemma 7.37]{Per18} makes the claim that such a relation can be performed on an ``elementary subtree'' \cite[Def. 7.16]{Per18} or, in our notation, on the outer closure $\bar{V}^{U_i}$; and, via grafting, it is straightforward to extend such a $<$ relation for $\bar{V}^{U_i}$ to one for $U_i$ \cite[Rem. 7.30]{Per18}). But then, as one sees from \eqref{GENLEXREL EQ},
all edges $e$ of $U_i$ that are not in $U_k$ are topped by the $S$-vertex $e^{\uparrow S}\leq e$, and thus it is $e \not \in \Xi^{U_i}$. Therefore $V \hookrightarrow U_k$, as desired.

Lastly, for (Ch2), suppose $V \hookrightarrow U$ and 
$(V - \Xi^{U}_V) \in A$.
If it were 
$(V-\Xi^{U}_V) \hookrightarrow S \otimes \Lambda^{G \xi}[T]$, then it would also be 
$V \hookrightarrow S \otimes \Lambda^{G \xi}[T]$
since all edges of $\Xi^{U}_V$ have 
$T$-coordinate in $G\xi$.
Now consider the more interesting case
$(V - \Xi^{U}_V) \hookrightarrow S' \otimes T$
for some face $S' \hookrightarrow S$.
Then it will also be 
$V \hookrightarrow S' \otimes T$
unless there is at least one edge
$(g \xi)_s \in \Xi^{U}_V$ such that $s \not \in S'$.
But then since the outer closure $\bar{V}^U$ can have no leaf with $S$-coordinate $s$ (this would contradict $s \not \in S'$), 
by Lemma \ref{CUPCAP LEM} there exists some minimal outer face $U_{(g\xi)_s}^{<s}$ of $U$ with root $(g\xi)_s$ and such that its leaves have $S$-coordinate $<_d s$.
By minimality, one has that $U_{(g\xi)_s}^{<s} \hookrightarrow \bar{V}^U$ and that all inner edges of $U_{(g\xi)_s}^{<s}$ have
$S$-coordinate $s$.
Further, note that $U_{(g\xi)_s}^{<s}$
has at least one inner edge
(since by definition of $\Xi^U$ the vertex 
$(g \xi)_s^{\uparrow U} \leq (g \xi)_s$ is a $T$-vertex)
and that $V$ contains none of those inner edges (or else it would be $s \in S'$).
Thus by applying \cite[Lemma 7.34]{Per18}
to $U_{(g\xi)_s}^{<s}$ one obtains a maximal subtree $U' < U$
containing all edges of $U$ that are not inner edges of $U_{(g\xi)_s}^{<s}$.
But then $V \hookrightarrow U'$ and (Ch2) follows.
\end{example}

\begin{remark}
We briefly outline how to modify the example above to prove 
\cite[Thm 7.1(ii)]{Per18}, in which case some notable subtleties arise.
The result again states that \eqref{THM71 EQ} is $G$-inner anodyne, but now with one of $S,T$ allowed to have stumps while the other is required to be linear, i.e. of the form $G/H \cdot [n]$.

One again sets $I=\mathsf{Max}(S\otimes T)$, with maximal trees defined just as before, but some caution is needed.
To see why, note that if the black nodes $\bullet$ in \eqref{GENLEXREL EQ} are replaced with stumps then $V'$
becomes a subtree of $V$, so that not all maximal trees are maximal with regard to inclusion.

When $S$ has stumps and $T$ is linear this causes no issues and the proof above holds
(notably, it can now be 
$\Xi^U = \emptyset$, in which case (Ch1) demands $U \in A$,
as indeed follows from the argument).

However, when $S$ is linear and $T$ has stumps the proof above breaks down (more specifically, the tree $U_{(g\xi)_s}^{<s}$ that appears when arguing (Ch2) may now fail to have inner edges). The solution is then to \textit{reverse} the poset structure on 
$\mathsf{Max}(S\otimes T)$
and to modify the $\Xi^U$ to be those inner edges $(g \xi)_s$ such that
$(g \xi)_s \in t_s^{\uparrow T}$ for some $t_s$
(pictorially, this says that these are the lowermost edges with $T$-coordinate in $G\xi$, whereas before they were the uppermost ones). The arguments for (Ch1),(Ch3) then hold.
For (Ch2), only the argument for the interesting case of
$V- \Xi_V^U \hookrightarrow S' \otimes T$, $s \not \in S'$
changes. In this case, there is then a $\leq_d$-largest edge $t'_s$ such that $(g \xi)_s <_d t'_s$, where $s$ can not be the root of $S$ (or else it would be $s \in S'$). Pictorially, $t'_s$ looks like the edge $e_1 \in V$ in \eqref{GENLEXREL EQ} in the case where the $\bullet$ node is unary (since $S$ is assumed linear). But then since $V$ can not contain $t'_s$ there exists a maximal subtree $U' > U$ such that $V \hookrightarrow U'$,
and (Ch2) follows.

Lastly, we note that \cite[Thm. 7.2]{Per18} follows from a minor variant of the argument for \cite[Thm. 7.1(ii)]{Per18} when $S$ is linear.
\end{remark}

\subsection{Segal core, horn and orbital horn inclusions}\label{HYPERSAT SEC}

\begin{proposition}\label{ORB_HORN_PROP}
	Let $T \in \Omega_G$.
	For $G$-subsets  
	$\emptyset \neq F \subseteq E \subseteq \boldsymbol{E}^{\mathsf{i}}(T)$
	the inclusions
\begin{equation}\label{ORBHORNINC EQ}
	\Lambda_o^{E}[T] \to \Omega[T],\qquad
	\Lambda_o^{E}[T] \to \Lambda_o^{F}[T]
\end{equation}
	are $G$-inner anodyne.
\end{proposition}

\begin{proof}
We are free to assume that $T \in \Omega^G \subseteq \Omega_G$. Indeed, otherwise writing $T = G \cdot_H T_{\**}$, where $T_{\**} \in \Omega^H$ is a fixed component and 
$E_{\**}=E \cap \boldsymbol{E}^{\mathsf{i}}(T_{\**})$, $F_{\**} = F \cap \boldsymbol{E}^{\mathsf{i}}(T_{\**})$,
the maps in \eqref{ORBHORNINC EQ} are
$G \cdot_H 
\left( \Lambda_o^{E_{\**}}[T_{\**}] \to \Omega[T_{\**}]\right)$
and
$G \cdot_H 
\left( \Lambda_o^{E_{\**}}[T_{\**}] \to \Lambda_o^{F_{\**}}[T_{\**}]\right)$.

In the $\Lambda_o^{E}[T] \to \Omega[T]$ case we apply
Lemma \ref{CHAREDGE LEM} with $I = \{\**\}$ a singleton and
\[
	\Xi^{\**} = E, \qquad 
	U_{\**} = T,\qquad
	A=\Lambda_o^{E}[T].
\]
It remains to check the characteristic conditions in Definition \ref{CHAREDGE DEF}.
(Ch0) and (Ch3) are clear.

Note that for $V\hookrightarrow T$ it is $V \not \in A$ iff 
$GV = T-E'$ for some $G$-subset
$E' \subseteq E$.

For (Ch1), the condition $\Xi_{V} = \emptyset$
says that none of the inner edges of $V$ are in $E$
and thus, by Remark \ref{GOUT REM},
that the orbital outer face $G V$ contains none of the edge orbits in $E$ as inner edge orbits. Since $E \neq \emptyset$, the orbital outer face $GV$ is not $T$ itself, 
and hence $A=\Lambda_o^{E}[T]$ contains $V$.

For (Ch2), note that if $V \not \in A$, i.e., 
$GV = T - E'$ for $E' \subseteq E$, then Remark \ref{GINNER REM} implies that
$G(V-\Xi_V) = T - E''$ for $E'\subseteq E'' \subseteq E$,
and thus also $(V-\Xi_V) \not \in A$.

In the $\Lambda_o^{E}[T] \to \Lambda_o^{F}[T]$ case 
we instead apply Lemma \ref{CHAREDGE LEM} with $I = (E \setminus F)/G$, with an 
\textit{arbitrary} choice of total order, and 
(writing elements of $(E \setminus F)/G$ as orbits $Ge \subseteq E \setminus F$)
\[
	\Xi^{Ge} = F, \qquad 
	U_{G e}= T - Ge, \qquad
	A=\Lambda_o^{E}[T].
\]
Note that the $U_{Ge}$ are the orbital inner faces $T - Ge$ for $Ge \subseteq E \setminus F$, and thus the map
\eqref{CHARLEM EQ}
in Lemma \ref{CHAREDGE LEM} is indeed $\Lambda_o^{E}[T] \to \Lambda_o^{F}[T]$.
Further, we are free to abbreviate $\Xi = \Xi^{Ge}$ and $\Xi_V = \Xi^{Ge}_V$, since
$\Xi^{Ge}$ is independent of $G e$.
We again check the characteristic conditions. (Ch0) is clear.

For (Ch1), note that for an outer face 
$V \hookrightarrow U_i$, and writing $\bar{V} = \bar{V}^T$,
Lemma \ref{INNINT LEM} implies
$\boldsymbol{E}^{\mathsf{i}}(V) = 
\boldsymbol{E}^{\mathsf{i}}(U_i) \cap \boldsymbol{E}^{\mathsf{i}}(\bar{V})$
and hence since $\Xi_{U_i} = F = \Xi$ the 
hypothesis $\Xi_{V} = \emptyset$ in (Ch1) implies it is also
$\Xi_{\bar{V}} = \emptyset$.
Hence just as before $G\bar{V}$ is an orbital outer face other than $T$, hence $V$ is in $A=\Lambda_o^{E}[T]$.
The argument for (Ch2) is identical to the one in the
$\Lambda_o^{E}[T] \to \Omega[T]$ case.
Lastly, (Ch3) follows since	if $V \not \in A$, so that
$GV = T - E'$ and $G(V - \Xi_V) = T-E'-F'$ with
$E' \subseteq E$, $F' \subseteq F$,
then $GV \hookrightarrow T-Ge$ iff $G(V - \Xi_V) \hookrightarrow T-Ge$
(by fullness of $T-Ge$; see Remark \ref{INNFULLORB REM})
and thus $V \hookrightarrow T-Ge$ iff $V - \Xi_V \hookrightarrow T-Ge$.
\end{proof}

\begin{example}
Keeping the setup in Example \ref{HORNEX EX}, we consider the inclusion $\Lambda_o^{Gb}[T] \to \Omega[T]$.
Denoting the leftmost tree in the middle row of \eqref{HORN_EX_FIG} by $S$, the intersection of the 
$G$-poset in \eqref{HORN_EX_FIG} with the $G$-poset 
$\mathsf{Face}^{lex}_{\Xi}(U)=
\mathsf{Face}^{lex}_{Gb}(T)$
specified by the proofs of Proposition \ref{ORB_HORN_PROP} and Lemma \ref{CHAREDGE LEM} is the $G$-poset
$S \to T \leftarrow -S$.
The argument in those proofs then shows that $\Lambda_o^{Gb}[T] \to \Omega[T]$ is built cellularly by attaching
$G \cdot \left(\Lambda^b[S] \to \Omega[S]\right)$
followed by $\Lambda^{Gb}[T] \to \Omega[T]$.
\end{example}

\begin{proposition}\label{REG_HORN_PROP}
Let $T \in \Omega_G$.
For $G$-subsets 
$\emptyset \neq F \subseteq E \subseteq 
\boldsymbol{E}^{\mathsf{i}}(T)$
the inclusions
\begin{equation}
	\Lambda^{E}[T] \to \Lambda^{F}[T]
\end{equation}
are $G$-inner anodyne.
\end{proposition}

\begin{proof}
Arguing just as at the beginning of Proposition \ref{ORB_HORN_PROP}, we may assume $T \in \Omega^G \subseteq \Omega_{G}$.

We now apply Lemma \ref{CHAREDGE LEM} with 
$I = \mathcal{P}_0(E\setminus F)$
the poset of non-empty subsets $\emptyset \neq E' \subseteq (E \setminus F)$, ordered by \textit{reverse inclusion}, and
\[
	\Xi^{E'} = F, \qquad
	U_{E'} = T - E', \qquad
	A=\Lambda^{E}[T].
\]
We again need to verify the characteristic conditions,
and as in the previous result we abbreviate
$\Xi = \Xi^{E'}$, $\Xi_V = \Xi^{E'}_V$.
(Ch0) is clear. (Ch1) follows as in the previous proof, except noting that $\bar{V}$ (rather than $G \bar{V}$) is not $T$.
(Ch2) follows since $V \in A$ iff $V-\Xi_V \in A$.
Similarly,
(Ch3) follows since 
$V \hookrightarrow T-E'$ iff $(V-\Xi_V) \hookrightarrow T-E'$
and since if
$V \hookrightarrow T-E',V \hookrightarrow T-E''$
then 
$V \hookrightarrow T-(E' \cup E'')$.
\end{proof}

\begin{remark}\label{TWOPROOF REM}
	By specifying to the non-equivariant case $G = \**$
	the previous results yield two distinct proofs
	that inclusions of non-equivariant horns
	$\Lambda^{E}[T] \to \Lambda^{F}[T]$
	are inner anodyne,
	with the first proof using $I = E \setminus F$ (with an arbitrary total order) and the second using 
	$I = \mathcal{P}_0(E \setminus F)$. 

	The discrepancy is explained as follows: 
	when $T$, $E$, $F$ are $G$-equivariant, showing that
	$\Lambda^{E}[T] \to \Lambda^{F}[T]$ 
	is $G$-inner anodyne requires a control of isotropies 
	not needed when showing that the underlying map is non-equivariant inner anodyne, and since this control is given by (Ch3), it is necessary to include in the $\{U_i\}$ the
	``intersections'' of $T-e$ and $T-ge$ for $e \in E \setminus F$. 
\end{remark}

The following repeats Remark \ref{RECOVER REM}(i),
but the alternative proof will simplify the proof of Proposition \ref{HYPER PROP}. 

\begin{proposition}\label{SCANOD PROP}
Let $T \in \Omega_G$. The inclusions $Sc[T] \to \Omega[T]$ are $G$-inner anodyne.
\end{proposition}

\begin{proof}
We may yet again assume $T \in \Omega^G \subseteq \Omega_G$. 
Now apply Lemma \ref{CHAREDGE LEM} with 
$I=\**$, $U_{\**} = T$, $\Xi^{\**} = \boldsymbol{E}^{\mathsf{i}}(T)$.
The conditions (Ch0),(Ch1),(Ch2),(Ch3) follow as in 
Remark \ref{RECOVER REM}(i).
\end{proof}

\begin{remark}\label{FACCES REM}
All $G$-inner horn inclusions attached in the proof of the characteristic edge lemma, Lemma \ref{CHAREDGE LEM}, correspond to $G$-trees whose non-equivariant components are faces of the $U_i$. Moreover, when $I \simeq G/H$ has a transitive $G$-action, the last horn inclusion attached (corresponding to the maximum of $\mathsf{Face}^{lex}_{\Xi}(U)$) is
$G \cdot_H \left( \Lambda^{\Xi}[U] \to \Omega[U] \right)$, while all other horn inclusions used are of the form
$G \cdot_K (\Lambda^{\Xi_W}[W] \to \Omega[W])$ for some $W$ with $|W| < |U|$.
\end{remark}

\begin{corollary}\label{REGGENHORN COR}
$G$-inner horn inclusions
$\Lambda^{E}[T] \to \Omega[T]$
are built cellularly from generating horn inclusions
$\Lambda^{Ge}[S] \to \Omega[S]$.
\end{corollary}

\begin{proof}
	The proof is by induction on $|T_{\**}|$ for $T_{\**}\in \Omega$ a tree component (cf. Remark \ref{DEGREE REM}). As in the proofs of Propositions \ref{ORB_HORN_PROP} and \ref{REG_HORN_PROP} one is free to assume $T \in \Omega^G$.	
	A choice of edge orbit $Ge$ in $E$ yields a factorization
	$\Lambda^{E}[T] \to \Lambda^{Ge}[T] \to \Omega[T]$,
	hence we need only show that 
	$\Lambda^{E}[T] \to \Lambda^{Ge}[T]$ is built cellularly from generating horns. 
	But this is immediate from the induction hypothesis,
	Remark \ref{FACCES REM}, and the proof of Proposition \ref{REG_HORN_PROP}, since all
	$U_{i}$ therein satisfy $|U_i|<|T|$.
\end{proof}

Following the discussion preceding \cite[Prop. 3.6.8]{HHM16},
a class of normal monomorphisms of $\mathsf{dSet}^G$
(or, more generally, a subclass of the cofibrations in a model category) is called
\textit{hypersaturated} if it is saturated (i.e. it is closed under
pushouts, transfinite composition and retracts) and satisfies the following additional cancellation property: 
if $f,g$ are normal monomorphisms (or, more generally, cofibrations)
\begin{equation}\label{CANCEL_EQ}
A \xrightarrow{f} B \xrightarrow{g} C
\end{equation}
such that both $f$ and $gf$ are in the class, then so is $g$.

The following is an equivariant generalization of 
\cite[Props. 2.4 and 2.5]{CM13a}.

\begin{proposition}\label{HYPER PROP}
The following sets of maps generate the same hypersaturated class:
\begin{itemize}
\item the generating $G$-inner horn inclusions
$\Lambda^{Ge} [T] \to \Omega[T]$ for $T \in \Omega_G$ and 
$e \in \boldsymbol{E}^{\mathsf{i}}(T)$; 
\item the $G$-inner horn inclusions
$\Lambda^{E} [T] \to \Omega[T]$ for $T \in \Omega_G$ and 
$G$-subset $\emptyset \neq E \subseteq \boldsymbol{E}^{\mathsf{i}}(T)$; 
\item the orbital $G$-inner horn inclusions
$\Lambda^{E}_o [T] \to \Omega[T]$ for $T \in \Omega_G$ and 
$G$-subset $\emptyset \neq E \subseteq \boldsymbol{E}^{\mathsf{i}}(T)$; 
\item the $G$-Segal core inclusions
$Sc[T] \to \Omega[T]$ for $T \in \Omega_G$.
\end{itemize}
\end{proposition}

In the following proof we refer to the hypersaturation
of the orbital horn (resp. Segal core) inclusions as the orbital (resp. Segal) hypersaturation.

\begin{proof}
	That the first two hypersaturations coincide is clear from Corollary \ref{REGGENHORN COR}.

	The fact that $G$-inner horn inclusions generate the maps in the orbital and Segal hypersaturations has been established in Proposition \ref{ORB_HORN_PROP} and Remark \ref{RECOVER REM}(i), Proposition \ref{SCANOD PROP}.
	
	To see that the $G$-inner horn inclusions are in the orbital hypersaturation, we again argue by induction on $|T_{\**}|$, with the base cases those where $\Lambda^{E} [T]=\Lambda^{E}_o [T]$.
	Recalling that in the proof of Proposition \ref{ORB_HORN_PROP}
	one sets
	$I=\**$, $U_{\**} = T$ and $\Xi^{\**} = E$,
	Remark \ref{FACCES REM} implies that in the factorization
	$\Lambda_o^E[T] \to \Lambda^E[T] \to \Omega[T]$
	the first map $\Lambda_o^E[T] \to \Lambda^E[T]$ is built cellularly out of $G$-horns with $|S_{\**}| < |T_{\**}|$.
	But then the induction hypothesis says that 
	$\Lambda_o^E[T] \to \Lambda^E[T]$ is in the orbital hypersaturation, and by the cancellation property \eqref{CANCEL_EQ} so is $\Lambda^E[T] \to \Omega[T]$.

	For the claim that the $G$-inner horn inclusions are in the Segal hypersaturation, note that Proposition \ref{SCANOD PROP} sets
	$I=\**$, $U_{\**} = T$, $\Xi^{\**} = \boldsymbol{E}^{\mathsf{i}}(T)$.
	Therefore, arguing exactly as above for the factorization
	$Sc[T] \to \Lambda^{\boldsymbol{E}^{\mathsf{i}}(T)}[T] \to \Omega[T]$,
	one obtains by induction on $|T_{\**}|$ that
	$\Lambda^{\boldsymbol{E}^{\mathsf{i}}(T)}[T] \to \Omega[T]$
	is in the Segal hypersaturation. 
	But now letting $\emptyset \neq E \subseteq \boldsymbol{E}^{\mathsf{i}}(T)$ be any $G$-subset and considering the factorization
	$\Lambda^{\boldsymbol{E}^{\mathsf{i}}(T)}[T] \to 
	\Lambda^{E}[T] \to
	\Omega[T]$ the induction hypothesis applies to the cells of
	$\Lambda^{\boldsymbol{E}^{\mathsf{i}}(T)}[T] \to \Lambda^{E}[T]$
	(just as in Corollary \ref{REGGENHORN COR}),
	which is thus also in the Segal hypersaturation.
	But by the cancellation property, so is 
	$\Lambda^{E}[T] \to \Omega[T]$, finishing the proof.
\end{proof}

\begin{remark}
      \label{ORBHYPER_REM}
	The identification between orbital subcomplexes 
	$\bigcup_i \Omega[S_i] \subseteq \Omega[T]$ and subcomplexes of 
	$\bigcup_i \Omega[S_i/G] \subseteq \Omega[T/G]$
	described in Remark \ref{ORB_HORN_REM}
	is compatible with attaching horn inclusions.
	As such, non-equivariant results concerning horns in $\Omega$
	imply the analogue results for orbital horns in $\Omega_G$. For example, mimicking \cite[Lemma 5.1]{MW09}, one has pushouts	
\begin{equation}\label{ORBHORNPUSH EQ}
\begin{tikzcd}
	\Lambda^{E - G e}_o[T - Ge] \arrow[r] \arrow[d]
	\arrow[dr, phantom, "\ulcorner", very near start] &
	\Lambda^{E}_o[T] \arrow[d]
\\
	\Omega[T - G e] \arrow[r]&
	\Lambda^{E - G e}_o[T]
\end{tikzcd}
\end{equation}	
which imply the orbital horn analogue of 
Corollary \ref{REGGENHORN COR}.
It is worth noting that while setting $G=\**$ in 
Corollary \ref{REGGENHORN COR} does recover
\cite[Lemma 5.1]{MW09}, 
the analogue of the pushouts \eqref{ORBHORNPUSH EQ}
does not hold for (non-orbital) $G$-inner horns,
so that the proof of Corollary \ref{REGGENHORN COR}
(see also the original proof in \cite[Prop. 6.17]{Per18})
is intrinsically harder when $G \neq \**$,
due to isotropy concerns (this is closely related to the two proofs discussed in Remark \ref{TWOPROOF REM}).

Similarly, \cite[Props. 2.4 and 2.5]{CM13a} 
(or Remark \ref{RECOVER REM}(i) and Proposition \ref{HYPER PROP} when $G=\**$)
imply that the Segal core inclusions 
$Sc[T] \to \Omega[T]$
are built cellularly from the orbital horn inclusions
$\Lambda^E_o[S] \to \Omega[S]$, and that the two classes have the same hypersaturation. 

	We note that this last observation indicates an alternate route for proving Proposition \ref{HYPER PROP}
	(which the authors considered in early versions of this work)
	without making direct use of the characteristic edge lemma machinery.
	Namely, following the considerations above, the main missing claim is the first part of Proposition \ref{ORB_HORN_PROP}, stating that the inclusions
	$\Lambda^E_o[T] \to \Omega[T]$ are $G$-inner anodyne, and this latter claim is not too hard to prove directly. 
	Indeed, while the proof does require some of the ideas in the proof of Lemma \ref{CHAREDGE LEM},
	many of the subtler arguments in that proof
	become trivial when $I=\**$ is a singleton, as is the case in Proposition \ref{ORB_HORN_PROP}.
\end{remark}

\begin{remark}\label{DUMBHYPER REM}
Since for each $T \in \Omega_G$ the cellular decomposition of 
$\partial \Omega[T]$ is built by attaching boundary inclusions 
$\partial \Omega[S] \to \Omega[S]$ with $|S_{\**}|<|T_{\**}|$, it readily follows that the sets
\[
\{\partial \Omega[T] \to \Omega[T] \colon T \in \Omega_G\}
\qquad
\{\emptyset \to \Omega[T] \colon T \in \Omega_G\}
\]
have the same hypersaturation. Similarly, the following sets also have the same hypersaturation.
\[
\{\partial \Omega[T] \to \Omega[T] \colon T \in \Omega_G,T \neq G/H \cdot \eta\}
\qquad
\{ \coprod_{e \in \boldsymbol{E}(T)} \Omega[\eta] \to \Omega[T] \colon T \in \Omega_G\} 
\]
\end{remark}

We end this section with some necessary remarks about hypersaturations of \textit{simplicial} horns.

\begin{remark}\label{SLICE REM}
	Setting $G=e$ and restricting to the overcategory
	$\mathsf{dSet}_{/\eta} \simeq \mathsf{sSet}$ 
	(where as usual we abbreviate $\Omega[\eta]$ as just $\eta$; see Notation \ref{DELTAOMEGA NOT}),
	Proposition \ref{HYPER PROP}
	recovers the well known claim that 
	the hypersaturation of the simplicial inner horn inclusions
	$\{\Lambda^i[m] \to \Delta[m] \colon 0< i < m\}$
	coincides with the hypersaturation of the simplicial Segal core inclusions
	$\{Sc[m] \to \Delta[m]\colon m \geq 0\}$.
\end{remark}

\begin{remark}\label{HYPERSATKAN REM}
	We will use of a variant of the previous remark for the hypersaturation of \textit{all} simplicial horns.
	Namely, we claim that the hypersaturation of all simplicial horns 
	$\{\Lambda^i[m] \to \Delta[m] \colon 0 \leq i \leq m,0<m\}$
	matches the hypersaturation of all vertex inclusion maps
	$\{\Delta[0] \to \Delta[m]\}$.
	
	Call the latter hypersaturation $S$. 
	One easily checks that the maps $\{0\} \to Sc[m]$ are in $S$, so that by cancellation so are the maps 
	$Sc[m] \to \Delta[m]$ and hence by Remark \ref{SLICE REM} so are all inner horn inclusions. 
	Moreover, for left horns $\Lambda^0[m]$
	the maps $\{0\} \to \Lambda^0[m]$ are built cellularly
	from left horn inclusions
	$\Lambda^0[k] \to \Delta[k]$ with $k<m$
	(in join notation (see \cite[\S 1.2.8]{Lur09} or \cite[\S 7.4]{Per18}),
	$\{0\} \to \Lambda^0[m]$ is 
	$\Delta[0] \star (\emptyset \to \partial \Delta[m-1])$, and the filtration follows from the cellular filtration of $\partial \Delta[m-1]$).
	But hence by induction and the cancellation property all left horn inclusions 
	$\Lambda^0[m] \to \Delta[m]$
are in $S$. The case of right horn inclusions $\Lambda^m[m] \to \Delta[m]$ is dual.
\end{remark}

\begin{remark}\label{ANHYPER REM}
The smallest hypersaturated class containing the inner horn inclusions and the left horn inclusion
$\Lambda^0[2] \to \Delta[2]$ 
in fact contains all left horn inclusions
$\Lambda^0[m] \to \Delta[m]$ 
for $m \geq 2$.
Indeed, this follows inductively from the left diagram
below since the bottom map is inner while the top and left maps are given by the center and right pushout diagrams.
\begin{equation}
\begin{tikzcd}
	\Lambda^{0,1}[m] \ar{r} \ar{d} &
	\Lambda^{0}[m] \ar{d}
& &
	\Lambda^{0}[m-1] \ar{r} \ar{d} \arrow[dr, phantom, "\ulcorner", very near start]&
	\Lambda^{0,1}[m] \ar{d}
&
	\Lambda^{0}[m-1] \ar{r} \ar{d} \arrow[dr, phantom, "\ulcorner", very near start]&
	\Lambda^{0,1}[m] \ar{d}
\\
	\Lambda^{1}[m] \ar{r} & \Delta[m] 
& &
	\Delta[m-1] \ar{r}[swap]{d^1} &
	\Lambda^0[m] 
&
	\Delta[m-1] \ar{r}[swap]{d^0} &
	\Lambda^1[m] 
\end{tikzcd}
\end{equation}
The case of right horn inclusions is dual.
\end{remark}

\begin{remark}\label{CONTGR REM}
Write 
$\widetilde{[m]} = 
(0 \rightleftarrows 1 
\rightleftarrows \cdots 
\rightleftarrows m)$
for the contractible groupoid on objects $0,1,\cdots,m$. Note that the $k$-simplices of $\widetilde{[m]}$
are encoded as strings $a_0 a_1 \cdots a_k$
with $a_{i} \in \{0,1,\cdots,m\}$, and that a simplex is non-degenerate iff $a_{i-1}\not = a_{i}, 1 \leq i \leq k$.
We claim that the maps
\begin{equation}\label{INVER EQ}
	\Delta[m] = N [m] \xrightarrow{012\cdots m} N \widetilde{[m]},\quad m \geq 1
\end{equation}
are built cellularly out of left horn inclusions $\Lambda^{0}[k] \to \Delta[k]$ with $k\geq 2$.

Indeed, we show a little more. We say a subcomplex 
$A \subseteq N \widetilde{[m]}$ is \textit{$0$-stable}
if a $k$-simplex $\underline{a}$ is in $A$ iff the $(k+1)$-simplex $0\underline{a}$ is.
We claim that any inclusion $A \to A'$ of $0$-stable subcomplexes is built cellularly from left horn inclusions $\Lambda^{0}[k] \to \Delta[k]$ with $k\geq 1$.
Indeed, it suffices to check this when $A'$ attaches as little as 
possible to $A$, and $0$-stability guarantees that in that case the only two non-degenerate simplices in $A \setminus A'$
have the form 
$\underline{a}$ and $0\underline{a}$
(note that $\underline{a}$ can not start with a $0$).
But then $A\to A'$ is a pushout of 
$\Lambda^{0}[k+1] \to \Delta[k+1]$ where $k$ is the dimension of $\underline{a}$.

The desired claim follows by noting that both the domain and codomain of \eqref{INVER EQ} are $0$-stable and that the  inclusions 
$\Lambda^0[1] \to \Delta[1]$ are unneeded since \eqref{INVER EQ} is an isomorphism on $0$-simplices.
\end{remark}

\subsection{Genuine equivariant operads}\label{GENEQOP SEC}

Recall that categories can be identified with their nerves, since the \textit{nerve functor}
$N \colon \mathsf{Cat} \to \mathsf{sSet}$, 
given by $N\mathcal C (n) = \mathsf{Cat}([n], \mathcal C)$, 
is fully faithful.
Moreover, the essential image of the nerve is characterized as those simplicial sets
with the strict right lifting property 
against the inner horn inclusions \cite[Prop. 1.1.2.2]{Lur09}
(here \textit{strict} means that the usual lifts are unique).

More generally, 
one has a similar operadic story. Any tree $U \in \Omega$ has a naturally associated colored operad 
$\Omega(U) \in \mathsf{Op}$ \cite[\S 3]{MW07}, and
\cite[Prop. 5.3 and Thm. 6.1]{MW09} show that the operadic nerve
$N \colon \mathsf{Op} \to \mathsf{dSet}$, 
given by $N\mathcal{O} (U) = \mathsf{Op}(\Omega(U), \mathcal {O})$,
is again fully faithful
with essential image the dendroidal sets with the strict right lifting property against dendroidal inner horn inclusions.
Moreover, \cite[Cor. 2.6]{CM13a} provides an alternative characterization via strict lifts against Segal core inclusions. The equivalence between these two characterizations is an observation concerning the notion of hypersaturation discussed in the previous section, as follows.


In the next result, note that we need not assume that the maps in
\eqref{CANCEL_EQ} are cofibrations. More precisely,
we slightly modify the notion of ``hypersaturation'' by using the cancellation property 
``if $f$ and $gf$ are in the class, then so is $g$'' without further  requirements on $f,g$. Alternatively, this means that normal monomorphisms/cofibrations are replaced with the class of all maps.

\begin{proposition}\label{HYPERLP PROP}
      If two classes $\mathcal{C},\mathcal{D}$
      of maps in a category
      have the same hypersaturation, then
      the two classes of maps with the strict right lifting property against $\mathcal C$ and $\mathcal D$ coincide.
\end{proposition}

\begin{proof}
	It suffices to check that the hypersaturation closure conditions are compatible with strict right lifting properties.
The claims concerning pushouts, transfinite compositions and retracts follow from the easy observation that the proofs of the analogue claims for the usual right lifting property \cite[Lemma 11.1.4]{Ri14} are compatible with the uniqueness requirement.


We thus address only the cancellation property \eqref{CANCEL_EQ}. Suppose that $r$ has the strict right lifting property against $f$ and $gf$, and consider a lifting problem as on the left below.
\begin{equation}
\begin{tikzcd}[column sep = 8em]
	B \ar{r}{p} \ar{dd}[swap]{g}& 
	X \ar{dd}{r}
&
	A \ar{r} \ar{d}[swap]{f} &
	X \ar{d}{r}
\\
&&
	|[alias=B]|
	B \ar{d}[swap]{g}
	\ar{ru}{p} &
	|[alias=Y]|
	Y \ar[equal]{d}
\\
	C \ar{r}[swap]{q} \ar[dashed]{ruu} & Y
&
	C \ar{r}[swap]{q}
	\ar[dashed]{ruu}[swap, near start]{\exists! H}&
	Y
\arrow[from=B, to=Y,crossing over]
\end{tikzcd}
\end{equation}
By assumption, there is a unique lift $H$ for the outer square on the right, and we claim that $H$ is also the unique lift for the left square.
Noting that
$pf = Hgf$
and 
$rp = qg =  r H g$
it follows that both $p$ and $Hg$ are lifts for the top square in the right diagram, so that by the uniqueness assumption it is
$p = Hg$. This shows that $H$ is also in fact a lift for the left square.
Uniqueness follows since any lift of the left square induces a lift of the outer right square.
\end{proof}
Roughly speaking, our goal in this section is that of describing those presheaves with the 
strict right lifting property against any of the classes of maps in Proposition \ref{HYPER PROP},
which we call genuine equivariant operads.
However, some care is needed.
Namely, it is essential to work with the category
$\mathsf{dSet}_G = \mathsf{Set}^{\Omega_G^{op}}$
of \textit{genuine $G$-dendroidal sets}
rather than with the category
$\mathsf{dSet}^G = \mathsf{Set}^{\Omega^{op} \times G}$
of $G$-dendroidal sets, i.e. it is essential to work with presheaves that are evaluated on 
$G$-trees $T \in \Omega_G$
rather than non-equivariant trees $U \in \Omega$
(the motivation for this is given in 
Remark \ref{DESTIAL REM} below).
To relate these presheaf categories, note that the fully faithful inclusion
$\upsilon \colon \Omega \times G^{op} \to \Omega_G$
given by $U \mapsto G \cdot U$ induces an adjunction
(note that $\upsilon_{\**}$ is itself fully faithful, being a right Kan extension along a fully faithful functor \cite[Cor. 1.4.5]{Ri14})
\begin{equation}\label{DSETG_EQ}
	\upsilon^{\**}: \dSet_G \rightleftarrows \dSet^G : \upsilon_{\**}
\end{equation}
Explicitly, one has 
$\upsilon_{\**}X(T) \simeq X(T_{\**})^H \simeq
\mathsf{dSet}^G(\Omega[T],X)$,
where $T \simeq G \cdot_H T_{\**}$ for 
$T_{\**} \in \Omega^H$,
and the $H$-action on 
$X(T_{\**})$ is defined diagonally, i.e. by the composites
$X(T_{\**}) \xrightarrow{X(h^{-1})}
X(T_{\**}) \xrightarrow{h} X(T_{\**})$.

\begin{remark}\label{UPSPUSHMON REM}
$\upsilon_{\**}$ does not preserve arbitrary colimits, due to the presence of fixed points in its formula. Nonetheless, $\upsilon_{\**}$ does preserve pushouts where one leg is a monomorphism.
\end{remark}

\begin{remark}\label{TWOYON REM}
Mimicking the notation in \S \ref{EQDENDSETS SEC},
we write $\Omega_G[-] \colon \Omega_G \to \mathsf{dSet}_G$
for the Yoneda embedding.
On the other hand, in 
\S \ref{EQDENDSETS SEC} we extended the notation $\Omega[-]$
to obtain a functor $\Omega[-] \colon \Omega_G \to \mathsf{dSet}^G$.
These two ``representable functors'' are related by 
$\Omega_G[T] \simeq \upsilon_{\**} \Omega[T]$.
\end{remark}

The following definition is then the main purpose of this section.

\begin{definition}\label{GEN_OP_DEF}
	$Z \in \dSet_G$ is called a \textit{genuine equivariant operad} if
	$Z$ has the strict right lifting property against the images under $\upsilon_{\**}$
	of the Segal core inclusions, i.e. against the maps
	\begin{equation}\label{GGIOP_EQ}
		\upsilon_{\**}\left(Sc[T] \to \Omega[T] \right),
		\qquad
		T \in \Omega_G.
	\end{equation}
Equivalently, by Propositions \ref{HYPER PROP} and \ref{HYPERLP PROP}, one may replace Segal core inclusions with either orbital $G$-inner horn inclusions or $G$-inner horn inclusions.
\end{definition}

\begin{example}\label{STRICTLIFT EX}
	To illustrate the role of the strict lifting condition against the maps in \eqref{GGIOP_EQ},
	consider the $G$-tree $T$ in \eqref{TWOREP EQ},
	along with the subgroup $K = \langle -1 \rangle$ therein
	and the orbital faces $R_1,R_2,S$ in Example \ref{ORBFACE EX}. 
	The strict lifting condition then says that the left map in	
\begin{equation}\label{NORM_COMP_GEN_EQ}
\begin{tikzcd}
	Z(R_1) \times_{Z(G/K \cdot \eta)} Z(R_2)  &
	Z(T) \ar{l}[swap]{\simeq} \ar{r} &
	Z(S)&
\end{tikzcd}
\end{equation}
is an isomorphism, so that $T$ induces a 
\textit{composition map}
$
	Z(R_1) \times_{Z(G/K \cdot \eta)} Z(R_2)  \to
	Z(S)
$.
Here we note that $R_1,R_2,S$ are \textit{$G$-corollas}
(i.e. $G$-trees with a single $G$-vertex).
Informally, one then thinks of the $Z(C)$, where
$C$ ranges over the $G$-corollas,
as the mapping sets of the genuine equivariant operad $Z$,
so that the strict lifting conditions equip these mapping sets with associative and unital composition maps.

We caution, however, that this is not quite the whole story,
since the composition maps need also be compatible with the presheaf structure, which is more complex in the equivariant context. 
More explicitly, non-equivariantly one needs only compatibility with the symmetric group actions,
reflecting the fact that all (non-degenerate) maps between corollas are symmetry isomorphisms. 
But in the equivariant context $G$-corollas are also related via \textit{quotient maps}
(such as the map in \eqref{QUOTMAP EQ}),
which induce subtler compatibility conditions.
Nonetheless, our intended application in \S \ref{EDSS_SEC}
will not require an explicit discussion of these additional compatibilities.
\end{example}

\begin{remark}\label{NORMMAP REM}
Consider a single colored $G$-operad $\mathcal{O}$ on the category of sets
(i.e. an operad with a $G$-action commuting with all structure)
and a finite $H$-set $A$ for some subgroup $H \leq G$.
Write $\Gamma_A \leq G \times \Sigma_{|A|}$
for the graph of the homomorphism 
$H \to \Sigma_{|A|}$ encoding $A$.
We then abbreviate 
$\mathcal{O}(A)^H = \mathcal{O}(|A|)^{\Gamma_A}$,
and call this the 
\textit{set of $A$-norm maps of $\mathcal{O}$}.
This name comes from the fact that, for each $\mathcal{O}$-algebra $R$,
$\mathcal{O}(A)^H$ indexes operations
$N^A R \to R$, where the
\textit{$A$-norm object} $N^A R$ denotes
$R^{\times |A|}$ together with the twisted $H$-action
given by the graph subgroup $\Gamma_A \leq G \times \Sigma_{|A|}$.

Letting $T,R_1,R_2,S \in \Omega_G$ and $H,K,L\leq G$ again be as in
\eqref{TWOREP EQ} and Example \ref{ORBFACE EX},
the diagram of hom sets 
(recall that $\Omega(T)$ denotes the colored operad generated by $T$
\cite[\S 3]{MW07},\cite[Rem. 4.4, Ex. 4.6]{Per18})
\begin{equation}\label{NORM_COMP_OP_EQ}
\begin{tikzcd}
	\mathsf{Op}^G(\Omega(R_1), \mathcal{O})\times 
	\mathsf{Op}^G(\Omega(R_2), \mathcal{O})  &
	\mathsf{Op}^G(\Omega(T), \mathcal{O}) \ar{r} \ar{l}[swap]{\simeq}&
	\mathsf{Op}^G(\Omega(S), \mathcal{O})&
\end{tikzcd}
\end{equation}
can be interpreted, after unpacking notation (see \cite[\S 4.3]{Per18}),
as a composition of norm maps
\begin{equation}\label{NORM_COMP_EQ}
\O(H/K)^H \times \O(K/L \amalg K/K)^K \to \O(H/L \amalg H/K)^H
\end{equation}
The diagrams \eqref{NORM_COMP_GEN_EQ} for genuine operads $Z$ can then be regarded as abstracting the diagrams \eqref{NORM_COMP_OP_EQ} for $G$-operads $\O$,
though with two key differences. 
The more obvious difference is the fact that 
\eqref{NORM_COMP_OP_EQ} features no analogue of
the $Z(G/K \cdot \eta)$ term, 
though this is simply since we chose $\O$ to be single colored.
The subtler, and more crucial, difference is the fact that the terms in \eqref{NORM_COMP_GEN_EQ}
need not be described by fixed point sets as in 
\eqref{NORM_COMP_EQ}.

Therefore, one can regard genuine equivariant operads as objects that mimic the composition combinatorics of the  norm maps in a (regular) equivariant operad, while relaxing the fixed point conditions.
In fact, the reader of \cite{BP17} may recognize this as the informal description of genuine equivariant operads given in the introduction to that work,
though our current formal setting is rather different.
The connection between the two settings is as follows.
There is a nerve functor
\begin{equation}\label{NG EQ}
N_G \colon \Op_G \to \dSet_G, \qquad N_G\P(T) = \Op_G(\Omega_G(T), \P)
\end{equation}
where $\Op_G$ denotes a colored generalization of the genuine equivariant operads of \cite{BP17}
and $\Omega_G(T)$ denotes $\Omega(T)$ upon the standard inclusion
$\upsilon_{\**} \colon \Op^G \to \Op_G$.
Moreover, $N_G$ is fully faithful and its essential image are the genuine equivariant operads in the sense of Definition \ref{GEN_OP_DEF}.
However, we do not presently require these facts, and thus delay their proof to a sequel.
\end{remark}

We end this section by explaining why it is that genuine $G$-dendroidal sets $\mathsf{dSet}_G$, rather than 
$G$-dendroidal sets $\mathsf{dSet}^G$,
must be used in Definition \ref{GEN_OP_DEF}.

\begin{remark}\label{DESTIAL REM}
Suppose $X \in \mathsf{dSet}^G$ has the strict right lifting property against all Segal core inclusions
$Sc[T] \to \Omega[T]$, $T \in \Omega_G$.
By specifying to the cases of $T \simeq G \cdot T_{\**}$
the free $G$-trees,
\eqref{T_DECOMP_EQ} implies that,
after forgetting the $G$-action so as to regard $X$ as an object in  $\mathsf{dSet}$,
$X$ has the strict lifting property
against the inclusions 
$Sc[T_{\**}] \to \Omega[T_{\**}]$, $T_{\**} \in \Omega$.
But the strict lifting properties with respect to all other $G$-trees $T \in \Omega_G$ are now automatic.
Indeed, writing $T \simeq G \cdot_H T_{\**}$
for some $T_{\**} \in \Omega^H$ one has that by \eqref{T_DECOMP_EQ}
$G$-equivariant lifts against 
$G \cdot_H \left(Sc[T_{\**}] \to \Omega[T_{\**}]\right)$
are the same as
$H$-equivariant lifts against 
$Sc[T_{\**}] \to \Omega[T_{\**}]$.
Consider now the following diagram,
where $\phi$ is a $H$-equivariant map and 
$\Phi$ is the unique non-equivariant lift.
\begin{equation}
\begin{tikzcd}[row sep = 3em]
	Sc[T_*] \arrow[d] \arrow[r, "h"] &
	Sc[T_*] \arrow{r}{\phi} \arrow[d] &
	X \ar{r}{h^{-1}}&
	X
\\
	\Omega[T_*] \arrow[r, "h"'] &
	\Omega[T_*]
	\arrow[dashed]{ru}[swap]{\Phi}
\end{tikzcd}
\end{equation}
Then $h^{-1} \Phi h$ is a lift for the composite lifting problem, but since $h^{-1} \phi h = \phi$, that composite lifting problem in fact coincides with the middle lifting problem, so that strictness implies it is also
$h^{-1} \Phi h = \Phi$.
In other words, $\Phi$ is in fact also the unique 
$H$-equivariant lift.

In summary, we have shown that if we had instead used $\mathsf{dSet}^G$ in Definition \ref{GEN_OP_DEF},
then non-free $G$-trees would be superfluous, 
so that by \cite[Thm. 6.1]{MW09}
the $X \in \mathsf{dSet}^G$ with such a lifting property
would be simply the nerves of $G$-operads.
To see why this is an unsatisfactory situation we recall a fundamental basic example.
The category $\mathsf{Top}^G$ of $G$-spaces admits two main equivariant notions of weak equivalence:
the fine/genuine equivalences, which care about all fixed point spaces, and the coarse/na\"ive equivalences, 
which care only about the total spaces.
However, this distinction vanishes when working in the discrete setting of $G$-sets $\mathsf{Set}^G$,
unless one instead works with 
$G$-coefficient systems $\mathsf{Set}^{\mathsf{O}_G^{op}}$ (recall that $\mathsf{O}_G$ is the \textit{orbit category} formed by the $G$-sets $G/H$ for $H \leq G$ and $G$-equivariant maps between them).
Similarly, the category $\mathsf{sOp}^G$ 
of $G$-simplicial operads
admits two natural notions of weak equivalence, one which cares about the spaces of norm maps for
all $H$-sets $A$ and one which cares only about the spaces of norm maps for trivial $H$-sets
(which are simply the usual multiplication).
However, this distinction vanishes when working in the discrete setting of $G$-dendroidal sets $\mathsf{dSet}^G$,
unless one works instead with genuine 
$G$-dendroidal sets $\mathsf{dSet}_G$.
\end{remark}


\section{Quillen equivalences}\label{QUIEQ SEC}

Our main goal in this section is to prove Theorems \ref{INC0AGJ THM} and \ref{ANOQUEQUIV THM},
which jointly establish the Quillen equivalence of three model categories:
the category of equivariant dendroidal sets $\mathsf{dSet}^G$
with the ``$G$-$\infty$-operad'' model structure of \cite[Thm. 2.1]{Per18};
the category of equivariant dendroidal spaces
$\mathsf{sdSet}^G$ with the
``complete equivariant dendroidal Segal space'' model structure in \S \ref{CEDSS SEC} and;
the category of equivariant preoperads $\mathsf{PreOp}^G$
with the ``equivariant Segal operad'' model structure in \S \ref{PREOP SEC}.

Our perspective will be that these Quillen equivalences are best understood in light of the observation from \cite[Thm. 6.6]{CM13a},
which says that the complete dendroidal Segal space model structure
on $\mathsf{sdSet} = \mathsf{dSet}^{\Delta^{op}}$ can be obtained via two distinct left Bousfield localization procedures.
%
%
As such, we will find it helpful to first discuss an abstract framework for 
such ``joint left Bousfield localizations'',
and then \textit{define} the model structure on 
$\mathsf{sdSet}^G$ within that framework
(Definition \ref{JOINREED DEF}),
with the analogues of the original definitions in
\cite{CM13a} recovered a posteriori (Remark \ref{RECOVDEF REM}).

\subsection{Joint left Bousfield localizations}\label{JOINBOUS SEC}

Throughout we assume familiarity with the theory of left Bousfield localizations as in \cite{Hir03}.

\begin{proposition}\label{COMBMODSTR PROP}
	Suppose that a category $\mathcal{C}$
	admits two model structures $\mathcal{C}_1=(C,W_1,F_1)$ and $\mathcal{C}_2 = (C,W_2,F_2)$
	with a common class of cofibrations $C$,
	and assume further that both model structures are cofibrantly generated and admit left Bousfield localizations with respect to any set of maps.

	Then $\mathcal{C}_1$, $\mathcal{C}_2$ have a smallest joint left Bousfield localization 
	$\mathcal{C}_{1,2}=(C,W,F)$ and:
\begin{itemize}
	\item[(i)] $c \in \mathcal{C}$ is $\mathcal{C}_{1,2}$-fibrant iff it is simultaneously 
	$\mathcal{C}_{1}$-fibrant and $\mathcal{C}_{2}$-fibrant;
	\item[(ii)] for $\mathcal{C}_{1,2}$-fibrant  
	$c,d \in \mathcal{C}$ one has that $c\to d$ is in $W$ iff it is in $W_1$ iff it is in $W_2$.
\end{itemize}
\end{proposition}

\begin{proof}
The joint localized model structure $\mathcal{C}_{1,2}$ can be obtained by either left Bousfield localizing $\mathcal{C}_{1}$ with regard to the generating trivial cofibrations of $\mathcal{C}_{2}$ or vice-versa. Indeed, denoting the first localization in the previous sentence by $L_2 \mathcal{C}_1$ and the vice-versa localization by $L_1 \mathcal{C}_2$,
the identity functors 
$\mathcal{C}_1 \to L_1 \mathcal{C}_2$ and
$\mathcal{C}_2 \to L_2 \mathcal{C}_1$
are left Quillen and thus, by the universal property of left Bousfield localizations \cite[Prop. 3.3.18]{Hir03} (cofibrant approximations \cite[Def. 8.1.2]{Hir03} cause no issue since $\mathcal{C}_1,\mathcal{C}_2$ have the same cofibrations), so are the identity functors
$L_2 \mathcal{C}_1 \to L_1 \mathcal{C}_2$ and 
$L_1 \mathcal{C}_2 \to L_2 \mathcal{C}_1$. This implies that $L_2 \mathcal{C}_1$ and $L_1 \mathcal{C}_2$ indeed coincide.
	
For (i), the claim that joint fibrant objects are fibrant in both of the original model structures follows since $C \cap W$ contains both $C \cap W_1$ and $C \cap W_2$ (in fact, this shows that $F \subseteq F_1 \cap F_2$).
The converse claim follows from the observation that fibrant objects in any model structure are already local with respect to the weak equivalences in that same model structure.

Lastly, (ii) follows from the local Whitehead theorem \cite[Thm. 3.2.13]{Hir03}, stating that 
the local equivalences between local objects match the
initial weak equivalences.
\end{proof}

The prototypical example of Proposition \ref{COMBMODSTR PROP} is given by the category 
$\mathsf{ssSet} \simeq \mathsf{Set}^{\Delta^{op} \times \Delta^{op}}$
of bisimplicial sets together with the two possible Reedy structures over the Kan model structure on $\mathsf{sSet}$.
Explicitly, writing the levels of 
$X \in \mathsf{ssSet}$ as $X_n(m)$
one can either form a Reedy model structure with respect to the 
\textit{horizontal index $m$}
or with respect to the 
\textit{vertical index $n$}.

In either case, the generating cofibrations are then given by the maps
\[
	\left( \partial \Delta[n] \to \Delta[n] \right)
\square
	\left( \partial \Delta[m] \to \Delta[m] \right),
	\qquad n,m\geq 0,
\]
where $\square$ denotes the pushout product (see, for example, \cite[11.1.7]{Ri14}).

Further, in the horizontal Reedy model structure the generating trivial cofibrations are the maps
\begin{equation}\label{GTRCOHOR EQ}
	\left( \Lambda^i[n] \to \Delta[n] \right)
\square
	\left( \partial \Delta[m] \to \Delta[m] \right),
\qquad n \geq 1, n \geq i \geq 0, m\geq 0, 
\end{equation}
while for the vertical Reedy model structure the generating trivial cofibrations are the maps
\begin{equation}\label{GTRCOVER EQ}
	\left( \partial \Delta[n] \to \Delta[n] \right)
\square
	\left( \Lambda^j[m] \to \Delta[m] \right),
\qquad n\geq 0, m\geq 1,m\geq j \geq 0.
\end{equation}

We caution the reader about a possible hiccup with the terminology: 
the weak equivalences for the horizontal Reedy structure are the 
\textit{vertical equivalences},
i.e. maps inducing Kan equivalences of simplicial sets
$X_{\bullet}(m) \to Y_{\bullet}(m)$
for each $m \geq 0$, and dually for the vertical Reedy structure.

\begin{notation}\label{UNIQUELIM NOT}
	Given a fixed $X \in \mathsf{ssSet}$ we will also write
	$X_{(-)} \colon \mathsf{sSet}^{op} \to \mathsf{sSet}$
	for the unique limit preserving functor such that
	$X_{\Delta[n]} = X_n$.
	Explicitly, $X_K(m) = \mathsf{sSet}(K,X(m))$.
	
        
        Note that for maps $K\to L$ and $A\to B$ in $\mathsf{sSet}$, $X$ has the right lifting property against
        $(K \to L) \square (A\to B)$ iff
        $X_L \to X_K$ has the right lifting property against $A\to B$
        (see Remark \ref{TWOVARADJ REM}).
\end{notation}


In the next result we refer to the localized model structure given by Proposition \ref{COMBMODSTR PROP} as the 
\textit{joint Reedy model structure} and
we write $\delta^{\**} \colon \mathsf{ssSet} \to \mathsf{sSet}$
for the diagonal functor.

\begin{proposition}\label{SSSETJREE PROP}
	Suppose that $X, Y \in \mathsf{ssSet}$ are horizontal Reedy fibrant. Then:
\begin{itemize}
	\item[(i)] for each fixed $n$ all vertex maps $X_{n} \to X_{0}$ are trivial Kan fibrations in $\mathsf{sSet}$;
	\item[(ii)] any vertical Reedy fibrant replacement $\tilde{X}$ of $X$ is fibrant in the joint Reedy model structure;
	\item[(iii)] a map $X \to Y$ is a joint weak equivalence
	iff it is a horizontal weak equivalence 
	iff $X_0 \to Y_0$ is a Kan equivalence in $\mathsf{sSet}$;
	\item[(iv)] the canonical map $X_{0} \to \delta^{\**}(X)$
	(with levels $X_0(n) \to X_n(n)$ induced by degeneracies) is a Kan equivalence in $\mathsf{sSet}$. 
\end{itemize}
\end{proposition}

\begin{proof}
(i) follows since the trivial cofibrations for the horizontal Reedy structure include all the maps of the form
$(\Delta[0] \to \Delta[n]) \square (\partial \Delta[m] \to \Delta[m])$.

For (ii), the fact that $\tilde{X}$ is vertical fibrant
implies that for any monomorphism $K \to L$ in $\mathsf{sSet}$
the induced map $\tilde{X}_L \to \tilde{X}_K$ is a Kan fibration. 
Since $X \to \tilde{X}$ is a horizontal equivalence,
(i) implies that all vertex maps
$\tilde{X}_{n} \to \tilde{X}_{0}$
are trivial Kan fibrations, so that by
Remark \ref{HYPERSATKAN REM} one has that 
$\tilde{X}_L \to \tilde{X}_K$ is a trivial Kan fibration whenever
$K \to L$ is anodyne
(since the $K \to L$ with this property are hypersaturated; see Remark \ref{HYPERMODEL REM} for a similar argument).
Therefore, $\tilde{X}$ also has the right lifting property against \eqref{GTRCOHOR EQ}. That is, $\tilde{X}$ is also horizontal fibrant, as desired.


The first ``iff'' in (iii) follows from (ii) since the localizing maps 
$X \to \tilde{X}$, $Y \to \tilde{Y}$
are horizontal equivalences
while the second ``iff'' in (iii) follows from (i).

For (iv), note first that 
$\delta^{\**} \colon \mathsf{ssSet} \to \mathsf{sSet}$
is left Quillen for either the horizontal or vertical Reedy structures (and thus also for the joint Reedy structure).
But noting that all objects in $\mathsf{ssSet}$ are cofibrant, and regarding 
$X_{0}$ as a bisimplicial set that is vertically constant, 
the claim follows by noting that by (i) the map
$X_{0} \to X$ is a horizontal weak equivalence in $\mathsf{ssSet}$.
\end{proof}

\begin{corollary}\label{WEAKDIAG COR}
	A map $f\colon X \to Y$ in $\mathsf{ssSet}$ is a joint equivalence iff it induces a Kan equivalence on diagonals
	$\delta^{\**}(X) \to \delta^{\**}(Y)$ in $\mathsf{sSet}$.
\end{corollary}

\begin{proof}
	Since horizontal Reedy fibrant replacement maps
	$X \to \tilde{X}$ are vertical equivalences they are also	
	diagonal equivalences (since $\delta^{\**}$ is left Quillen), 
	so one reduces to the case of $X,Y$ horizontal Reedy fibrant. The result now follows by combining
	parts (iii) and (iv) of Proposition \ref{SSSETJREE PROP}.
\end{proof}

\begin{corollary}\label{SSETSSETADJ COR}
	The adjunction (where $\mathsf{ssSet}$ has the joint Reedy model structure)
\[
	\delta_{!}\colon \mathsf{sSet} 
		\rightleftarrows 
	\mathsf{ssSet} \colon \delta^{\**}
\]
is a Quillen equivalence.
Moreover, given $f\colon X \to Y$ in $\mathsf{ssSet}$, 
\begin{itemize}
\item $\delta^{\**} (f)$ is a Kan fibration in $\mathsf{sSet}$ if 
 $f$ has the right lifting property
against both sets of maps in
\eqref{GTRCOHOR EQ} and \eqref{GTRCOVER EQ}.
\item $\delta^{\**}(X)$ is a Kan complex in $\mathsf{sSet}$ if $X$ has the right lifting property against all maps in \eqref{GTRCOHOR EQ} 
as well as the maps in \eqref{GTRCOVER EQ}
with $m \geq 2$ (and dually for \eqref{GTRCOHOR EQ} with $n \geq 2$, all maps in \eqref{GTRCOVER EQ}).
\end{itemize}
\end{corollary}

Note that the first ``moreover'' claim 
is not quite formal, since the maps in \eqref{GTRCOHOR EQ},\eqref{GTRCOVER EQ} are not known to be generating trivial cofibrations for the joint model structure on $\mathsf{ssSet}$.

\begin{proof}
	Recall that $\delta_!$ is the unique colimit preserving functor such that 
	$\delta_!(\Delta[n])=\Delta[n] \times \Delta[n]$.

Throughout the proof we write $F$ for both a non-empty subset $\emptyset \neq F \subseteq \{0,1,\cdots,n\}$ and the corresponding face $F \to [n]$ in $\Delta$, and let $\Delta F \subseteq \Delta[n]$ denote the associated subpresheaf.

Moreover, note that for any simplex
$x \in \Delta[n]^{\times 2}$ there is a smallest pair of faces
$F_1^x, F_2^x$ such that
$x \in \Delta F_1^x \times \Delta F_2^x$, and thus also a smallest 
$F^x$ such that $x \in (\Delta F^x)^{\times 2}$,
namely $F^x = F^x_1 \cup F^x_2$.

	To see that $\delta_!$ preserves cofibrations 
	it is enough to show that 
	$\delta_{!}\left( \partial \Delta[n] \to \Delta[n]\right)$
	is a monomorphism for all $n\geq 0$.
	Since
	$\partial \Delta[n] = \mathop{\colim}_{\text{faces }F \neq [n]}\Delta F$, this map is
	$\left(\mathop{\colim}_{\text{faces }F \neq [n]}\Delta F^{\times 2}\right) \to \Delta[n]^{\times 2}$, which is a monomorphism since:
\begin{inparaenum}
	\item[(i)] for any face $F$ the map $\Delta F^{\times 2} \to \Delta[n]^{\times 2}$ is a monomorphism;
	\item[(ii)] in the colimit defining $\delta_{!} \left(\partial \Delta[n]\right)$ any simplex $x \in \Delta F^{\times 2}$ is identified with the simplex $x \in (\Delta F^x)^{\times 2}$.
\end{inparaenum}

The claim that $\delta_!$ preserves trivial cofibrations follows easily from Remark \ref{HYPERSATKAN REM} together with Corollary \ref{WEAKDIAG COR}, but here we give a harder argument needed to establish the stronger ``moreover'' claims.
Namely, we will argue that the maps
$\delta_! \left( \Lambda^i[n] \to \Delta[n]\right)$
are built cellularly out of the maps in 
\eqref{GTRCOHOR EQ}, \eqref{GTRCOVER EQ}.
One has a factorization
\[
	\delta_! \Lambda^i[n] \to
	\Lambda^i[n] \times \Delta[n] \to \Delta[n]^{\times 2}
\]
where the second map is clearly built cellularly out of the maps in 
\eqref{GTRCOHOR EQ}, and we claim that 
the first map is likewise built cellularly out of the maps in \eqref{GTRCOVER EQ}.
Writing $\mathsf{Face}_{\supsetneq \{i\}}$ for the poset
of faces of $[n]$ strictly containing $\{i\}$, this first map 
can be factored as a sequence of maps of the form
\begin{equation}\label{ANOTHERCONV EQ}
\delta_! \Lambda^i[n] \cup \bigcup_{G \in C} \Lambda^i[n] \times \Delta G
\to
\delta_! \Lambda^i[n] \cup \bigcup_{G \in C'} \Lambda^i[n] \times \Delta G
\end{equation}
for convex
$C \subseteq C' \subseteq \mathsf{Face}_{\supsetneq \{i\}}$ such that $C' = C \amalg \{F\}$ for some face $F\supsetneq \{i\}$.
Note that a simplex 
$x \in \Lambda^i[n] \times \Delta F$
will be in $\bigcup_{G \in C} \Lambda^i[n] \times \Delta G$ iff 
$F_2^x \not \supseteq F-\{i\}$, 
and that if 
$F_2^x \supseteq F-\{i\}$
then $x \in \delta_! \Lambda^i[n]$ 
iff $F^x=F_1^x \cup F_2^x \not \supseteq [n] - \{i\}$
iff $F_1^x \not \supseteq [n] - F$,
it follows that \eqref{ANOTHERCONV EQ} is a pushout of the map
\begin{equation}\label{LAMBDAFPUSH EQ}
	\left(\Lambda^{F}[n] \to \Lambda^i[n] \right)
		\square
	\left(\Lambda^{i}F \to \Delta F \right)
\end{equation}
where $\Lambda^{F}[n]$ is the union of those $\Delta H$ with $H \not \supseteq [n] - F$ (this is consistent with the horn notation in \S \ref{EQDENDSETS SEC})
and $\Lambda^i F$ the union of the $\Delta H$ with $F \supseteq H \not \supseteq F - \{i\}$.
Most maps in \eqref{LAMBDAFPUSH EQ}
can in fact be built from either \eqref{GTRCOHOR EQ} or \eqref{GTRCOVER EQ}, but there is one crucial exception:
for $F=[n]$ it is $\Lambda^F[n]=\emptyset$, 
so only \eqref{GTRCOVER EQ} can be used.
The first ``moreover'' claim now follows. For the second ``moreover'' claim, note that any $X \in \mathsf{ssSet}$ admits (degenerate) liftings against 
$\delta_!\left(\Lambda^i[1] \to \Delta[1]\right) =
 \left(\{i\}^{\times 2} \to \Delta[1]^{\times 2}\right)$,
and that the only step in the filtration of 
$\delta_!\left( \Lambda^i[n] \to \Delta[n]\right)$ for $n\geq 2$
requiring \eqref{GTRCOVER EQ} (i.e. the case $F=[n]$ in \eqref{LAMBDAFPUSH EQ})
uses the map
$\Lambda^i[n] \times \left(\Lambda^i[n] \to \Delta[n] \right)$.

Lastly, the Quillen equivalence condition 
is that for all $X \in \mathsf{sSet}$ and joint fibrant
$Y \in \mathsf{ssSet}$ a map
$X \to \delta^{\**} Y$ is a weak equivalence iff 
$\delta_!X \to Y$ is. Factoring the former map as
$X \to \delta^{\**} \delta_! X \to \delta^{\**} Y$,
by Corollary \ref{WEAKDIAG COR}
this reduces to showing
that the unit maps $X \to \delta^{\**} \delta_! X$
are weak equivalences. This latter claim follows by cellular induction on $X$, since those pushouts attaching cells are homotopy pushouts (due to $\mathsf{sSet}$ being left proper).
\end{proof}

\begin{remark}\label{HYPERSIMPL REM}
	Just as in the proof of Proposition \ref{SSSETJREE PROP}, hypersaturations simplify the lifiting conditions
	in the previous result. 
	
	Indeed,	$X \to Y$ is a vertical fibration (i.e. it has the lifting property against \eqref{GTRCOVER EQ})
	iff, for each monomorphism $K \to L$ in $\mathsf{sSet}$,
	$X_L \to X_K \times_{Y_K} Y_L$
	is a Kan fibration.
	The lifting property against \eqref{GTRCOHOR EQ}
	then requires that 
	$X_L \to X_K \times_{Y_K} Y_L$ is a trivial Kan fibration
	when $K \to L$ is a horn inclusion. 
	But a straightforward argument (cf. Remark \ref{HYPERMODEL REM}) shows that the $K \to L$ with this property are hypersaturated,
	so that by 
Remark \ref{HYPERSATKAN REM} it suffices to check that the maps $X_n \to X_0 \times_{Y_0} Y_n$, induced by maps $[0] \to [n]$, are trivial Kan fibrations.

	Similarly, if $X$ is vertical fibrant (i.e. $X$ has the lifting property against \eqref{GTRCOVER EQ}) then, by 
	Remarks \ref{SLICE REM} and \ref{ANHYPER REM}, to check that $X$ has the lifting property against \eqref{GTRCOHOR EQ} for $n\geq 2$
	it suffices to check that the maps
	$X_n \to X_{Sc[n]}$, $n\geq 0$, 
	$X_2 \to X_{\Lambda^0[2]}$,
	$X_2 \to X_{\Lambda^2[2]}$ are trivial Kan fibrations.
\end{remark}

\begin{remark}
The adjunction 
$
	\delta^{\**}\colon \mathsf{ssSet}
		\rightleftarrows 
	\mathsf{sSet} \colon \delta_{\**}
$
can also be shown to be a Quillen equivalence.
\end{remark}

\subsection{Complete equivariant dendroidal Segal spaces}
\label{CEDSS SEC}

We now turn to our main application of Proposition \ref{COMBMODSTR PROP}, the category 
$\mathsf{sdSet}^G = \mathsf{Set}^{\Delta^{op} \times \Omega^{op} \times G}$
of $G$-equivariant simplicial dendroidal sets.

Since $\Delta$ is a (usual) Reedy category the model structure on $\mathsf{dSet}^G$ 
in \cite[Thm. 2.1]{Per18} induces 
a model structure on $\mathsf{sdSet}^G$
that we will refer to as the \textit{simplicial Reedy model structure}.

On the other hand, in the context of Definition \ref{GENRED DEF},
$\Omega^{op} \times G$ is a generalized Reedy category such that the families $\{\mathcal{F}_{U}^{\Gamma}\}_{U \in \Omega}$
of $G$-graph subgroups are Reedy-admissible 
(see Example \ref{GGRAPHREEDY EX})
and hence, using the underlying 
Kan model structure on $\mathsf{sSet}$, 
Theorem \ref{REEDYADM THM} yields
a model structure on $\mathsf{sdSet}^G$
that we will refer to as the \textit{equivariant dendroidal Reedy model structure}, 
or simply as the \textit{dendroidal Reedy model structure} for the sake of brevity.

Throughout, we will write the levels of 
$X \in \mathsf{sdSet}^G$ as 
$X_n(U)$ for $n\geq 0$, $U \in \Omega$.
We now extend Notation \ref{UNIQUELIM NOT}.
Note that the representable functor of
$U \in \Omega \times G^{op}$ is given by $\Omega[G \cdot U] = G \cdot \Omega[U]$.

\begin{notation}\label{UNILIMDEN NOT}
	Given a fixed $X \in \mathsf{sdSet}^G$ we will also write
\[
	X(-)\colon \left(\mathsf{dSet}^G \right)^{op} \to \mathsf{sSet},\qquad 
	X_{(-)} \colon \mathsf{sSet}^{op} \to \mathsf{dSet}^{G}
\]
	for the unique limit preserving functors such that
	$X(\Omega[G \cdot U]) = X(U)$, $X_{\Delta[n]} = X_n$.
	
	Explicitly, $\left(X(A)\right)_n = \mathsf{dSet}^G(A,X_n)$
	and $X_K(U) = \mathsf{sSet}(K,X(U))$.
%
  
Moreover,
for fixed $J \in \mathsf{dSet}^{G}$
we define $X^J \in \mathsf{sdSet}^G$ by 
$X^J(U) = X(\Omega[G\cdot U] \otimes J)$,
	where $\otimes$ is the tensor product of dendroidal sets (see Example \ref{THM71 EX} for an informal definition or \cite[\S 9]{MW09},\cite[\S 7]{Per18} for an in-depth discussion).
\end{notation}

\begin{remark}\label{TWOVARADJ REM}
The notation just defined fits into a 
two-variable adjunction 
(cf. \cite[\S 10.1]{Ri14})
\[
	\mathsf{sSet} \times \mathsf{dSet}^G
	\xrightarrow{\times} \mathsf{sdSet}^G
\qquad
	\mathsf{sdSet}^G \times \left(\mathsf{dSet}^G\right)^{op}
	\xrightarrow{(-)(-)} \mathsf{sSet}
\qquad
	\mathsf{sdSet}^G \times \left(\mathsf{sSet}\right)^{op}
	\xrightarrow{(-)_{(-)}} \mathsf{dSet}^G
\]
For our purposes, the relevance of this is that for maps
$K \to L$ in $\mathsf{sSet}$,
$A \to B$ in $\mathsf{dSet}^G$ and
$X \to Y$ in $\mathsf{sdSet}^G$ the three induced lifting problems shown below are all equivalent.
\begin{equation}
	\begin{tikzcd}
		K \times B \coprod_{K \times A} L \times A \ar{r} \ar{d} & X \ar{d}
	&
		K \ar{r} \ar{d} & X(B) \ar{d}
	&
		A \ar{r} \ar{d} & X_L \ar{d}
	\\
		L \times B \ar{r} \ar[dashed]{ru} &  Y
	&
		L \ar{r} \ar[dashed]{ru} & X(A) \times_{Y(A)} Y(B)
	&
		B \ar{r} \ar[dashed]{ru} & X_K \times_{Y_K} Y_L
	\end{tikzcd}
\end{equation}
\end{remark}

\begin{remark}\label{HYPERMODEL REM}
Fix $X \in \mathsf{sdSet}^G$ and suppose that,
for all normal monomorphisms $A\to B$ in $\mathsf{dSet}^G$,
$X(B) \to X(A)$ is a Kan fibration in $\mathsf{sSet}$.
Then the class of normal monomorphisms $A \to B$
such that $X(B) \to X(A)$
is a \emph{trivial} Kan fibration is hypersaturated. 
Indeed, saturation follows from properties of trivial fibrations  while \eqref{CANCEL_EQ} follows from $2$-out-of-$3$ for Kan equivalences.
A similar remark holds if
$X_L \to X_K$ is 
a fibration in $\mathsf{dSet}^G$
for all monomorphisms $K\to L$ in $\mathsf{sSet}$.
\end{remark}



\begin{notation}\label{JM NOT}
Writing
	$\widetilde{[m]} = 
	(0 \rightleftarrows 
	1 \rightleftarrows \cdots
	\rightleftarrows m)$
for the contractible groupoid with objects 
$0,1,\cdots,m$ (cf. Remark \ref{CONTGR REM}),
we write $J^m \in \mathsf{sSet}$ for the nerve
\[
J^m = N \widetilde{[m]}
\]
and abbreviate $J = J^1$.
Moreover, we slightly abuse notation by also writing $J^m \in \mathsf{dSet}^G$ for the corresponding dendroidal set 
(under the standard inclusion $\iota_! \colon \mathsf{sSet} \to \mathsf{dSet}$; see Notation \ref{DELTAOMEGA NOT}) with the trivial $G$-action.
In fact, we use $J^m \in \mathsf{dSet}^G$ by default and the few  exceptions,
Proposition \ref{SESP PROP},
the end of
Proposition \ref{JDDK PROP}
and
Theorem \ref{COMPIFFDK THM},
should be clear from context.

Lastly, we make a notational remark that may avoid confusion. $J^m$ will always feature in expressions $X(J^m)$, $X^{J^m}$ (or slight variants) as defined in Notation \ref{UNILIMDEN NOT}. As such, $J^{m}$ as $m$ changes should be thought of as ``varying in the dendroidal direction'' (which corresponds to the horizontal simplicial direction $m$), never in the (vertical) simplicial direction $n$.
\end{notation}

\begin{remark}\label{JREEDYCOF REM}
$J^{\bullet} \in \left(\mathsf{sSet}\right)^{\Delta}$ 
(and thus $J^{\bullet} \in \left(\mathsf{dSet}^G\right)^{\Delta}$) is a Reedy cofibrant cosimplicial object. 

To see this, we mimic the notation (and argument) in the proof of
Corollary \ref{SSETSSETADJ COR} by writing 
$F$ for both a subset $\emptyset \neq F \subseteq \{0,1,\cdots,m\}$
and face $F \to [m]$, as well as $\widetilde{F} \to \widetilde{[m]}$ for the natural subgroupoid.
Further, note that for any simplex $x \in N\widetilde{[m]}$, which can be identified with a non-empty string on $\{0,1,\cdots,n\}$, there is a smallest face $F^x \to [m]$ such that 
$x \in N\widetilde{F^x}$, namely $F^x$ is the set of ``letters'' of $x$.
The latching map $L_m J^{\bullet} \to J^m$ is then
$
\left(\mathop{\colim}_{\text{faces }F \neq [m]} N \widetilde{F}\right) \to N\widetilde{[m]}
$
which is monomorphism since:
\begin{inparaenum}
	\item[(i)] for any face $F$ the map $N \widetilde{F} \to N\widetilde{[m]}$ is a monomorphism;
	\item[(ii)] any simplex $x \in N \widetilde{F}$ is identified with the simplex $x \in N \widetilde{F^x}$ in the colimit.
\end{inparaenum}
\end{remark}

\begin{proposition}
      \label{SDSETJRCOF PROP}
	Both the simplicial and dendroidal Reedy model structures on 
	$\mathsf{sdSet}^G$ have generating cofibrations given by the maps
\begin{equation}\label{JOINTCOF EQ}
	\left(\partial \Delta [n] \to \Delta[n]\right)
		\square
	\left(\partial \Omega[T] \to \Omega[T]\right),
	\qquad
	n\geq 0, T \in \Omega_G.
\end{equation}
  Further, the dendroidal Reedy structure has as generating trivial cofibrations the maps
\begin{equation}\label{DENDTRIVCOF EQ}
	\left(\Lambda^i [n] \to \Delta[n]\right)
		\square
	\left(\partial \Omega[T] \to \Omega[T]\right),
	\qquad
 	n\geq 1, n\geq i \geq 0, T \in \Omega_G.
\end{equation}
while the simplicial Reedy structure has as generating trivial cofibrations the maps
\begin{equation}\label{SIMPTRIVCOF EQ}
	\left(\partial \Delta [n] \to \Delta[n]\right)
		\square
	\left(A \to B\right),
	\qquad
	n\geq 0
\end{equation}
for $\{A \to B\}$ a set of generating trivial cofibrations of
$\mathsf{dSet}^G$.
\end{proposition}

\begin{proof}
The claims concerning the simplicial Reedy structure are standard. For the dendroidal Reedy model structure, the result follows by Proposition \ref{REEDYCOFGEN PROP} and Example \ref{REEDYCOFGEN EX}. 
%
\end{proof}

We call the saturation of the maps in \eqref{JOINTCOF EQ} the class of \textit{normal monomorphisms} of $\mathsf{sdSet}^G$.

\begin{remark}\label{MAPSPACE REM}
$\mathsf{sdSet}^G$ is simplicially tensored and cotensored via the pointwise 
formulas $(K \times X)(U) = K \times X(U)$ and
$\left(X^K\right)(U) = X(U)^K$
for $K \in \mathsf{sSet}, X \in \mathsf{sdSet}^G$, 
and hence becomes a simplicial category with mapping spaces
$\left(\underline{\mathsf{sdSet}}^G(A,X)\right)_{\bullet} = \mathsf{sdSet}^G(\Delta[\bullet] \times A,X)$.

The generating sets \eqref{JOINTCOF EQ},\eqref{DENDTRIVCOF EQ} then show that these mapping spaces make the dendroidal Reedy model structure into a simplicial model structure, though the same does not hold for the simplicial Reedy model structure
(since the pushout product of an anodyne map in $\mathsf{sSet}$ with a map in \eqref{JOINTCOF EQ} will be in the saturation of \eqref{DENDTRIVCOF EQ} rather than in that of \eqref{SIMPTRIVCOF EQ}).
\end{remark}

\begin{definition}\label{JOINREED DEF}
The \textit{joint Reedy model structure} on $\mathsf{sdSet}^G$ is the smallest joint left Bousfield localization (cf. Proposition \ref{COMBMODSTR PROP}) of the dendroidal and simplicial Reedy model structures.
\end{definition}

\begin{remark}\label{JOINREED REM}
Historically, the joint Reedy model structure has been called the \textit{Rezk model structure} and its fibrant objects have been called 
\textit{complete (dendroidal) Segal spaces}.
This is because most discussions in the literature 
\cite{Rez01,CM13a}
prefer to first introduce the Segal space model structure on
$\mathsf{ssSet}$/$\mathsf{sdSet}$, which is an intermediate localization of the 
horizontal/dendroidal Reedy model structure.
In this work we first focus on the properties of the Rezk model structure that are consequences of the joint perspective, 
postponing the Segal space perspective to \S \ref{EDSS_SEC}.
\end{remark}

\begin{proposition}\label{JOINTFIBCHAR PROP}
The joint fibrant objects $X \in \mathsf{sdSet}^G$ have the following equivalent characterizations:
\begin{itemize}
	\item[(i)] $X$ is both simplicial Reedy fibrant and dendroidal Reedy fibrant;
	\item[(ii)] $X$ is simplicial Reedy fibrant and all the simplicial structure maps 
	$X_0 \to X_n$ are weak equivalences in $\mathsf{dSet}^{G}$;
	\item[(iii)] $X$ is dendroidal Reedy fibrant and all the natural maps
\begin{equation}\label{JOINTFIBCHAR EQ}
	X\left(\Omega[T]\right) \to X\left(Sc[T]\right)
\qquad \text{and} \qquad
	X\left(\Omega[T]\right) \to X(\Omega[T]\otimes J)
\end{equation}
for $T \in \Omega_G$ are Kan equivalences in $\mathsf{sSet}$.
\end{itemize}
\end{proposition}

\begin{proof}
(i) simply repeats Proposition \ref{COMBMODSTR PROP}(i). In the remainder we write $K \to L$ for a generic monomorphism in 
$\mathsf{sSet}$
and $A \to B$ for a generic normal monomorphism in $\mathsf{dSet}^G$.

For (ii), note that $X$ is simplicial fibrant iff 
$X_L \to X_K$ is always a fibration in $\mathsf{dSet}^G$. 
And $X$ will also have the right lifting property against \eqref{DENDTRIVCOF EQ} iff 
$X_L \to X_K$ is a trivial fibration whenever $K \to L$ is anodyne. But by Remark \ref{HYPERSATKAN REM}
it suffices to consider the vertex inclusions $\Delta[0] \to \Delta[n]$.
The claim now follows from 2-out-of-3 applied to the composites $X_0 \to X_n \to X_0$.

For (iii), note first that $X$ is dendroidal fibrant iff $X(B) \to X(A)$ is always a Kan fibration.
Further, $X$ will have the right lifting property against \eqref{SIMPTRIVCOF EQ} iff 
$X(B) \to X(A)$ is a trivial Kan fibration whenever $A\to B$ is a generating trivial cofibration of $\mathsf{dSet}^G$.
By adjunction, this is equivalent to 
$X_L \to X_K$ always being a fibration in $\mathsf{dSet}^G$.
Further, since for $K=\emptyset$ it is $X_{K}=\**$, this requires $X_L$ to always be fibrant. In other words, $X$ has the right lifting property against \eqref{SIMPTRIVCOF EQ} iff
$X_L \to X_K$ is always a fibration between fibrant objects in $\mathsf{dSet}^G$.
By the characterizations of fibrant objects and of fibrations between fibrant objects in $\mathsf{dSet}^G$ (see \cite[Prop. 8.8]{Per18} and the beginning of \cite[\S 8.1]{Per18}), it suffices to check that the maps $X_L \to X_K$ have the right lifting property against the maps 
($\square^{\otimes}$ denotes the pushout product with respect to
$\otimes$)
\begin{equation}\label{FIBBFIB EQ}
	\Lambda^{G e} [T] \to \Omega[T],
	\phantom{1}
	T \in \Omega_G, e \in \boldsymbol{E}^{\mathsf{i}}(T)
\quad
	\left( \partial\Omega[T] \to \Omega[T]\right) \square^{\otimes} \left( \{i\} \to J\right),
	\phantom{1}
	T \in \Omega_G, i = \{0,1\}
\end{equation}
(note that the case $K=\emptyset$ shows that all the $X_L$ are fibrant).

 
 It now suffices to check that $X(B) \to X(A)$ is a trivial Kan fibration when $A\to B$ is in \eqref{FIBBFIB EQ}.
To finish the proof, we use Remark \ref{HYPERMODEL REM}
combined with Proposition \ref{HYPER PROP}
and Remark \ref{DUMBHYPER REM}.
%
\end{proof}

\begin{remark}\label{RECOVDEF REM}
The characterizations of joint fibrant objects in Proposition \ref{JOINTFIBCHAR PROP} show that, to obtain the joint Reedy structure, 
one needs not localize the dendroidal Reedy structure with respect to \emph{all generating trivial cofibrations}
of the simplicial Reedy structure, and vice-versa. 

Indeed, since all objects appearing in 
\eqref{JOINTFIBCHAR EQ}
can be regarded as mapping spaces (cf. Remark \ref{MAPSPACE REM}), 
one needs only localize the dendroidal Reedy model structure with respect to the maps
\begin{equation}\label{RECOVDEN REM}
Sc[T] \to \Omega[T], \qquad \Omega[T] \otimes J \to \Omega[T], \qquad T \in \Omega_G,
\end{equation}
recovering 
the definition of the \emph{dendroidal Rezk model category} 
from \cite[Defs. 5.4 and 6.2]{CM13a}.

Similarly, Proposition \ref{JOINTFIBCHAR PROP}(ii) shows that one needs only localize the simplicial Reedy model structure with respect to the maps
$(\Delta[0] \to \Delta[n])\square (\partial \Omega[T] \to \Omega[T])$ (since then for $X$ locally fibrant the maps $X_n \to X_0$ must be trivial fibrations in $\mathsf{dSet}^G$), and by Remark \ref{DUMBHYPER REM}
it is also enough to localize the vertex maps 
$\Omega[T] \to \Delta[n] \times \Omega[T]$, which is the same as localizing 
the projection maps
\[\Delta[n] \times \Omega[T] \to \Omega[T],
\qquad n\geq 0, T \in \Omega_G\]
recovering
the definition of the \emph{locally constant model category} 
from \cite[Def. 4.6]{CM13a}.

Lastly, the fact that the two localizations just described coincide recovers \cite[Thm. 6.6]{CM13a}.
      %
      %
\end{remark}

We now obtain the following partial analogue of Proposition \ref{SSSETJREE PROP}. Note that the equivalences in the simplicial Reedy model structure are the dendroidal equivalences and vice versa.

\begin{corollary}\label{SDSETG COR}
	Suppose that $X, Y \in \mathsf{sdSet}^G$ are dendroidal Reedy fibrant. Then:
\begin{itemize}
	\item[(i)] for all $n$ the vertex maps $X_{n} \to X_{0}$ are trivial fibrations in $\mathsf{dSet}^G$;
	\item[(ii)] any simplicial Reedy fibrant replacement $\tilde{X}$ of $X$ is fibrant in the joint Reedy model structure;
	\item[(iii)] a map $X \to Y$ is a joint weak equivalence
	iff it is a dendroidal weak equivalence iff 
	$X_0 \to Y_0$ is an equivalence in $\mathsf{dSet}^G$;
	\item[(iv)] regarding $X_0$ as a simplicially constant object in $\mathsf{sdSet}^G$, the map $X_0 \to X$ is a dendroidal equivalence, and thus a joint equivalence. 
\end{itemize}
\end{corollary}

\begin{proof}
	The proof adapts that of Proposition \ref{SSSETJREE PROP}.
	(i) follows since $X$ then has the right lifting property with respect to all maps 
	$(\Delta[0] \to \Delta[m]) \square (\partial \Omega[T] \to \Omega[T])$. (ii) follows from (i) and the characterization in  Proposition \ref{JOINTFIBCHAR PROP} (ii). The first ``iff'' in (iii) follows from (ii) since the simplicial fibrant replacement maps 
	$X \to \tilde{X}$ are dendroidal equivalences
	and the second ``iff'' in (iii) follows from (i).
	(iv) follows from (i).
\end{proof}

\begin{theorem}\label{INC0AGJ THM}
	The constant/$0$-th level adjunction
	\[
	c_!\colon 
	\mathsf{dSet}^G \rightleftarrows \mathsf{sdSet}^G
	\colon (-)_0
	\]
	where $\mathsf{sdSet}^G$ is given the Rezk/joint Reedy model structure,
	is a Quillen equivalence.
\end{theorem}

\begin{proof}
	It is clear that the constant functor $c_!$ preserves both normal monomorphisms and all weak equivalences, hence the adjunction is Quillen. 
	Consider any map $c_!(A) \to X$ with $X$ joint fibrant and perform a ``trivial cofibration followed by fibration'' factorization as on the left
\[
c_!(A) \overset{\sim}{\rightarrowtail} \widetilde{c_!(A)} \twoheadrightarrow X
	\qquad
A \xrightarrow{\sim} \widetilde{c_!(A)}_0 \to X_0
\]
	for the simplicial Reedy model structure. 
	Proposition \ref{JOINTFIBCHAR PROP}(ii) now implies that 
	$\widetilde{c_!(A)}$ is in fact joint fibrant
	and thus that the leftmost composite is a joint equivalence iff $\widetilde{c_!(A)} \to X$ is a dendroidal equivalence in $\mathsf{sdSet}^G$ iff $\widetilde{c_!(A)}_0 \to X_0$ is an equivalence in  $\mathsf{dSet}^G$ iff the rightmost composite is an equivalence in $\mathsf{dSet}^G$.
\end{proof}

\begin{remark}\label{CONCRECOM REM}
      Given a $G$-$\infty$-operad $X \in \mathsf{dSet}^G$,
      one can obtain an explicit model for $\widetilde{c_{!} (X)}$
      as the object $X^{J^{\bullet}} \in \mathsf{sdSet}^G$,
      where $J^{m}$ was defined in Notation \ref{JM NOT}
      and $X^{J^m}\in \mathsf{dSet}^G$ is defined as in 
      Notation \ref{UNILIMDEN NOT}.
      Indeed, since $J^{\bullet}$ is a Reedy cofibrant cosimplicial object in $\mathsf{dSet}^G$ (cf. Remark \ref{JREEDYCOF REM}),
      one has that $X^{J^{\bullet}} \in \mathsf{sdSet}^G$
      is simplicial fibrant.
      Hence, by Proposition \ref{JOINTFIBCHAR PROP}(ii)
      $c_{!}(X) \to X^{J^{\bullet}}$
      will be a joint fibrant replacement provided that it is a dendroidal equivalence.
      But this follows from \cite[Cor. 8.21]{Per18},
      which implies that the maps $X^{J^m} \to X^{J^0}=X$
      are trivial fibrations in $\mathsf{dSet}^G$
      (formally, \cite[Cor. 8.21]{Per18} says that $\upsilon_{\**}(X^{J^m}) = X^{(J^m)} \to \upsilon_{\**}(X)$
      is a trivial fibration in $\mathsf{dSet}_G$, which is an equivalent statement, as noted at the end of the proof of \cite[Thm. 8.22]{Per18}).
%
\end{remark}

\subsection{Equivariant Segal operads}\label{PREOP SEC}

Recall that the category $\mathsf{PreOp}$
of \textit{pre-operads} is the full subcategory
$\mathsf{PreOp} \subset \mathsf{sdSet}$
of those $X$ such that $X(\eta)$ is a discrete simplicial set.
Writing $\gamma^{\**}$ for the inclusion one has left and right adjoints $\gamma_!$ and $\gamma_{\**}$
\begin{equation}
\begin{tikzcd}[column sep =5em]
	\mathsf{PreOp}^G \ar{r}[swap]{\gamma^{\**}} 
	&
	\mathsf{sdSet}^G
	\ar[bend right]{l}[swap,midway]{\gamma_{!}}
	\ar[bend left]{l}{\gamma_{\**}}
\end{tikzcd}
\end{equation}
described as follows \cite[\S 7]{CM13a}:
$\gamma_{!}X (U) = X(U)$ if $U \not \in \Delta$
while $\gamma_{!}X ([n])$ for $[n] \in \Delta$ is given by the pushout on the left below; 
$\gamma_{\**}X(U)$ is given by the pullback on the right below.
\begin{equation}\label{GAMMASTAR_EQ}
\begin{tikzcd}
	X(\eta) \ar{r} \ar{d} \arrow[dr, phantom, "\ulcorner", very near start]  &
	\pi_0 X(\eta) \ar{d}
&& 
	\gamma_{\**}X(U) \ar{r} \ar{d} & X(U) \ar{d}
\\
	X([n]) \ar{r} & \gamma_! X([n]) 
&&
	\prod_{\boldsymbol{E}(U)} X_0(\eta) \ar{r} &
	\prod_{\boldsymbol{E}(U)} X(\eta)
	\arrow[lu, phantom, "\lrcorner", very near start]
\end{tikzcd}
\end{equation}

\begin{remark}\label{GAMMASH REM}
Any monomorphism $A \to B$ in $\mathsf{sdSet}^G$
such that $A(\eta) \to B(\eta)$ is an isomorphism
induces a pushout square
\begin{equation}\label{GAMMASH EQ}
\begin{tikzcd}
	A \ar{r} \ar{d} \arrow[dr, phantom, "\ulcorner", very near start]&
	\gamma^{\**}\gamma_! A \ar{d}
\\
	B \ar{r} & \gamma^{\**}\gamma_! B 
\end{tikzcd}
\end{equation}
\end{remark}

The assignment
$U \mapsto \prod_{\boldsymbol{E}(U)} Y(\eta)$
is the $0$-coskeleton of $Y$ in the dendroidal Reedy direction
(see \eqref{SKELADJ EQ}).
To avoid confusion with the coskeleton in the simplicial direction,
and since $\eta$ is the only tree of degree $0$, 
we denote this dendroidal coskeleton by $\mathsf{csk}_{\eta} Y$. We then have the following.

\begin{proposition}\label{CSKETALT PROP}
Let $X \in \mathsf{sdSet}^G$. Then:
\begin{itemize}
	\item[(i)] if $X \in \mathsf{sdSet}^G$ is dendroidal Reedy fibrant
then so is $\gamma^{\**}\gamma_{\**} X$;
	\item[(ii)] regarding $X_0$ as a simplicially constant object of $\mathsf{sdSet}^G$, the left square below is a pullback;
	\item[(iii)] if $A\to A'$ is a map in $\mathsf{dSet}^G$
	such that $A(\eta) \simeq A'(\eta)$,
	the right square below is a pullback.
\end{itemize}
\[
\begin{tikzcd}
	\gamma^{\**} \gamma_{\**}X \ar{r} \ar{d} & X \ar{d}
&&
	\gamma^{\**} \gamma_{\**}X \ar{d} \ar{r}(A') &
	X(A')\ar{d}
\\
	\mathsf{csk}_{\eta} X_0 \ar{r} &
	\mathsf{csk}_{\eta} X
	\arrow[lu, phantom, "\lrcorner", very near start]
&&
	\gamma^{\**} \gamma_{\**}X \ar{r}(A) &
	X(A) \arrow[lu, phantom, "\lrcorner", very near start]
\end{tikzcd}
\]
\end{proposition}

\begin{proof}
(ii) is immediate from the observation that
$\left(\mathsf{csk}_{\eta} Y\right)(U)=
\prod_{\boldsymbol{E}(U)} Y(\eta)$.
Moreover, it readily follows that
for $B \in \mathsf{dSet}^G$ it is
$(\mathsf{csk}_{\eta} Y)(B)=
\left(\prod_{B(\eta)} Y(\eta) \right)^G$, where $(-)^G$ denotes fixed points for the conjugation action. Hence, since $(-)(B)$ preserves pullbacks, 
(iii) follows from (ii).

For (i), formal considerations imply that if $X$ is dendroidal (Reedy) fibrant then the map 
$X \to \mathsf{csk}_{\eta} X$ is a dendroidal fibration
(and $\mathsf{csk}_{\eta} X$ is dendroidal fibrant).
Hence, the result will follow provided that 
$\mathsf{csk}_{\eta} X_0$ is also dendroidal fibrant.
But since $\mathsf{csk}_{\eta} X_0$ is $\eta$-coskeletal, 
it suffices to check that the $\eta$-matching map
$(\mathsf{csk}_{\eta} X_0)(\eta) \to 
M_{\eta}(\mathsf{csk}_{\eta} X_0)$
is a $G$-fibration in $\mathsf{sSet}^G$.
But this is simply $X_0(\eta) \to \**$ regarded as a map of constant simplicial sets, and the result follows.
\end{proof}

\begin{notation}
In the remainder of the section we write $\mathcal{I}'$ for the set of maps
\begin{equation}\label{BOUNDRED EQ}
	\left( \partial \Delta[n] \to \Delta [n] \right)
\square
	\left( \partial \Omega[T] \to \Omega [T] \right),
\qquad
	n\geq 0, T\in \Omega_G, T \not = G/H \cdot \eta.
\end{equation}
Further, we note that Remark \ref{GAMMASH REM} applies to these maps.
\end{notation}

In what follows, we say a map in $\mathsf{PreOp}^G$ is a normal monomorphism if it is one in $\mathsf{sdSet}^G$.

\begin{lemma}\label{GENSET LEM}
	The normal monomorphisms in $\mathsf{PreOp}^G$ are the saturation of the set of maps
$\{\emptyset \to G/H\cdot \eta \colon H \leq G\} \cup \gamma_! (\mathcal{I}')$.
\end{lemma}

\begin{proof}
	If $A \to B$ in $\mathsf{PreOp}^G$ is a normal monomorphism, then using the cellular filtration in $\mathsf{sdSet}^G$ one can write 
	$\gamma^{\**}(A \to B)$ as a transfinite composition of pushouts of maps in 
	$\{\emptyset \to G/H\cdot \eta\} \cup \mathcal{I}'$, 
	and hence $\gamma_!\gamma^{\**}(A \to B) \simeq (A \to B)$ can similarly be written as a transfinite composition of pushouts of maps in 
	$\{\emptyset \to G/H\cdot \eta\} \cup \gamma_!\left(\mathcal{I}'\right)$.
	The remaining claim that the maps in $\{\emptyset \to G/H\cdot \eta\} \cup \gamma_!(\mathcal{I}')$ are normal monomorphisms follows from the pushouts \eqref{GAMMASH EQ}.
\end{proof}

\begin{lemma}\label{TRIVFIB LEM}
	Any map in $\mathsf{PreOp}^G$ which has the right lifting property against all normal monomorphisms in $\mathsf{PreOp}^G$
	is a joint equivalence in $\mathsf{sdSet}^G$.
\end{lemma}

\begin{proof}
We simply adapt the proof of \cite[Lemma 8.12]{CM13a} mutatis mutandis. 

Choose a normalization $E_{\infty}$ of $\**$ in 
$\mathsf{dSet}^G$, i.e. a normal object such that 
$E_{\infty} \to \**$ is a trivial fibration. 
Regarding $E_{\infty}$ as a simplicially constant object in $\mathsf{sdSet}^G$, a map $X\to Y$ in $\mathsf{PreOp}^G$ will have the right lifting property against all normal monomorphisms iff so does 
$E_{\infty}\times (X\to Y)$, so that one is free to assume that $X,Y$ are normal.

One is thus free to pick a section $s\colon Y \to X$
of $p\colon X\to Y$ and,
regarding $J \in \mathsf{dSet}^G$ as a simplicially constant object of $\mathsf{sdSet}^G$,
our assumption yields the lift below, showing that $p$ is a homotopy equivalence (the implicit claim that $X \otimes J$ is a cylinder object follows from \cite[Prop. 7.25]{Per18}, which implies $X \amalg X \to X \otimes J$ is a normal monomorphism, and \cite[Thm. 7.1]{Per18}, which implies the inclusions $X \to X \otimes J$ are $G$-inner anodyne on each simplicial level, thus dendroidal equivalences and hence joint equivalences).
\begin{equation}
\begin{tikzcd}[column sep = 4em]
	X \amalg X \ar{r}{(id_X,sp)} \ar{d} &
	X \ar{d}{p}
\\
	X \otimes J \ar[dashed]{ru} \ar{r} & Y
\end{tikzcd}
\end{equation}
\end{proof}

\begin{theorem}\label{PREOPMOD THM}
	The category $\mathsf{PreOp}^G$ of $G$-preoperads has a model structure such that
	\begin{itemize}
		\item the cofibrations are the normal monomorphisms;
		\item the weak equivalences are the maps 
		that become Rezk/joint equivalences when regarded as maps in 
		$\mathsf{sdSet}^G$.
	\end{itemize}
\end{theorem}

\begin{proof}
We repeat the proof of the non-equivariant analogue \cite[Thm. 8.13]{CM13a}, applying J. Smith's theorem \cite[Thm. 1.7]{Bek00} with the required set of generating cofibrations the 
set $\{\emptyset \to G/H\cdot \eta | H \leq G\} \cup \gamma_! (\mathcal{I}')$ given by Lemma \ref{GENSET LEM}.
Indeed, conditions c0 and c2 in \cite{Bek00} are inherited from 
$\mathsf{sdSet}^G$ and c1 follows from Lemma \ref{TRIVFIB LEM}.
The technical ``solution set'' condition c3 follows from 
\cite[Prop. 1.15]{Bek00} since weak equivalences are accessible, being the preimage by $\gamma^{\**}$ of the weak equivalences in 
$\mathsf{sdSet}^G$ 
(see \cite[Cor. A.2.6.5]{Lur09} and \cite[Cor. A.2.6.6]{Lur09}). 
\end{proof}

\begin{definition}
      \label{SEGALOP_DEF}
	Following \cite[Def. 5.5]{CM13b} (compare with \cite[Def. 8.1]{CM13a}), an equivariant pre-operad $X \in \mathsf{PreOp}^G$ is called an \textit{equivariant Segal operad} if the maps
	\[
	\gamma^{\**} X (T) \to \gamma^{\**} X (Sc(T))
	\]
are Kan equivalences in $\mathsf{sSet}$ for every $T \in \Omega_G$.
\end{definition}

\begin{remark}
      \label{SEGALOP_REM}
Combining Corollary \ref{FIB_PREOP_COR} with the characterization in Proposition \ref{DSSCHAR_PROP}
one obtains that the fibrant objects of $\mathsf{PreOp}^G$ are the equivariant Segal operads $X$ such that
$\gamma^{\**} X$ is dendroidal Reedy fibrant.
\end{remark}

\begin{theorem}\label{ANOQUEQUIV THM}
The adjunction
\[
	\gamma^{\**} \colon \mathsf{PreOp}^G	
\rightleftarrows
	\mathsf{sdSet}^G \colon \gamma_{\**}
\]
is a Quillen equivalence.
\end{theorem}

\begin{proof}
	It is tautological that the left adjoint $\gamma^{\**}$
	preserves and detects cofibrations and weak equivalences,
	so it suffices to show that for all fibrant
	$X \in \mathsf{sdSet}^G$
	the counit map 
	$\gamma^{\**} \gamma_{\**} X \to X$
	is a weak equivalence. 
	But by
	Proposition \ref{CSKETALT PROP}(i) both
	$\gamma^{\**} \gamma_{\**} X$ and $X$
	are dendroidal fibrant, 
	so that the result follows from 
	Corollary \ref{SDSETG COR}(iii) together with the observation that $\left(\gamma^{\**} \gamma_{\**} X\right)_0 = X_0$.
\end{proof}

\section{Equivariant dendroidal Segal spaces}
\label{EDSS_SEC}

As outlined in the introduction, one of the main aims of our overall project is to show that the model structures on 
$\mathsf{sdSet}^G$ and $\mathsf{PreOp}^G$
defined in \S \ref{CEDSS SEC} and \S \ref{PREOP SEC}
are Quillen\ equivalent to a suitable model structure on the category 
$\mathsf{sOp}^G$ of (colored) $G$-operads.
However, our present description of the weak equivalences in 
$\mathsf{sdSet}^G$ and $\mathsf{PreOp}^G$
is rather different from the description of the desired weak equivalences in $\mathsf{sOp}^G$,
which are the Dwyer-Kan equivalences, 
characterized by fully faithfulness and essential surjectivity requirements.

As such, our goal in this final main section is to prove
Theorem \ref{COMPIFFDK THM},
which states that weak equivalences between fibrant objects in either of 
$\mathsf{sdSet}^G$, $\mathsf{PreOp}^G$
do indeed admit a Dwyer-Kan type description.
Moreover, in Corollary \ref{FIB_PREOP_COR} we also characterize the fibrant objects
of $\mathsf{PreOp}^G$
(this independently extends a result that first appeared in Bergner's work \cite{Ber07}).

To do so, it is useful to consider yet another model structure on the category $\mathsf{sdSet}^G$,
whose fibrant objects are the so called
\textit{equivariant dendroidal Segal spaces},
and which ``interpolate''
between the fibrant objects 
in the categories $\mathsf{sdSet}^G$ and $\mathsf{PreOp}^G$ (see Remark \ref{INTERP REM} for a precise statement).

\subsection{The homotopy genuine operad and Dwyer-Kan equivalences}\label{HMPTYGEN SEC}

\begin{definition}
      \label{HMPTYGEN DEF}
	The \textit{equivariant Segal space model structure} on the category $\mathsf{sdSet}^G$, which we denote 
	$\mathsf{sdSet}^G_S$, 
	is the left Bousfield localization of the dendroidal Reedy model structure with respect to the equivariant Segal core inclusions 
	(see Remark \ref{RECOVDEF REM}, more specifically \eqref{RECOVDEN REM})
\[
	Sc[T] \to \Omega[T], \qquad T \in \Omega_G.
\]
\end{definition}

\begin{notation}\label{FIB_PREOP_NOT}
We will refer to the fibrant objects of
$\mathsf{sdSet}^G_S$
as \textit{equivariant dendroidal Segal spaces}, 
or just \textit{dendroidal Segal spaces}.
Further, a pre-operad $X \in \mathsf{PreOp}^G$ is called \textit{fibrant}
if $\gamma^{\**}X$ is a dendroidal Segal space
(for now this is just terminology,
foreshadowing Corollary \ref{FIB_PREOP_COR}).
\end{notation}

The following is the equivariant analogue of \cite[Cor. 5.6]{CM13a}, and is an immediate consequence of Proposition \ref{HYPER PROP} and Remark \ref{MAPSPACE REM}.
\begin{proposition}
      \label{DSSCHAR_PROP}
      Suppose $X \in \mathsf{sdSet}^G$ is dendroidal Reedy fibrant. The following are equivalent:
      \begin{enumerate}[label = (\roman*)]
      \item $X$ is an equivariant dendroidal Segal space;
      \item $X(\Omega[T]) \to X(Sc[T])$ is a trivial Kan fibration for all $T \in \Omega_G$;
	\item $X(\Omega[T]) \to X(\Lambda^{Ge}[T])$ is a trivial Kan fibration for all $T \in \Omega_G$, $e \in \boldsymbol E^{\mathsf{i}}(T)$;
      \item $X(\Omega[T]) \to X(\Lambda^E[T])$ is a trivial Kan fibration for all $T \in \Omega_G$, $G$-subsets $\varnothing \neq E \subseteq \boldsymbol E^{\mathsf{i}}(T)$;
      \item $X(\Omega[T]) \to X(\Lambda_o^E[T])$ is a trivial Kan fibration for all $T \in \Omega_G$, $G$-subsets $\varnothing \neq E \subseteq \boldsymbol E^{\mathsf{i}}(T)$.
      \end{enumerate}
\end{proposition}

\begin{proposition}\label{DSSFIB_PROP}
If $X \in \mathsf{sdSet}^G$ is a dendroidal Segal space, then
$\gamma_{\**}X \in \mathsf{PreOp}^G$ is fibrant.
\end{proposition}

\begin{proof}
	By Proposition \ref{CSKETALT PROP}(i) $\gamma_{\**}X$ is dendroidal fibrant. And, since 
	$Sc[T](\eta)=\Omega[T](\eta)$,
	Proposition \ref{CSKETALT PROP}(iii) shows that 
	$(\gamma^{\**}\gamma_{\**}X)(\Omega[T]) \to 
	(\gamma^{\**} \gamma_{\**}X)(Sc[T])$
	is a trivial Kan fibration.
\end{proof}


To define fully faithfulness for dendroidal Segal spaces, we now
discuss mapping spaces. Non-equivariantly, these are indexed by tuples of objects and described using corollas \cite[3.6]{CM13a}.
Equivariantly, these are indexed by ``equivariant tuples'' of objects, as suggested by $G$-corollas.


\begin{notation}\label{GCOR NOT}
Given subgroups $H_i \leq G$, $0\leq i \leq k$ such that
$H_0 \geq H_i$, $1 \leq i \leq k$ we write
$C_{\amalg_i H_0/H_i}$ for the $G$-corolla (well defined up to isomorphism)
whose orbital representation is
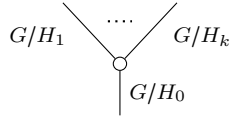
\[
\begin{tikzpicture}
[grow=up,auto,level distance=2.3em,every node/.style = {font=\footnotesize},dummy/.style={circle,draw,inner sep=0pt,minimum size=1.75mm}]
	\node at (0,0) {}
		child{node [dummy] {}
			child{
			edge from parent node [swap,near end] {$G/H_k$} node [name=Kn] {}}
			child{
			edge from parent node [near end] {$G/H_1$}
node [name=Kone,swap] {}}
		edge from parent node [swap] {$G/H_0$}
		};
		\draw [dotted,thick] (Kone) -- (Kn) ;
\end{tikzpicture}
\]
Writing $C_n$ for the non-equivariant corolla with $n$ leaves, we note that
$C_{\amalg_i H_0/H_i} \simeq 
G \cdot_{H_0} C_{\Sigma_i |H_0/H_i|}$,
where $C_{\Sigma_i |H_0/H_i|}$ is regarded as a (non-equivariant) corolla together with the obvious $H_0$-action.
\end{notation}

\begin{definition}\label{PROF DEF}
	Let $X\in \mathsf{dSet}^G$ be a $G$-$\infty$-operad.
	A \textit{$G$-profile on $X$} is a map
\[
	\partial \Omega[C] \to X
\]
	for some $G$-corolla $C$. 
	More explicitly, a $G$-profile is described by the following data:
	\begin{itemize}
	\item subgroups $H_i \leq G$, $0\leq i \leq k$ such that
		$H_0 \geq H_i$ for $1 \leq i \leq k$;
	\item objects $x_i \in X(\eta)^{H_i}$ for $0 \leq i \leq k$.
	\end{itemize}
	To simplify notation, we denote a $G$-profile as 
	$(x_1,\cdots,x_k;x_0)$, and refer to it as a 
	\textit{$C$-profile on $X$}.
Further, for $X \in \mathsf{sdSet}^G_S$ a dendroidal Segal space we define a \textit{$C$-profile} on $X$ 
as a $C$-profile on $X_0$.
\end{definition}

\begin{definition}\label{MAPSPACESEG DEF}
Given a dendroidal Segal space $X \in \mathsf{sdSet}^G_S$
and a $C$-profile $(x_1,\cdots,x_k ; x_0)$
on $X$ 
we define the space of maps 
$X(x_1,\cdots,x_k ; x_0) \in \mathsf{sSet}$ via the pullback square
\[
\begin{tikzcd}[column sep=4em]
	X(x_1,\cdots,x_k;x_0) \ar{r} \ar[->>]{d}&
	X(\Omega[C]) \ar[->>]{d}
\\
	\Delta[0] \ar{r}[swap]{(x_1,\cdots,x_k;x_0)} &
	\prod_{0\leq i \leq k} X(\eta)^{H_i}
	\arrow[lu, phantom, "\lrcorner", very near start]
\end{tikzcd}
\]
\end{definition}

To discuss essential surjectivity, and adapting \cite[8.8]{CM13a},
we associate to each equivariant dendroidal Segal space $X$ a discretized algebraic structure $ho(X)$, which in the equivariant setting is a genuine equivariant operad (cf. Definition \ref{GEN_OP_DEF}).

\begin{definition}\label{HMTPYGEN DEF}
	Let $X \in \mathsf{sdSet}^G$ be a dendroidal Segal space.
	The \textit{homotopy genuine operad} 
	$ho(X)\in \mathsf{dSet}_G$ is defined
by
	\[
	ho(X) = \pi_0 \left( \upsilon_{\**} \left( \gamma_{\**}X \right) \right),
    \]
where $\upsilon_{\**}$  and $\gamma_{\**}$ are defined in \eqref{DSETG_EQ} and \eqref{GAMMASTAR_EQ}.
\end{definition}

\begin{proposition}
      \label{HMTPYGEN PROP}
For any dendroidal Segal space $X \in \mathsf{sdSet}^G$ one has that 
$ho(X) \in \mathsf{dSet}_G$ is a genuine equivariant operad. 
\end{proposition}

\begin{proof}
By Proposition \ref{DSSFIB_PROP} we are free to assume 
$X$ is in $\mathsf{PreOp}^G \subset \mathsf{sdSet}^G$, so we suppress $\gamma_{\**}$ from the notation throughout.

The required strict lifting condition is equivalent to the 
 maps
$\mathsf{dSet}_G\left(\upsilon_{\**}\Omega[T], ho(X)\right)
\to
\mathsf{dSet}_G\left(\upsilon_{\**}Sc[T], ho(X)\right)$
being isomorphisms.
On the other hand, the Segal condition for $X$ says that the maps 
$X(\Omega[T]) \to X(Sc[T])$
are trivial Kan fibrations, so that the maps 
$\pi_0\left(X(\Omega[T])\right) \to \pi_0\left(X(Sc[T])\right)$
are isomorphisms. 
Hence, noting that
for $A \in \mathsf{dSet}^G$
there is a natural transformation
(letting $\bullet$ highlight the index with regards to which the $\pi_0$ are computed, and using that $\upsilon_{\**}$ is fully faithful)
\begin{equation}\label{HOGENEQOP EQ}
\begin{split}
	\pi_0\left( X_{\bullet}(A) \right)=
	\pi_0 \left(
	\mathsf{dSet}^G(A,X_{\bullet})
	\right)
= &
	\pi_0 \left(
	\mathsf{dSet}_G(\upsilon_{\**} A,\upsilon_{\**} X_{\bullet})
	\right)
\to 
\\
\to &
	\mathsf{dSet}_G(\upsilon_{\**} A,\pi_0 \upsilon_{\**} X_{\bullet})
=
	\mathsf{dSet}_G(\upsilon_{\**} A,ho(X_{\bullet}))
\end{split}
\end{equation}
it is enough to show that \eqref{HOGENEQOP EQ}
is an isomorphism when $A$ is of the form $\Omega[T]$ or $Sc[T]$.
The case of $A=\Omega[T]$
is tautological since, by Remark \ref{TWOYON REM},
$\mathsf{dSet}_G(\upsilon_{\**}\Omega[T],Y) = Y(T)$
for any $Y \in \mathsf{dSet}_G$. 

It remains to tackle the case $A=Sc[T]$.
We argue by induction on the number of $G$-vertices of $T$.
The base cases of $T$ either a $G$-corolla or a single $G$-edge orbit are automatic since then $Sc[T] = \Omega[T]$.
Otherwise, choosing an equivariant grafting decomposition
$T = R \amalg_{Ge} S$ (cf. \cite[(5.18) and Prop. 6.19]{Per18})
one has a pushout decomposition
$Sc[T] \simeq Sc[R] \amalg_{\Omega[Ge]} Sc[S]$
so that, by induction, it suffices to check that both sides of 
\eqref{HOGENEQOP EQ} turn this pushout into a pullback.
For the right side this follows since by 
Remark \ref{UPSPUSHMON REM}
one has the analogue decomposition
$\upsilon_{\**} Sc[T] \simeq \upsilon_{\**} Sc[R] \amalg_{\upsilon_{\**}\Omega[Ge]} \upsilon_{\**} Sc[S]$.
For the left side this follows since $\pi_0$
preserves the pullback decomposition
$X(Sc[T]) \simeq X(Sc[R]) \times_{X(\Omega[Ge])} X(Sc[S])$ due to
$X(\Omega[Ge])$ being discrete.
\end{proof}

Recall that $\mathsf{O}_G$ denotes the \textit{orbit category} with objects the $G$-sets $G/H$ for $H \leq G$ and arrows the $G$-equivariant maps between them. In the following, we use the natural inclusion
$\Omega \times \mathsf{O}_G \to \Omega_G$
given by 
$(T,G/H) \mapsto G/H \cdot T$.

\begin{remark}
      Writing $\iota \colon \Delta \to \Omega$ and
      $\iota_G \colon \Delta \times \mathsf{O}_G \to \Omega \times \mathsf{O}_G \to \Omega_G$
      for the (composite) inclusions,
        unpacking definitions shows that $\iota_G^{\**}ho(X)$ is the $G$-coefficient system of nerves of categories $\left(\iota_G^{\**}ho(X)\right)(G/H)$ for $H \leq G$,
        which are the simplicial sets with $m$-simplices
	\begin{equation}\label{HOCOEFF EQ}
              ho(X)\left(G/H \cdot [m]\right) =
              \pi_0 \left( \left(\gamma_{\**}X([m])\right)^H\right) = 
              \pi_0 \left( \iota^{\**} \gamma_{\**}
              \left(X^H\right)\right)(m) = 
              ho\left(\iota^{\**}\left(X^H\right)\right)(m),
        \end{equation}
        where the second identity follows from
        the discussion following \eqref{DSETG_EQ},
        and the second $ho$ is the analogue of Definition \ref{HMTPYGEN DEF} for simplicial Segal spaces \cite[\S 5.5]{Rez01}.
\end{remark}

\begin{definition}\label{DKEQUIV DEF}
	A map $f \colon X \to Y$ of equivariant dendroidal Segal spaces is called 
\begin{itemize}
	\item \textit{fully faithful} if for all $G$-corollas $C$ and $C$-profiles $(x_1,\cdots,x_n;x_0)$ on $X$ the maps
\[
	X(x_1,\cdots,x_k;x_0) \to
	Y\left(f(x_1),\cdots,f(x_k);f(x_0)\right)
\]
are Kan equivalences in $\mathsf{sSet}$;
	\item \textit{essentially surjective} if
              the map $\iota_G^{\**}ho(X) \to \iota_G^{\**}ho(Y)$
              of $G$-coefficient systems of categories is levelwise essentially surjective;
	\item a \textit{DK-equivalence} if it is both fully faithful and essentially surjective.
\end{itemize}
\end{definition}

\begin{remark}\label{ONLYPREOP REM}
Definitions \ref{MAPSPACESEG DEF}, \ref{HMTPYGEN DEF} and \ref{DKEQUIV DEF} depend only on the 
fibrant pre-operads $\gamma_{\**}X,\gamma_{\**}Y$,
since $X(x_1,\cdots,x_k;x_0) = \gamma_{\**}X(x_1,\cdots,x_k;x_0)$.
In fact, for each $G$-corolla $C$
one has a decomposition
\[
	ho(X)(C)=
	\coprod_{\text{$C$-profiles }(x_1,\cdots,x_k;x_0)}
	\pi_0 \left( X(x_1,\cdots,x_k;x_0) \right)
\]
so that, given $\varphi \in X_0(x_1,\cdots,x_k;x_0)$
we will write $[\varphi] \in ho(X)(C)$
for the corresponding class.
\end{remark}

\begin{remark}
	One can extend the previous definitions to $G$-$\infty$-operads $X,Y \in \mathsf{dSet}^G$
	by applying them to the dendroidal Segal spaces
	$X^{J^{\bullet}},Y^{J^{\bullet}} \in \mathsf{sdSet}^G$
	(cf. Remark \ref{CONCRECOM REM}). 
\end{remark}

\begin{remark}\label{DKCOM REM} 
Given a map $X \to Y$ between dendroidal Segal spaces and $T \in \Omega_G$, consider the diagram
\begin{equation}\label{DKCOM EQ}
\begin{tikzcd}[row sep=17]
	X(\Omega[T]) \ar{r} \ar[->>]{d}[swap]{\sim}&
	Y(\Omega[T]) \ar[->>]{d}{\sim}
\\
	X(Sc[T]) \ar{r} \ar[->>]{d}&
	Y(Sc[T]) \ar[->>]{d}
\\
	\underset{[e_i] \in \boldsymbol{E}_G(T)} {\prod} \left(X(\eta)\right)^{H_i} \ar{r} &
	\underset{[e_i] \in \boldsymbol{E}_G(T)} {\prod} \left(Y(\eta)\right)^{H_i}
\end{tikzcd}
\end{equation}
where $H_i \leq G$ is the isotropy of $e_i \in \boldsymbol{E}(T)$.
Since the fibers of the bottom vertical maps are products of spaces of maps for $G$-profiles and the top vertical maps are trivial fibrations by the Segal conditions, 
$X \to Y$ is fully faithful iff \eqref{DKCOM EQ} 
induces Kan equivalences between the fibers of the vertical composite maps 
(for the ``if'' claim, let 
 $T$ be a $G$-corolla, cf. \cite[Prop. 5.7]{CM13a}).
%
\end{remark}

This remark readily implies the following, which is the analogue of  
\cite[Cor. 5.10]{CM13a}.

\begin{corollary}\label{DKCOM COR} 
Let $f\colon X \to Y$ in $\mathsf{sdSet}_G$ be a map between dendroidal Segal spaces. Then:
\begin{enumerate}[label=(\roman*)]
\item if $f$ is a simplicial equivalence then $f$ is fully faithful;
\item if $f$ is fully faithful, then $f$ is also a simplicial equivalence iff
the maps $X(\eta)^H \to Y(\eta)^H, H\leq G$, are Kan equivalences.
\end{enumerate}
\end{corollary}

\begin{remark}\label{HOMOLIFTS REM}
	In what follows, we will repeatedly use the observation that, for $X\to Y$ a trivial Kan fibration
	in $\mathsf{sSet}$, any two lifts  of the form below are homotopic.
\[
\begin{tikzcd}[row sep =1.5em]
	&
	X \ar[->>]{d}{\sim}
\\
	A \ar[dashed]{ru} \ar{r} &
	Y 
\end{tikzcd}
\]
\end{remark}

\begin{definition}\label{HEQUIV DEF}
	Let $X\in \mathsf{sdSet}^G$ be a dendroidal Segal space.

	For $H \leq G$, we call 
	$f \in X_0(\Omega[C_{H/H}]) = X_0([1])^H$ a 
	\textit{$H$-equivalence} 
	if $[f]$ is an isomorphism in the category
	$\iota_G^{\**} ho(X)(G/H) = ho\left(\iota^{\**}\left( X^H \right)\right)$.
\end{definition}

In what follows, and in analogy to \cite[\S 11.2]{Rez01},
we will need to understand the interaction between the homotopy genuine operad $ho(X)$ and the mapping spaces
$X(x_1,\cdots,x_n;x_0)$.

Suppose $C,D$ are $G$-corollas that can be grafted,
i.e. that $C$ has a leaf orbit and $D$ a root orbit both isomorphic to $G/H$. Denote this orbit as $G e$
and write $T= C \amalg_{G e} D$ for the grafted $G$-tree 
\cite[
(5.18) and Prop. 6.19]{Per18}. 
For any dendroidal Segal space $X$ one then has
$X(Sc[T]) \simeq X(\Omega[C]) \times_{X(\eta)^H} X(\Omega[D])$
and one can hence choose a section in the middle row below
\begin{equation}\label{HOMOTCIRC EQ}
\begin{tikzcd}
	\{\varphi\} \times X(z_1,\cdots,z_l;e)
	\ar[dashed]{rr}{\varphi \circ_{Ge} (-)}
&&
	X(z_1,\cdots,z_l,y_2,\cdots,y_k;x)
\\
	X(\Omega[C]) \times_{X(\eta)^H} X(\Omega[D]) \ar[bend left=17,dashed]{r}
	\ar[hookleftarrow]{u}
&
	X(\Omega[T]) \ar[->>]{l}{\sim} \ar[->>]{r}
&
	X(\Omega[T-Ge])
	\ar[hookleftarrow]{u}
\\
	X(e,y_2,\cdots,y_k;x) \times \{\psi\}
	\ar[hookrightarrow]{u}
	\ar[dashed]{rr}[swap]{(-)\circ_{Ge} \psi}
&&
	X(z_1,\cdots,z_l,y_2,\cdots,y_k;x)
	\ar[hookrightarrow]{u}
\end{tikzcd}
\end{equation}
thus defining maps 
$\varphi \circ_{Ge} (-)$ (resp. $(-)\circ_{Ge} \psi$)
for any choice of 
$\varphi \in X_0(e,y_2,\cdots,y_k;x)$
(resp. $\psi \in X_0(z_1,\cdots,z_l;e)$).

\begin{proposition}\label{GENOPHO PROP}
\begin{itemize}
	\item[(i)] the maps $\varphi \circ_{Ge} (-)$, $(-)\circ_{Ge} \psi$
are well defined up to homotopy;
	\item[(ii)] if $[\varphi]=[\bar{\varphi}]$ then 
the maps $\varphi \circ_{Ge} (-)$, $\bar{\varphi} \circ_{Ge} (-)$ are homotopic, and likewise for $[\psi] = [\bar{\psi}]$;
	\item[(iii)] $[\varphi \circ_{Ge} \psi]$
	depends only on $[\varphi]$, $[\psi]$;
	\item[(iv)] the homotopy classes of the maps $\varphi \circ_{Ge} (-)$, $(-)\circ_{Ge} \psi$ are natural with respect to maps $f\colon X \to Y$ between dendroidal Segal spaces.
\end{itemize}
\end{proposition}

\begin{proof}
	Noting that all possible middle row sections in \eqref{HOMOTCIRC EQ} (and homotopies between them)
	are necessarily compatible with the projections to 
$X(\partial \Omega[T-Ge])$, 
	(i) follows from Remark \ref{HOMOLIFTS REM}.
	The middle row in \eqref{HOMOTCIRC EQ}
	gives the necessary homotopies for (ii). 
	(iii) is immediate from (ii).
Lastly, (iv) follows from Remark \ref{HOMOLIFTS REM}
applied to the two
diagonal $\nearrow$
paths in
\[
\begin{tikzcd}[column sep =13]
	X(\Omega[T]) \ar{r} \ar[->>]{d}{\sim} &
	Y(\Omega[T]) \ar[->>]{d}{\sim}
\\
	X(\Omega[C]) \underset{X(\eta)^H}{\times} X(\Omega[D])
	\ar[bend left=30,dashed]{u} \ar{r}
&
	Y(\Omega[C]) \underset{Y(\eta)^H}{\times} Y(\Omega[D])
	\ar[bend left=30,dashed]{u}
\end{tikzcd}
\]
\end{proof}

We will now show that the operations   
$\varphi \circ_{Ge} (-)$, $(-)\circ_{Ge} \psi$
satisfy the obvious compatibilities one expects, but we will find it convenient to first package these compatibilities into a common format. In the categorical case (corresponding to linear trees), there are three types of ``associativity'' compatibilities, corresponding to homotopies
\[
	\varphi \circ \left( \psi \circ (-) \right)
		\sim
	(\varphi \circ \psi) \circ (-) 
\qquad
	\varphi \circ \left( (-)  \circ \psi \right)
		\sim
	\left( \varphi \circ (-) \right) \circ \psi 
\qquad
	\left( (-) \circ \varphi \right) \circ \psi 
		\sim
	(-) \circ (\varphi \circ \psi) 
\]	
but in the operadic case there are instead five cases, corresponding to the different possible roles of the nodes in 
$G$-trees $T$ with exactly three $G$-vertices, whose \textit{orbital} representation falls into one of the two cases illustrated below.
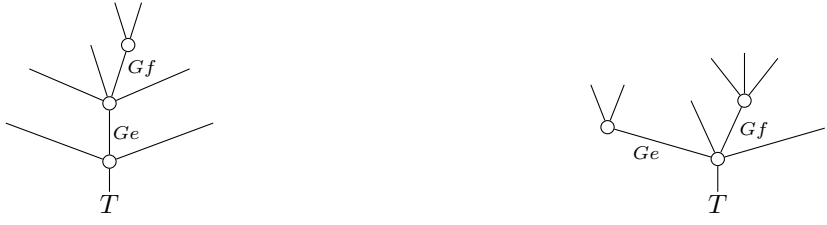
\[%
	\begin{tikzpicture}[auto,grow=up, every node/.style = {font=\scriptsize,inner sep=1pt},
	dummy/.style = {circle,draw,inner sep=0pt,minimum size=1.75mm}]%
	\begin{scope}[level distance = 1.6em]
	\tikzstyle{level 2}=[sibling distance=4em]%
	\tikzstyle{level 3}=[sibling distance=2em]%
	\tikzstyle{level 4}=[sibling distance=1em]%
		\node at (0,0) [font=\normalsize]{$T$}%
			child{node [dummy] {}%
				child[level distance = 1.5em]{node {}}%
				child[level distance = 2.2em]{node [dummy] {}%
					child[level distance = 1.3em]{}%
					child[level distance = 2.2em,sibling distance=1.4em]{node [dummy] {}
						child[level distance = 1.6em]
						child[level distance = 1.6em]
					edge from parent node [swap,very near end] {$Gf$}}%
					child[level distance = 2.2em,sibling distance=1.4em]{}%
					child[level distance = 1.3em]{}%
				edge from parent node [swap]{$Ge$}}%
				child[level distance = 1.5em]{node {}}%
			};%
	\end{scope}
	\begin{scope}[level distance = 1.7em]
	\tikzstyle{level 2}=[sibling distance=2.75em]%
	\tikzstyle{level 3}=[sibling distance=1.25em]%
		\node at (8,0) [font=\normalsize] {$T$}%
			child{node [dummy] {}%
				child[level distance = 1.2em]{node {}}%
				child[sibling distance=2em,level distance = 2.2em]{node [dummy] {}%
					child[level distance = 1.6em]{}%
					child[level distance = 1.8em]{}%
					child[level distance = 1.6em]{}%
				edge from parent node [swap,near end] {$Gf$}}%
				child[sibling distance=2em,level distance = 2.2em]{}%
				child[level distance = 1.2em]{node [dummy] {}%
					child[level distance = 1.6em]{}%
					child[level distance = 1.6em]{}%
				edge from parent node {$Ge$}}%
			};%
	\end{scope}
	\end{tikzpicture}%
\]%
Since all these compatibilities can be simultaneously encoded in terms of such trees, we will refer to all types of compatibility simply as \textit{associativity}.
As noted pictorially above, such a $G$-tree $T$
has exactly two inner edge orbits $Ge$ and $Gf$.
In the next result, we write $T[Ge]$ (resp. $T[Gf]$) for the orbital outer face of $T$ with $Ge$ (resp. $Gf$) as its single inner edge orbit. 
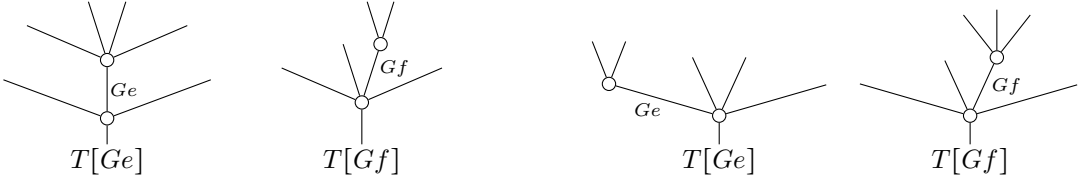
\[%
	\begin{tikzpicture}[auto,grow=up, every node/.style = {font=\scriptsize,inner sep=1pt},
	dummy/.style = {circle,draw,inner sep=0pt,minimum size=1.75mm}]%
	\begin{scope}[level distance = 1.6em]
	\tikzstyle{level 2}=[sibling distance=4em]%
	\tikzstyle{level 3}=[sibling distance=2em]%
	\tikzstyle{level 4}=[sibling distance=1em]%
		\node at (-1.675,0) [font=\normalsize]{$T[Ge]$}%
			child{node [dummy] {}%
				child[level distance = 1.5em]{node {}}%
				child[level distance = 2.2em]{node [dummy] {}%
					child[level distance = 1.3em]{}%
					child[level distance = 2.2em,sibling distance=1.4em]{
					}%
					child[level distance = 2.2em,sibling distance=1.4em]{}%
					child[level distance = 1.3em]{}%
				edge from parent node [swap]{$Ge$}}%
				child[level distance = 1.5em]{node {}}%
			};%
	\end{scope}
	\begin{scope}[level distance = 1.6em]
	\tikzstyle{level 2}=[sibling distance=2em]%
	\tikzstyle{level 3}=[sibling distance=1em]%
		\node at (1.675,0) [font=\normalsize]{$T[Gf]$}%
			child[level distance = 2.2em]{node [dummy] {}%
				child[level distance = 1.3em]{}%
				child[level distance = 2.2em,sibling distance=1.4em]{node [dummy] {}
					child[level distance = 1.6em]
					child[level distance = 1.6em]
				edge from parent node [swap,very near end] {$Gf$}}%
				child[level distance = 2.2em,sibling distance=1.4em]{}%
				child[level distance = 1.3em]{}%
			};%
	\end{scope}
	\begin{scope}[level distance = 1.7em]
	\tikzstyle{level 2}=[sibling distance=2.75em]%
	\tikzstyle{level 3}=[sibling distance=1.25em]%
		\node at (6.375,0) [font=\normalsize] {$T[Ge]$}%
			child{node [dummy] {}%
				child[level distance = 1.2em]{node {}}%
				child[sibling distance=2em,level distance = 2.2em]{
				}%
				child[sibling distance=2em,level distance = 2.2em]{}%
				child[level distance = 1.2em]{node [dummy] {}%
					child[level distance = 1.6em]{}%
					child[level distance = 1.6em]{}%
				edge from parent node {$Ge$}}%
			};%
		\node at (9.675,0) [font=\normalsize] {$T[Gf]$}%
			child{node [dummy] {}%
				child[level distance = 1.2em]{node {}}%
				child[sibling distance=2em,level distance = 2.2em]{node [dummy] {}%
					child[level distance = 1.6em]{}%
					child[level distance = 1.8em]{}%
					child[level distance = 1.6em]{}%
				edge from parent node [swap,near end] {$Gf$}}%
				child[sibling distance=2em,level distance = 2.2em]{node {}}%
				child[level distance = 1.2em]{
				}%
			};%
	\end{scope}
	\end{tikzpicture}%
\]%

\begin{proposition}\label{ASSOC PROP}
	The operations
	$\varphi \circ_{Ge} (-)$, $(-)\circ_{Ge} \psi$
	satisfy all associativity conditions with respect to 
	$G$-trees with three $G$-vertices.
	Further, if $C=C_{H/H}$ and $\varphi = s(e)$ is the degeneracy on $e$, then $\varphi \circ_{Ge} (-)$ is homotopic to the identity, and similarly for 
	$D=C_{H/H}$ and $\psi = s(e)$.
\end{proposition}

\begin{proof}
We abbreviate
$Sc_{T[Ge]}[T] =
Sc[T] \amalg_{Sc[T[Ge]]} \Omega[T[Ge]] =
Sc[T] \amalg_{\Lambda^{Ge}_o[T[Ge]]} \Omega[T[Ge]]$,
which can be regarded as the union
$Sc[T] \cup \Omega[T[Ge]]$
of subcomplexes of $\Omega[T]$.
We now consider the following diagram, 
where all solid maps are Kan fibrations, 
and the maps labelled $\sim$ are trivial Kan fibrations
($Sc_{T[Ge]}[T]$ is a cover in the sense of Remark \ref{RECOVER REM}(i),
hence both maps $Sc[T] \to Sc_{T[Ge]}[T] \to \Omega[T]$ are $G$-inner anodyne), so that one can choose the indicated sections.
\begin{equation}\label{FOURSQ EQ}
\begin{tikzcd}[column sep = 2.5em,row sep = 2.5em]
	X(Sc[T]) 
	\ar[bend left=17,dashed]{r}
	\ar[bend right=30,dashed]{d}&
	X(Sc_{T[Ge]}[T])
	 \ar[->>]{l}{\sim} \ar[->>]{r} 
	\ar[bend right=30,dashed]{d}&
	X(Sc[T-Ge])
	\ar[bend right=30,dashed]{d}
\\
	X(Sc_{T[Gf]}[T]) \ar[->>]{u}[swap]{\sim} \ar[->>]{d}
	\ar[bend left=17,dashed]{r}& 
	X(\Omega[T]) \ar[->>]{l}{\sim} \ar[->>]{u}[swap]{\sim} \ar[->>]{r} 
	\ar[->>]{d}&
	X(\Omega[T-Ge]) \ar[->>]{u}[swap]{\sim} \ar[->>]{d}
\\
	X(Sc[T-Gf]) 
	\ar[bend left=17,dashed]{r}&
	X(\Omega[T-Gf]) \ar[->>]{r} \ar[->>]{l}{\sim}&
	X(\Omega[T-Ge-Gf])
\end{tikzcd}
\end{equation}
Since the desired associativity conditions now amount to the claim that the top right and left bottom composites 
$X(Sc[T]) \to X(\Omega[T-Ge-Gf])$
are homotopic, the associativity result follows from Remark \ref{HOMOLIFTS REM}.
For the ``further'' claim, note that by Remark \ref{HOMOLIFTS REM} one is free to modify \eqref{HOMOTCIRC EQ} so as to use any lift of the form below.
But then since the $G$-tree $T$ is degenerate on the $G$-corolla $D$, such a lift is given by the degeneracy operator and the result follows.
\[
\begin{tikzcd}[row sep =1.5em]
&
	X(\Omega[T]) \ar[->>]{d}{\sim}
\\
	\{s(e)\} \times X(z_1,\cdots,z_l;e) \ar[dashed]{ru} \ar{r}
&
	X(Sc[T])
\end{tikzcd}
\]
\end{proof}

\begin{remark}
	In the non-equivariant case the associativity and unit conditions in the previous result capture all the key compatibilities of the
	$\varphi \circ_{e} (-)$, $(-)\circ_{e} \psi$
	operations.
	However, in the equivariant case there are further 
	``compatibilities with quotients of $G$-trees'',
	which reflect the remarks in 
	Example \ref{STRICTLIFT EX}.
	Nonetheless, describing these extra compatibilities would require using $G$-trees with more than three $G$-vertices, and since such compatibilities are not needed for our present goals, we omit their discussion. 
\end{remark}



\begin{corollary}\label{26COR}
DK-equivalences between dendroidal Segal spaces satisfy 2-out-of-6, i.e. when in
$X \xrightarrow{f} 
Y \xrightarrow{g}
Z \xrightarrow{h} W$ the maps
$gf$ and $hg$ are DK-equivalences then so are
$f$, $g$, $h$, $hgf$.
\end{corollary}

\begin{proof}
Applying the 2-out-of-6 properties in $\mathsf{sSet}$ and $\mathsf{Cat}$ to mapping spaces and homotopy categories $\iota_G^{\**}ho(-)$,
the only non obvious conditions are the fully faithfulness of $g,h$ for $C$-profiles not in the image of $f$. 
But since by Proposition \ref{ASSOC PROP} the maps
$j \circ_{Ge} (-)$, $(-)\circ_{Ge} j$
are weak equivalences when $j$ is a $H$-equivalence,
this last claim follows from essential surjectivity.
\end{proof}

Recall that by replacing $\mathsf{sdSet}^G$
with the simpler category $\mathsf{ssSet}$ in Definition \ref{HMPTYGEN DEF}
one recovers the Segal spaces of \cite{Rez01}.
The following roughly summarizes (and slightly refines)
\cite[Lemma 5.8, Theorem 6.2, Prop. 11.1, Lemma 11.10]{Rez01} in our setup.

Adapting Notation \ref{UNILIMDEN NOT}, for fixed $X \in \mathsf{ssSet}$ we write $X(-)\colon \mathsf{sSet}^{op} \to \mathsf{sSet}$ for the limit preserving functor such that
$X(\Delta[m]) = X(m)$.

\begin{proposition}\label{SESP PROP}
	Let $X \in \mathsf{ssSet}$ be a Segal space. Then:
\begin{itemize}
	\item[(i)] equivalences of $X$ define a subset of connected components
	$X^h(1) \subseteq X(1)$;
	\item [(ii)] the pullbacks
\begin{equation}\label{XHDEF EQ}
\begin{tikzcd}[row sep=1.5em]
	X^h(m) \ar{r} \ar{d} & X(m) \ar{d}
\\
	X^h(1) \times_{X(0)} \cdots \times_{X(0)} X^h(1) \ar{r} &
	X(1) \times_{X(0)} \cdots \times_{X(0)} X(1)
	\arrow[lu, phantom, "\lrcorner", very near start]
\end{tikzcd}
\end{equation}
define a Segal space $X^h \subseteq X$, consisting of a union of connected components at each level;
	\item[(iii)] the maps
	$X^h(2) \xrightarrow{(d_2,d_1)}
	X^h(\Lambda^0[2])$, 
	$X^h(2) \xrightarrow{(d_0,d_1)} 
	X^h(\Lambda^2[2])$
	are trivial fibrations;
	\item[(iv)] the maps $X(J^m) \to X({\Delta[m]}) = X(m)$ factor through weak equivalences 
	$X(J^m) \xrightarrow{\sim} X^h(m)$.
\end{itemize}
\end{proposition}

\begin{proof}
For (i), given 
$f \colon x \to y$ in $X_0(1)$
one has that $[f]$ has a left inverse iff there exists $p$
as on the left diagram below. But for any path $H$ between $f$ and $f'$ in $X(1)$, there is a lift in the right diagram
\begin{equation}\label{LEFTINV EQ}
\begin{tikzcd}[column sep =40]
	& X(2) \ar[->>]{d}{(d_2,d_1)}
&
	\{0\} \ar{r}{p} \ar[>->]{d}[swap]{\sim} &
	X(2) \ar[->>]{d}{(d_2,d_1)}
\\
	\{0\} \ar{r}[swap]{(f,s_0(x))} \ar{ru}{p} &
	X(1) \times_{X(0)} X(1)
&
	\Delta[1] \ar{r}[swap]{(H,s_0 d_1(H))} \ar[dashed]{ru} &
	X(1) \times_{X(0)} X(1)
\end{tikzcd}
\end{equation}
showing that $f'$ is also left-invertible. The situation for right inverses is identical, thus (i) follows.

For (ii), that $X^h$ is closed under the simplicial operators follows since equivalences are closed under composition.
Moreover, noting that \eqref{XHDEF EQ} can be reinterpreted as on the left below,
cellular induction yields the more general right pullbacks for all 
$K \in \mathsf{sSet}$.
\begin{equation}\label{XHGENDEF EQ}
\begin{tikzcd}[row sep=1.5em]
	X^h(\Delta[m]) \ar{r} \ar{d} &
	X(\Delta[m]) \ar{d}
&&
	X^h(K) \ar{r} \ar{d} &
	X(K) \ar{d}
\\
	X^h(\mathsf{sk}_1 \Delta[m]) \ar{r} &
	X(\mathsf{sk}_1 \Delta[m])
	\arrow[lu, phantom, "\lrcorner", very near start]
&&
	X^h(\mathsf{sk}_1 K) \ar{r} &
	X(\mathsf{sk}_1 K)
	\arrow[lu, phantom, "\lrcorner", very near start]
\end{tikzcd}
\end{equation}
Since 
$\mathsf{sk}_1 (\partial \Delta[m]) = 
\mathsf{sk}_1 \Delta[m]$
if $m \geq 2$
it follows that the maps
$X^h(m) \to X^h(\partial \Delta[m])$, $m \geq 2$
are Kan fibrations, and since the composite
$X^h(1) \to X(1) \to X(0) \times X(0)$
is clearly a Kan fibration, $X^h$ is indeed Reedy fibrant. 
The Segal condition for $X^h$ is obvious from the pullback \eqref{XHDEF EQ}.

For (iii), it suffices by symmetry to establish the first claim.
It is then enough to show that for any choice of section 
in the following diagram the top composite is a Kan equivalence.
\begin{equation}
\begin{tikzcd}[column sep = 60]
	X^h(1) \times_{X_0} X^h(1) \ar[bend left=17,dashed]{r} 
	\ar[->>]{rd}[swap]{(id,d_0)}&
	X^h(2) \ar[->>]{r}[swap]{(d_2,d_1)} \ar[->>]{l}{(d_2,d_0)}[swap]{\sim}&
	X^h(1) \times_{X_0} X^h(1)
	\ar[->>]{dl}{(id,d_0)}
\\
	& X^h(1) \times X(0)
\end{tikzcd}
\end{equation}
But this composite is a map of Kan fibrations over
$X^h(1) \times X(0)$ with the map between the fibers over 
$(f \colon x \to y,z)$
computing the map
$(-) \circ f \colon X^h(y;z) \to X^h(x;z)$,
which is a Kan equivalence since $f \in X_0^h(1)$ is an equivalence.
Thus the composite is a Kan equivalence, establishing (iii).

Lastly, for (iv) note first that (iii) says that $X^h$ is local with respect to the outer horn inclusions
$\Lambda^0[2] \to \Delta[2]$ and
$\Lambda^2[2] \to \Delta[2]$, 
and hence by Remarks 
\ref{ANHYPER REM} and \ref{CONTGR REM}
the maps 
$X^h(J^m) \to X^h(m)$ 
are Kan equivalences. 
It remains to show
$X^h(J^m) = X(J^m)$. We first focus on the case of $J=J^1$, for which we consider
the following diagrams, where maps out of $X(J)$ are labelled as $\partial_{\underline{a}}$ for $\underline{a}$ the associated simplex of $J$ (which is a string on $\{0,1\}$)
\[
\begin{tikzcd}[column sep = 20]
	X(J) \ar[bend left=19]{rr}{\partial_{00}} \ar{r}{\partial_{010}}
	\ar[->]{rd}[swap]{(\partial_{01},\partial_{10})}&
	X(2) \ar[->]{r}{d_1} \ar[->]{d}{(d_2,d_0)} &
	X(1) 
&
	X(J) \ar[bend left=19]{rr}{\partial_{11}} \ar{r}{\partial_{101}}
	\ar[->]{rd}[swap]{(\partial_{10},\partial_{01})}&
	X(2) \ar[->]{r}{d_1} \ar[->]{d}{(d_2,d_0)} &
	X(1) 
\\
	& X(1) \times_{X(0)} X(1) &
&
	& X(1) \times_{X(0)} X(1) &
\end{tikzcd}
\]
Since $\partial_{00},\partial_{11}$ are degenerate, for any $x \in X_0(J)$ one has that 
$\partial_{01}(x)$, $\partial_{10}(x)$ are inverse equivalences, as exhibited by $\partial_{010}(x), \partial_{101}(x)$
(cf. \eqref{LEFTINV EQ}).
This shows that $\partial_{01}, \partial_{10}$
factor through $X^h(1)$
so that
$X(J) \to X(\mathsf{sk}_1 J)$ factors through
$X^h(\mathsf{sk}_1 J)$.
But since all $1$-simplices of $J^m$ are in the image of some map $J \to J^m$, 
the map $X(J^m) \to X(\mathsf{sk}_1 J^m)$ likewise factors through
$X^h(\mathsf{sk}_1 J^m)$ for any $m$. 
The pullback \eqref{XHGENDEF EQ} now finishes the proof.
\end{proof}

\begin{remark}
The proof of (ii) shows that the inclusion $X^h \to X$ is a Reedy fibration.
\end{remark}

\begin{remark}\label{LAMBJREAL REM}
Writing $J\xrightarrow{ij}J^m$ for the map sending $0$ to $i$ and $1$ to $j$, Proposition \ref{SESP PROP}(iv) combined with the Segal condition for $X^h$ and 
Proposition \ref{SESP PROP}(iii)
shows that the maps
\[
      X\left(J^m \right) 
      \xrightarrow{\left(01,12,\cdots,(m-1)m\right)} X(J) \underset{X(0)}{\times} X(J) \underset{X(0)}{\times} \cdots \underset{X(0)}{\times} X(J),
\]
\[
      X\left(J^2\right) 
      \xrightarrow{\left(01,02 \right)} X(J) \underset{X(0)}{\times} X(J),
      \qquad \qquad
      X\left(J^2\right) 
      \xrightarrow{\left(12,02 \right)} X(J) \underset{X(0)}{\times} X(J),
\]
are trivial Kan fibrations.
\end{remark}

\subsection{Rezk completion and fibrant Segal operads}
\label{REZKCOMP SEC}

To prove the characterization of complete/joint equivalences in Theorem \ref{COMPIFFDK THM}
we will need to establish some technical properties of the completion 
$X \to \tilde{X}$
of a dendroidal Segal space $X \in \mathsf{sdSet}^G$,
which are given by Propositions \ref{JDDK PROP} and \ref{COMPLE PROP}.

We first need to discuss some preliminary constructions.
We will make use of a decomposition of the tensor product $[1] \otimes C$,
where $[1]$ is the $1$-simplex regarded as a $G$-trivial $G$-tree, and $C$ is a $G$-corolla 
(see Notation \ref{GCOR NOT}). Adapting the discussion
in Example \ref{THM71 EX}, $[1] \otimes C$ is the union of two maximal $G$-subtrees $C \star \eta$ and $\eta \star C$,
whose orbital representations are depicted below.
Explicitly, and noting that the edges (i.e. $\eta$-dendrices) of $\Omega[1] \otimes \Omega[C]$
are $\{0,1\} \times \boldsymbol{E}(C)$,
the tree
$C \star \eta$ (resp. $\eta \star C$)
has the edges $(i,e)$ such that
$i=0$ or $e\in \boldsymbol{E}(C)$ is a root 
(resp. $i=1$ or $e \in \boldsymbol{E}(C)$ is a leaf),
where we recall that in our convention $0$ is the leaf of $[1]$ while $1$ is the root.
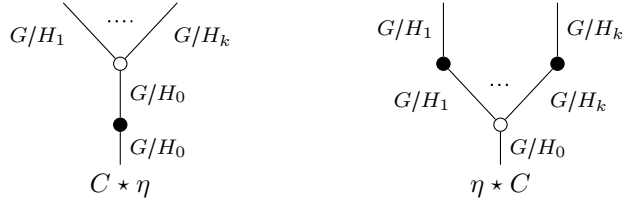
\begin{equation}\label{OTIMESDECOMP EQ}
\begin{tikzpicture}
[grow=up,auto,level distance=2.3em,every node/.style = {font=\footnotesize},dummy/.style={circle,draw,inner sep=0pt,minimum size=1.75mm}]
	\node at (0,0) [font=\normalsize] {$C \star \eta$}
		child{node [dummy,fill=black] {}
			child{node [dummy] {}
				child{
				edge from parent node [swap,near end] {$G/H_k$} node [name=Kn] {}}
				child{
				edge from parent node [near end] {$G/H_1$}
node [name=Kone,swap] {}}
			edge from parent node [swap] {$G/H_0$}}
		edge from parent node [swap] {$G/H_0$}};
		\draw [dotted,thick] (Kone) -- (Kn) ;
	\node at (5,0) [font=\normalsize] {$\eta \star C$}
		child{node [dummy] {}
			child{node [dummy,fill=black] {}
				child{
				edge from parent node [swap] {$G/H_k$}}
			edge from parent node [swap,near end] {$G/H_k$} node [name=Kn] {}}
			child{node [dummy,fill=black] {}
				child{
				edge from parent node {$G/H_1$}}
			edge from parent node [near end] {$G/H_1$}
node [name=Kone,swap] {}}
		edge from parent node [swap] {$G/H_0$}
		};
		\draw [dotted,thick] (Kone) -- (Kn) ;
\end{tikzpicture}
\end{equation}
Moreover, just as in \eqref{GENLEXREL EQ}, 
$C \star \eta$ and $\eta \star C$ have a common orbital \emph{inner} face, which we denote simply by $C$
(since this face is canonically isomorphic to the original $C$),
leading to a decomposition
\begin{equation}
      \label{OTIMESOC_EQ}
      \Omega[1] \otimes \Omega[C]
      \simeq
      \Omega[C \star \eta] \coprod_{\Omega[C]} \Omega[\eta \star C]
\end{equation}
We note that this holds even if $k=0$, 
which is an exceptional case since then
$[1] \otimes C = C \star \eta$.

\begin{remark}
\eqref{OTIMESOC_EQ} is the (equivariant) dendroidal generalization of the familiar 
decomposition
      \[
            \Delta[1] \times \Delta[1] \simeq \Delta[2] \coprod_{\Delta[1]} \Delta[2]
      \]
\end{remark}

Proposition \ref{JDDK PROP} will make use of the cube diagram below, where 
$\partial \Omega[C] = \partial^l \Omega[C] \amalg \partial^r \Omega[C]$ is the decomposition of the edges of $C$ into leaves and roots. All maps are the inclusions determined by the decomposition \eqref{OTIMESOC_EQ}, but it seems worthwhile to be explicit regarding how 
$\partial^l \Omega[C]$ and $\partial^r \Omega[C]$
include into
$\Omega[1] \otimes \Omega[C]$. 
An edge $l$ in $\partial^l \Omega[C]$ includes as $(0,l)$ while an edge $r$ in $\partial^r \Omega[C]$ includes as $(1,r)$.
\begin{equation}\label{BIGCUBE EQ}
\begin{tikzcd}[column sep=1pt,row sep=10pt]
	\partial \Omega[C] \ar{rr} \ar{rd} \ar{dd} && 
	\Omega[1] \otimes \partial^l \Omega[C]  \amalg \partial^r \Omega[C] \ar{rd} \ar{dd}
\\
	&
	|[alias=WW]|
	\partial^l \Omega[C]  \amalg  \Omega[1] \otimes \partial^r \Omega[C] && 
	|[alias=RR]|
	\Omega [1] \otimes \partial \Omega[C] \ar{dd}
\\
	\Omega[C] \ar{rr} \ar{rd} &&
	\Omega[\eta \star C] \ar{rd}
\\
	&
	|[alias=LL]|
	\Omega[C \star \eta]  \ar{rr} &&
	\Omega[1] \otimes \Omega[C]  
\arrow[from=WW,to=RR,crossing over]
\arrow[from=WW,to=LL,crossing over]
\end{tikzcd}
\end{equation}
Moreover, note that this is a projective cofibrant cube (recall that $\chi \colon (0 \to 1)^{\times n} \to \mathsf{dSet}^G$
is projective cofibrant if the maps 
$\colim_{\underline{j} < \underline{i}} \chi_{\underline{j}} \to \chi_{\underline{i}}$
are normal monomorphisms
for each 
$\underline{i} \in (0 \to 1)^{\times n}$).
Indeed, since all maps in \eqref{BIGCUBE EQ} are monomorphisms of presheaves, we can regard all objects as subpresheaves of 
$\Omega[1] \otimes \Omega[C]$, so that projective cofibrancy is the same as \eqref{BIGCUBE EQ} being strongly cartesian, i.e., 
the four objects
$\partial \Omega[C]$, $\Omega[C]$,
$\Omega[1] \otimes \partial^l \Omega[C]  \amalg \partial^r \Omega[C] $ and 
$\partial^l \Omega[C]  \amalg  \Omega[1] \otimes \partial^r \Omega[C]$
being the intersection of the objects they map to.


Lastly, note that both of the horizontal faces of \eqref{BIGCUBE EQ}
are pushout squares.

\begin{proposition}\label{JDDK PROP}
	Let $X\in \mathsf{sdSet}^G$ be a dendroidal Segal space. 
	Then the map $X \to X^{J}$ is a $DK$-equivalence.
\end{proposition}

\begin{proof}
	Note first that for any $T \in \Omega_G$ the map
	$X^{J}(\Omega[T]) \to X^{\Omega[1]}(\Omega[T])$ can be rewritten as
	$\left(X^{\Omega[T]}\right)(J) \to
	\left(X^{\Omega[T]}\right)(\Delta[1])
	$,
	where $X^{\Omega[T]} \in \mathsf{ssSet}$ is defined 
	by $X^{\Omega[T]}(m) = X(\Omega[m] \otimes \Omega[T])$,
	with $[m]$ given the $G$-trivial action (i.e., we are restricting Notation \ref{UNILIMDEN NOT}).
	Since $X^{\Omega[T]}$
	is a (simplicial) Segal space (by adjunction together with \cite[Prop. 7.25, Thm. 7.1]{Per18}),
	Proposition \ref{SESP PROP}(iv) says that this map is 
	a weak equivalence onto a subset of components,
	i.e. a \textit{homotopy monomorphism}.
	Hence, for any $G$-corolla
	$C \simeq C_{\amalg_i H_0/H_i}$ the horizontal maps in 
	the right square below are homotopy monomorphisms.
\begin{equation}\label{BIGSQ EQ}
\begin{tikzcd}
	X(\Omega[C]) \ar{r} \ar[->>]{d} & 
	X^{J}(\Omega[C]) \ar{r} \ar[->>]{d} & 
	X^{\Omega[1]}(\Omega[C]) \ar[->>]{d}
\\
	\prod_{0\leq i\leq k} X(\eta)^{H_i} \ar{r}&
	\prod_{0\leq i\leq k} \left(X^{J}(\eta)\right)^{H_i} \ar{r} &
	\prod_{0\leq i\leq k} \left(X^{\Omega[1]}(\eta)\right)^{H_i}
\end{tikzcd}	
\end{equation}

Since fully faithfulness of $X \to X^{J}$
is the statement that the leftmost square in \eqref{BIGSQ EQ} induces weak equivalences on fibers, it suffices to show that so does the composite square.


Now note that
$X^{\Omega[1]}(\Omega[C]) = X\left(\Omega[1] \otimes \Omega[C]\right)$, so that \eqref{BIGCUBE EQ} induces the cube \eqref{BIGCUBEDU EQ} below with dual properties: this is an injective fibrant cube whose horizontal faces are pullback squares.
Moreover, we will find it convenient to slightly repackage these properties. Regarding the top and bottom faces of \eqref{BIGCUBEDU EQ} as objects in the category of square diagrams, the vertical maps in \eqref{BIGCUBEDU EQ} can then be collectively regarded as a projective fibration between projective fibrant pullback squares.
\begin{equation}\label{BIGCUBEDU EQ}
\begin{tikzcd}[column sep=1pt,row sep=10pt]
	X\left(\Omega[1] \otimes \Omega[C]\right) \ar[->>]{rr} \ar[->>]{rd} \ar[->>]{dd} && 
	X(\Omega[\eta \star C]) \ar[->>]{rd} \ar[->>]{dd}
\\
	&
	|[alias=WW]|
	X(\Omega[C \star \eta]) && 
	|[alias=RR]|
	X(\Omega[C]) \ar[->>]{dd}
\\
	\underset{i}{\prod} X([1])^{H_i} \ar[->>]{rr} \ar[->>]{rd} &&
	\left(\underset{i \neq 0}{\prod} X([1])^{H_i}\right)
	\times X(\eta)^{H_0} \ar[->>]{rd}
\\
	&
	|[alias=LL]|
	\left(\underset{i \neq 0}{\prod} X(\eta)^{H_i}\right)
	\times X([1])^{H_0}  \ar[->>]{rr} &&
	\underset{i}{\prod} X(\eta)^{H_i} 
\arrow[from=WW,to=RR,crossing over,->>]
\arrow[from=WW,to=LL,crossing over,->>]
\end{tikzcd}
\end{equation}
Noting next that fibers (in the category of square diagrams)
of a fibration between pullback squares are fibrant pullback squares, it follows that the 
fibers of the top left vertical map in \eqref{BIGCUBEDU EQ}
are homotopy pullbacks of the fibers of the remaining vertical maps.
And since this top left vertical map is the right vertical map in \eqref{BIGSQ EQ}, the desired claim that the total diagram in \eqref{BIGSQ EQ}
induces equivalences on fibers will follow provided that the same holds for the following total diagrams
(which are also the composites of \eqref{BIGSQ EQ} with the left and back faces of \eqref{BIGCUBEDU EQ}; compare with 
\cite[Lemma 12.4]{Rez01}),
where the horizontal maps are the obvious degeneracies.
\[
\begin{tikzcd}[column sep=17]
	X(\Omega[C]) \ar{r}{s} \ar[equal]{d}&
	X(\Omega[C \star \eta]) \ar[->>]{d}{\sim}
&
	X(\Omega[C]) \ar{r}{s} \ar[equal]{d}&
	X(\Omega[\eta \star C]) \ar[->>]{d}{\sim}
\\
	X(\Omega[C]) \ar{r}{s} \ar[->>]{d}&
	X(\Omega[C]) \underset{X(\eta)^{H_0}}{\times} X([1])^{H_0} \ar[->>]{d}
&
	X(\Omega[C]) \ar{r}{s} \ar[->>]{d} & \underset{i \neq 0}{\prod} X([1])^{H_i} 
	\underset{\underset{i \neq 0}{\prod} X(\eta)^{H_i}}{\times} X(\Omega[C]) \ar[->>]{d}
\\
	\underset{i}{\prod} X(\eta)^{H_i} \ar{r}[swap]{s} &
	\left(\underset{i \neq 0}{\prod} X(\eta)^{H_i}\right)
	\times X([1])^{H_0}
&
	\underset{i}{\prod} X(\eta)^{H_i} \ar{r}[swap]{s} &
	\left(\underset{i \neq 0}{\prod} X([1])^{H_i}\right)
	\times X(\eta)^{H_0}
\end{tikzcd}
\]
But this is clear from the fact that the top right vertical
maps in these diagrams are trivial Kan fibrations
(by the Segal condition for $X$) and the fact that the bottom squares are pullback squares.

Lastly, to check essential surjectivity, 
since $G$ acts trivially on $J$, equation \eqref{HOCOEFF EQ} yields
$\iota_G^{\**} ho\left( X^J \right)(G/H)=
ho\left(\iota^{\**}\left(X^J\right)^H\right)=
ho\left(\iota^{\**}\left(X^H\right)^J\right)$,
where $(-)^H$ denotes fixed points,
so we reduce to the case of $X \in \mathsf{ssSet}$ a (simplicial) Segal space.
Since $J$ is a contractible Kan complex,
one has a map 
$F \colon J \times J \to \{0\} \times J$
such that 
$F|_{\{0\}\times J} = id_{\{0\} \times J}$
and
$F|_{\{1\}\times J} = (0,0)$.
But noting that 
$X(J \times J) \to X(\{0\} \times J)$
can be written as 
$X^J(0) \to X^J(J)$, the composites below
show that any object in $\left(X^J\right)_0$ is equivalent to a 
degenerate object, which is thus in the image of $X\to X^J$.
\[
	X^J(0) \xrightarrow{X(F)} 
	X^{J}(J) \to
	X^J(1) \rightrightarrows X^J(0)
\]
%
\end{proof}

\begin{definition}
	Two maps $f,f'\colon X \rightrightarrows Y$ between dendroidal Segal spaces are called \textit{$J$-homotopic}, written $f \sim_J f'$, if
	there is a homotopy $H$ such that
	the two composites
	$X \xrightarrow{H} Y^J \rightrightarrows Y$
	are $f,f'$.
	
	Further, a map $f\colon X \to Y$ of dendroidal Segal spaces is called a \textit{$J$-homotopy equivalence} if there exists $g \colon Y \to X$
	such that $gf \sim_J id_X$, $fg \sim_J id_Y$.
\end{definition}

\begin{remark}
	For $f\sim_J f'$, Proposition \ref{JDDK PROP} and 2-out-of-3 applied to $X \xrightarrow{H} Y^J \rightrightarrows Y$ imply that $f$ is a DK-equivalence iff $f'$ is.
	Thus by 2-out-of-6 for DK-equivalences (cf. Corollary \ref{26COR}), $J$-homotopy equivalences are DK-equivalences.
\end{remark}

\begin{remark}\label{ALLXJK REM}
	Let $X$ be a dendroidal Segal space. All simplicial operators
	$X^{J^m} \to X^{J^{m'}}$
	(see Notation \ref{JM NOT})
	are induced by equivalences of groupoids $\widetilde{[m']} \to \widetilde{[m]}$, 
	implying that the operators 
	$X^{J^m} \to X^{J^{m'}}$
	are $J$-homotopy equivalences and hence also DK-equivalences.
\end{remark}

\begin{proposition}\label{COMPLE PROP}
Let $X \in \mathsf{sdSet}^G$ be a dendroidal Segal space. 
Then there is a complete dendroidal Segal space $\tilde{X}$
and complete/joint equivalence $X \to \tilde{X}$ such that
\begin{itemize}
	\item[(i)] $X \to \tilde{X}$ is 
	a DK-equivalence;
	\item[(ii)] $X_0(\eta) \to \tilde{X}_0(\eta)$ is an isomorphism.
\end{itemize}
\end{proposition}

Our proof will adapt the construction of the completion functor in \cite[\S 14]{Rez01}.

Firstly, let 
$X^{J^{\bullet}} \in (\mathsf{sdSet}^G)^{\Delta^{op}}
= \mathsf{ssdSet}^G$
be the object whose $m$-th level
is $X^{J^m}$. 
Then, writing
$\delta^{\**} \colon \mathsf{ssdSet}^G \to \mathsf{sdSet}^G$
for the diagonal functor in the two simplicial directions
(adapting the functor in \S \ref{JOINBOUS SEC}), we set
$\tilde{X} = \delta^{\**}\left(X^{J^{\bullet}}\right)$ with the natural map $X \to \tilde{X}$ induced by degeneracies (cf. Proposition \ref{SSSETJREE PROP})(iv).

\begin{remark}\label{NOTTRANS REM}
For fixed $n$ and $K\in \mathsf{sSet}$, and regarding 
$X_n^{J^{\bullet}} \in \mathsf{sdSet}^G$,
one can translate between the 
$(-)_{(-)}$ and $(-)^{(-)}$ notations in Notation \ref{UNILIMDEN NOT} via
\[
	\left(X^{J^{\bullet}}_n\right)_K
=
	\left(X^{J^{\bullet}}_n\right)_{
	\left(\colim_{\Delta[k] \to K} \Delta[k]\right)}
=
	\lim_{
	\Delta[k] \to K
	} X^{J^k}_n
=
	X^{\left(\colim_{\Delta[k] \to K} J^k\right)}_n
\]
where $K \simeq \colim_{\Delta[k] \to K} \Delta[k]$ is the standard
decomposition of $K\in \mathsf{sSet}$ as a colimit of simplices.
In the special cases of $K$ either $\partial \Delta[m]$ or $Sc[m]$
one has (recall that $L_m$ is the latching object)
\[	\left(\colim_{\Delta[k] \to \partial \Delta[m]} J^k\right) \simeq L_m J^{\bullet}
\qquad
	\left(\colim_{\Delta[k] \to Sc[m]} J^k\right) \simeq 
	J \amalg_{\Delta[0]} \cdots \amalg_{\Delta[0]} J
\]
\end{remark}

\begin{proof}[Proof of Proposition \ref{COMPLE PROP}]

Most of the proof will be spent showing that 
$\tilde{X} = \delta^{\**}\left(X^{J^{\bullet}}\right) \in \mathsf{sdSet}^G$ is dendroidal Reedy fibrant
by establishing the lifting conditions needed to apply 
Corollary \ref{SSETSSETADJ COR}.

Since $J^{\bullet}$ is a Reedy cofibrant cosimplicial object (cf. Remark \ref{JREEDYCOF REM}),
it follows from \cite[Prop 7.25]{Per18} that 
the maps (where $\square^{\otimes}$ denotes the pushout product with regard to $\otimes$)
\[
X \left( \left(L_m J^{\bullet} \to J^m \right) \square^{\otimes} \left( \partial \Omega[T] \to \Omega[T] \right) \right)\]
are Kan fibrations. Remark \ref{NOTTRANS REM} then implies that
$X^{J^{\bullet}} \in \mathsf{ssdSet}^{G}$
has the right lifting property against all maps
\begin{equation}\label{LIFTBASIC EQ}
	\left( \Lambda^i[n] \to \Delta[n] \right)
\square
	\left( \partial \Delta[m] \to \Delta[m] \right)
\square
	\left( \partial \Omega[T] \to \Omega[T] \right),
	\qquad n \geq 1, n \geq i \geq 0, m\geq 0, T \in \Omega_G.
\end{equation}
We next claim that 
$X^{J^{\bullet}} \in \mathsf{ssdSet}^{G}$
also has the lifting property against the two sets of maps
\begin{equation}\label{LIFTETA EQ}
	\left( \partial \Delta[n] \to \Delta[n] \right)
\square
	\left( \Lambda^j [m] \to \Delta[m] \right)
\times
	\Omega[G/H \cdot \eta],
	\quad n \geq 0, m\geq 2, m \geq j \geq 0, H \leq G
\end{equation}
\begin{equation}\label{LIFTCSK EQ}
	\left( \partial \Delta[n] \to \Delta[n] \right)
\square
	\left( \Lambda^j [m] \to \Delta[m] \right)
\square
	\left( \partial \Omega[T] \to \Omega[T] \right),
	\quad n \geq 0, m\geq 1, m \geq j \geq 0, T \in \Omega_G \setminus \{G/H \cdot \eta\}
\end{equation}
Note that, just as in Remark \ref{HYPERMODEL REM}, one can replace the set of maps
$\Lambda^j[m] \to \Delta[m]$, $\partial\Omega[T] \to \Omega[T]$
appearing in \eqref{LIFTETA EQ} and  \eqref{LIFTCSK EQ}
with any set which has the same hypersaturation.

For \eqref{LIFTETA EQ}, by Remarks \ref{SLICE REM} and \ref{ANHYPER REM}
one needs only consider the 
inclusions 
$Sc[m] \to \Delta[m]$ and
$\Lambda^0[2] \to \Delta[2]$,
$\Lambda^2[2] \to \Delta[2]$.
But the claimed lifting condition against \eqref{LIFTETA EQ} then amounts to 
Remark \ref{LAMBJREAL REM}
applied to each of the simplicial Segal spaces $\iota^{\**}\left(X^H\right), H\leq G$ (also, see Remark \ref{NOTTRANS REM}).
For \eqref{LIFTCSK EQ}, by Remarks \ref{DUMBHYPER REM} and \ref{HYPERSATKAN REM}
one needs only show that there are lifts against the maps  
\[
\left(\partial \Delta[n] \to \Delta[n]\right)
\square
\left(\Delta[0] \to \Delta[m]\right) 
\square 
\left( \coprod_{e \in \boldsymbol{E}(T)} \Omega[\eta] \to \Omega[T]\right).
\]
But lifts against these maps are equivalent to the claim that in the square
\begin{equation}\label{MOREMORE EQ}
\begin{tikzcd}
	X^{J^m}(\Omega[T]) \ar[->>]{r} \ar[->>]{d}&
	X(\Omega[T]) \ar[->>]{d}
\\
	\underset{[e_i] \in \boldsymbol{E}_G(T)} {\prod} \left(X^{J^m}(\eta)\right)^{H_i} \ar[->>]{r} &
	\underset{[e_i] \in \boldsymbol{E}_G(T)} {\prod} 
	X(\eta)
	^{H_i}
\end{tikzcd}
\end{equation}
the map from 
$X^{J^m}(\Omega[T])$
to the pullback of the other terms is a trivial Kan fibration. That it is a Kan fibration follows from the lifting condition against \eqref{LIFTBASIC EQ}, and that it is a Kan equivalence follows since Remarks \ref{ALLXJK REM} and \ref{DKCOM REM} imply 
\eqref{MOREMORE EQ} induces Kan equivalences between 
fibers.

We 
finally show that 
$\tilde{X} = \delta^{\**} \left(X^{J^{\bullet}} \right)$
is dendroidal 
 fibrant, i.e. that 
$\left(\tilde{X}(\Omega[T]) \to \tilde{X}(\partial\Omega[T])\right) = 
\delta^{\**} \left( X^{J^{\bullet}}(\Omega[T]) \to X^{J^{\bullet}}(\partial \Omega[T])\right)$
is a Kan fibration for any $T \in \Omega_G$.
%
When $T= G/H \cdot \eta$ the target is the terminal object so this
follows from the lifting conditions against \eqref{LIFTBASIC EQ},\eqref{LIFTETA EQ}
and the second ``moreover'' condition in Corollary \ref{SSETSSETADJ COR}.
For $T \neq G/H \cdot \eta$ this follows from the lifting conditions against \eqref{LIFTBASIC EQ},\eqref{LIFTCSK EQ} and the first ``moreover'' condition in Corollary \ref{SSETSSETADJ COR}.

To see that $\tilde{X}$ is a complete Segal space, note that the natural  map
$X_0^{J^{\bullet}} \to
\delta^{\**}\left( X^{J^{\bullet}} \right)
= \tilde{X}$
is a dendroidal Reedy equivalence by Proposition \ref{SSSETJREE PROP}(iv)
so that, since $X_0^{J^{\bullet}}$ is a complete Segal space by Remark \ref{CONCRECOM REM}, so is $\tilde{X}$.

For the remaining claim that 
$X = X^{J^{0}} \to 
\delta^{\**}\left( X^{J^{\bullet}} \right)
= \tilde{X}$
is a DK equivalence,
fully faithfulness is the claim that this map induces equivalences on the fibers over
$\prod_{[e_i]} \left(X(\eta)\right)^{H_i}$
for each $T\in \Omega_G$.
But the fibers of 
$\tilde{X} = \delta^{\**}\left( X^{J^{\bullet}} \right) $
are diagonals of the fibers of 
$X^{J^{\bullet}}$ over 
$\prod_{[e_i]} \left(X^{J^{\bullet}}(\eta)\right)^{H_i}$
for each $T\in \Omega_G$,
and the lifting conditions against 
\eqref{LIFTBASIC EQ},\eqref{LIFTCSK EQ}
imply that these fibers are joint Reedy fibrant in $\mathsf{ssSet}$, 
so fully faithfullness now follows from Proposition \ref{SSSETJREE PROP}(iv) applied to these fibers.
Essential surjectivity is trivial since $X \to \tilde{X}$ is the identity on objects.
\end{proof}

\begin{remark}
      \label{COMPLE REM}
A notable difference between the proof of Proposition \ref{COMPLE PROP}
and the arguments in \cite[\S 14]{Rez01} being adapted is that 
\cite{Rez01} did not establish the analogue of our main fibrancy claim
 when $X \in \mathsf{ssSet}$ is a Segal space, namely that $\delta^{\**} \left(X^{J^{\bullet}} \right)$ is horizontal Reedy fibrant in $\mathsf{ssSet}$.
Instead, \cite{Rez01} defines $\tilde{X}$ as the horizontal fibrant replacement of $\delta^{\**} \left(X^{J^{\bullet}} \right)$
and the analogue of \eqref{MOREMORE EQ} is used to conclude that the fibers of 
$\delta^{\**} \left(X^{J^{\bullet}} \right)$ 
are homotopy fibers, and thus preserved by the replacement. 
As such, the lifting conditions against 
\eqref{LIFTETA EQ} are technically not essential for the proof, serving only to obtain the neat observation that 
$\delta^{\**} \left(X^{J^{\bullet}} \right)$ 
needs not be replaced.
\end{remark}

\begin{theorem}\label{COMPIFFDK THM}
	A map $f\colon X \to Y$ of dendroidal Segal spaces is a complete/joint equivalence iff it is a DK-equivalence.
\end{theorem}

\begin{proof}
By Proposition \ref{COMPLE PROP}(i) we can assume that $X,Y$ are complete/joint fibrant, so that by Proposition \ref{COMBMODSTR PROP}(ii)
complete/joint equivalences coincide with simplicial equivalences. But then by Corollary \ref{DKCOM COR}
it remains only to show that if $f$ is a DK-equivalence then the maps
$X(\eta)^H \to Y(\eta)^H, H \leq G$ are Kan equivalences.
Since the induced maps 
$\iota^{\**}\left(X^H\right) \to \iota^{\**}\left(Y^H\right)$
in $\mathsf{ssSet}$
are DK-equivalences of simplicial Segal spaces, this last claim is immediate from the simplicial analogue \cite[Thm. 7.7]{Rez01}, but we nonetheless include a full argument for
the sake of
completeness.

As such, assuming that 
$f\colon X \to Y$ in $\mathsf{ssSet}$
is a DK-equivalence of complete simplicial Segal spaces, it remains to show $X(0)\to Y(0)$ is a Kan equivalence.
The completion condition states that
$Z(J) \xrightarrow{\sim} Z(0)$ for $Z=X,Y$ are Kan equivalences, so that the fibers of the left diagram below
are weakly equivalent to the homotopy fibers of the right diagram,
i.e. the loop space of $Z(0)$ at $z$.
\begin{equation}\label{WHATEV EQ}
\begin{tikzcd}[column sep =40]
	& Z(J) \ar[->>]{d}
&
	& Z(0) \ar{d}
\\
	\Delta[0] \ar{r}{(z,z)} &
	Z(0) \times X(0)
&
	\Delta[0] \ar{r}{(z,z)} &
	Z(0) \times Z(0)
\end{tikzcd}
\end{equation}
Therefore, since $Z(J) \to Z(1)$ is a homotopy monomorphism (cf. Proposition \ref{SESP PROP}(i)(iv)),
fully faithfulness implies that 
$X(0) \to Y(0)$
induces isomorphisms on homotopy groups.
Injectivity of $X(0) \to Y(0)$
on components is similar 
($z,z'\in Z_0(0)$ are in the same component iff the $(z,z')$ fiber in \eqref{WHATEV EQ} is non-empty) 
while essential surjectivity implies surjectivity on components
(the equivalences $Z(0) \xrightarrow{\sim} Z(J) \xrightarrow{\sim} Z^h(1)$ 
imply both vertex maps $Z^h(1) \rightrightarrows Z(0)$ are equivalences, 
and hence $z,z'\in Z_0(0)$ are isomorphic in $ho(Z)$ iff they are connected in $Z(0)$).

It follows that $X(0) \to Y(0)$ is indeed a Kan equivalence, as required.
\end{proof}

\begin{remark}\label{COMPLENOR REM}
It is clear from the construction of the completion $\tilde{X}$
in Proposition \ref{COMPLE PROP}
that the map $X \to \tilde{X}$
is a monomorphism. 
However, it seems unlikely that this is a 
\emph{normal} monomorphism (unless $X$ is assumed normal).
To address this while preserving the properties in Proposition \ref{COMPLE PROP}, 
note that the identification 
$X_0(\eta) = \tilde{X}_0(\eta)$
means that one can perform a
``cofibration followed by trivial fibration'' factorization
\[
X \rightarrowtail \hat{X} \overset{\sim}{\twoheadrightarrow}  \tilde{X}
\]
in the dendroidal Reedy model structure in such a way that 
$X_0(\eta) = \hat{X}_0(\eta) = \tilde{X}_0(\eta)$. Indeed, this follows by performing the small object argument against \eqref{JOINTCOF EQ} with the maps $\emptyset \to \Delta[0] \times \Omega[G/H \cdot \eta], H\leq G$ omitted (note that $\hat{X} \to \tilde{X}$ will still have the lifting property against the omitted maps).
The claims that $\hat{X}$ is complete and $X \to \hat{X}$ is a DK-equivalence are inherited from the analogous properties of $\tilde{X}$.
\end{remark}

\begin{corollary}\label{FIB_PREOP_COR}
	A pre-operad $X \in \mathsf{PreOp}^G$ is 
	fibrant (in the model structure from Theorem \ref{PREOPMOD THM})
	iff $\gamma^{\**}X$ is fibrant in the dendroidal Segal space model structure on 
	$\mathsf{sdSet}^G$.
\end{corollary}

\begin{proof}
	We start with the ``only if'' direction.
Recall that $\gamma^{\**} X$ is a dendroidal Segal space iff it has the right lifting property against the maps of the form
\begin{equation}\label{SOMEMAPS EQ}
	(\Lambda^i[n] \to \Delta[n]) \square (\partial \Omega[T] \to \Omega[T])
\qquad
	(\partial \Delta[n] \to \Delta[n]) \square (Sc[T] \to \Omega[T]).
\end{equation}
With the exception of the first type of maps when $T = G\cdot_H \eta$, in which case the lifting condition
against $\gamma^{\**} X$ is automatic since
$\gamma^{\**}X(\eta)$ is discrete, all other maps induce isomorphisms at the $\eta$-level, so that by 
Remark \ref{GAMMASH REM} applying $\gamma_{!}$ to these maps yields trivial cofibrations in 
$\mathsf{PreOp}^G$.
Thus, if $X \in \mathsf{PreOp}^G$ is fibrant, an adjunction argument shows that $\gamma^{\**}(X)$ indeed has the lifting property against all maps \eqref{SOMEMAPS EQ}, i.e. that
$\gamma^{\**}(X)$ is a dendroidal Segal space.

For the ``if'' direction, we form the completion 
$\gamma^{\**}X \to \hat{X}$
described in Remark \ref{COMPLENOR REM}.
Then $\gamma_{\**}\hat{X} \in \mathsf{PreOp}^G$
is fibrant by Theorem \ref{ANOQUEQUIV THM}
and the adjoint map $X \to \gamma_{\**}\hat{X}$
has the following properties:
\begin{inparaenum}
	\item[(i)] it is a normal monomorphism (this is inherited from $\gamma^{\**}X \to \hat{X}$; see the characterization of normal monomorphisms given before Theorem \ref{DSETGMODEL THM});
	\item[(ii)] it is an isomorphism at the $\eta$-level;
	\item[(iii)] it is a DK-equivalence when regarded as a map
	in $\mathsf{sdSet}^G$
	(since by Remark \ref{ONLYPREOP REM} $\gamma^{\**}\gamma_{\**}\hat{X} \to \hat{X}$ is tautologically a DK-equivalence);
	\item[(iv)] by Corollary \ref{DKCOM COR}(ii) it is hence a simplicial equivalence and thus
	a trivial dendroidal Reedy cofibration when regarded as a map
	in $\mathsf{sdSet}^G$. 
\end{inparaenum}	
	But then the hypothesis that
	$\gamma^{\**}X$ is a dendroidal Segal space yields a lift
\[
\begin{tikzcd}
      \gamma^{\**} X \ar[equal]{r} \ar{d}&
      \gamma^{\**} X
      \\
      \gamma^{\**}\gamma_{\**}\hat{X} \ar[dashed]{ru}
\end{tikzcd}
\]
showing that $X$ is a retract of $\gamma_{\**}\hat{X}$ and finishing the proof.
\end{proof}

\begin{remark}\label{INTERP REM}
      For any dendroidal Segal space 
      $X \in \mathsf{sdSet}^G$ one hence has complete equivalences
      \[
            \gamma_{\**} X \to X \to \tilde{X}
      \] 
      where $\gamma_{\**}X$ is a fibrant preoperad and $\tilde{X}$ is
      a complete dendroidal Segal space.
%
\end{remark}

\section{Indexing system analogue results}\label{INDEX SEC}

Just as in \cite[\S 9]{Per18}, we dedicate our final section
to outlining the generalizations of our results parametrized by the 
\textit{indexing systems} of Blumberg and Hill \cite{BH15}.
Or more precisely, we will work with the \textit{weak indexing systems}
of \cite[\S 9]{Per18}, \cite[\S 4.4]{BP17},
which are a slight generalization of indexing systems, 
and were also independently identified by Gutierrez and White in \cite{GW18}.

We begin by recalling the key notion of sieve.

\begin{definition}
A \textit{sieve} of a category $\mathcal{C}$
is a full subcategory $\mathcal{S} \subseteq \mathcal{C}$ such that for any arrow $c \to s$ in $\mathcal{C}$ such that $s \in \mathcal{S}$ it is also $c \in \mathcal{S}$.
\end{definition}

Note that a sieve $\mathcal{S} \subseteq \mathcal{C}$
determines a presheaf 
$\delta_{\mathcal{S}} \in \mathsf{Set}^{\mathcal{C}^{op}}$ 
via 
$\delta_{\mathcal{S}}(c) = \**$ if $c \in \mathcal{S}$ and
$\delta_{\mathcal{S}}(c) = \emptyset$ if $c \not \in \mathcal{S}$.
In fact, there is a clear bijection between sieves and such 
\textit{characteristic presheaves}, i.e. presheaves taking only the values $\**$ and $\emptyset$, 
and we will hence blur the distinction between the two concepts.

Sieves are prevalent in equivariant homotopy theory. Indeed, families $\mathcal{F}$ of subgroups of $G$ are effectively the same as sieves $\mathsf{O}_{\mathcal{F}} \subseteq \mathsf{O}_G$
of the orbit category
$\mathsf{O}_G$ (formed by the $G$-sets $G/H$).

Weak indexing systems can then be thought of as the operadic analogue of families. In particular, they are described by certain sieves $\Omega_{\mathcal{F}} \subseteq \Omega_G$,
though additional conditions are needed to ensure compatibility with the operadic composition and unit.
In the following, we abbreviate 
$\delta_{\mathcal{F}} = \delta_{\Omega_{\mathcal{F}}}$ and,
for each $G$-vertex $v$ of $T \in \Omega_G$, we write 
$T_v \hookrightarrow T$ for the orbital outer face whose only
$G$-vertex is $v$.

\begin{definition}
A \textit{weak indexing system} is a full subcategory
$\Omega_{\mathcal{F}} \subseteq \Omega_G$ such that:
\begin{itemize}
	\item[(i)] $\Omega_{\mathcal{F}}$ is a sieve of 
              $\Omega_G$;              
	\item[(ii)] for each $T \in \Omega_G$ it is
	$T \in \Omega_{\mathcal{F}}$ iff 
	$\forall_{v \in V_G(T)} T_v \in \Omega_{\mathcal{F}}$
	or, equivalently, if
        \begin{equation}\label{SEGCOMB EQ}
              \delta_{\mathcal{F}}(T) =
              \prod_{v \in \boldsymbol{V}_G(T)}\delta_{\mathcal{F}}(T_v).
        \end{equation}
\end{itemize}
\end{definition}

\begin{remark}\label{SEGCOMB REM}
Given (i), condition (ii) can be reinterpreted as combining the following: 
\begin{itemize}
	\item[(ii')] the characteristic presheaf $\delta_{\mathcal{F}}$ is Segal, i.e. 
	$\delta_{\mathcal{F}}(T) = 
	\delta_{\mathcal{F}}(Sc[T])$
	for all $T \in \Omega_G$;
	\item[(ii'')] $(G/G \cdot \eta) \in \Omega_{\mathcal{F}}$.
\end{itemize}

Here, (ii'') reflects the existence of units in $G$-operads,
which are encoded by the
$G$-trivial $1$-corolla $G/G\cdot [1]$
(note that by the sieve condition (i) it is
$(G/G\cdot [1]) \in \Omega_{\mathcal{F}}$
iff 
$(G/G\cdot \eta) \in \Omega_{\mathcal{F}}$).

Similarly, (ii') reflects the composition in $G$-operads.
Indeed, (i) and (ii'') imply that
all stick $G$-trees $G/H \cdot \eta$ are in 
$\Omega_{\mathcal{F}}$,
so that the right hand side of \eqref{SEGCOMB EQ}
can be reinterpreted as $\delta_{\mathcal{F}}(Sc[T])$
(more formally, $\delta_{\mathcal{F}}(Sc[T])$ is defined via an analogue of 
Notation \ref{UNILIMDEN NOT}, 
so as to obtain a functor
$\delta_{\mathcal{F}}(-) \colon 
(\mathsf{dSet}_G)^{op} \to \mathsf{sSet}$ 
and by reinterpreting $Sc[T]$ as an object in 
$\mathsf{dSet}_G$ via applying $\upsilon_{\**}$).
\end{remark}

\begin{remark}
The original notion of indexing system 
in \cite[Def. 3.22]{BH15} is recovered by demanding that all
$G$-trivial $n$-corollas $G/G \cdot C_n$ are in 
$\Omega_{\mathcal{F}}$.
\end{remark}

\begin{remark}\label{WHYF REM}
	The $\mathcal{F}$ in the notation $\Omega_{\mathcal{F}}$
is meant to suggest an alternate description of (weak) indexing systems in terms of families of subgroups.

Namely, given a weak indexing system $\Omega_{\mathcal{F}}$ and $n \geq 0$, we let $\mathcal{F}_n$
denote the family of those subgroups 
of $\Gamma \leq G \times \Sigma_n = G \times \mathsf{Aut}(C_n)$ which are graphs of partial homomorphisms
	$G \geq H \to \Sigma_n$
	such that the associated $G$-corolla $G\cdot_H C_n$ is in $\Omega_{\mathcal{F}}$.
	$\mathcal{F}$ then stands for the collection
	$\mathcal{F} = \{\mathcal{F}_{n}\}_{n \geq 0}$.

More generally, for each $U \in \Omega$, we similarly write
$\mathcal{F}_U$
for the family of graph subgroups of $G \times \mathsf{Aut}(U)$ encoding partial homomorphisms 
$G \geq H \to \mathsf{Aut}(U)$ such that
$G \cdot_H U \in \Omega_{\mathcal{F}}$.

The fact that each $\mathcal{F}_U$ is a family is a consequence of the sieve condition (i).
On the other hand, (ii) imposes more complex conditions on $\{\mathcal{F}_{n}\}_{n \geq 0}$
which \cite[Def. 3.22]{BH15} makes explicit.
\end{remark}

All results in the paper now extend to the context of a general weak indexing system $\Omega_{\mathcal{F}}$
by essentially replacing $\Omega_G$ with $\Omega_{\mathcal{F}}$ throughout. The following are some notable modifications:
\begin{itemize}
\item notions in $\mathsf{dSet}^G$ discussed in \S \ref{PREL SEC} such as 
	``$G$-normal monomorphism'', ``$G$-inner horn'', ``$G$-inner anodyne'', ``$G$-$\infty$-operad'' are replaced
	(by restricting $T \in \Omega_{G}$ to $T \in \Omega_{\mathcal{F}}$) with
	``$\mathcal{F}$-normal monomorphism'', ``$\mathcal{F}$-inner horn'', ``$\mathcal{F}$-inner anodyne'', ``$\mathcal{F}$-$\infty$-operad'';
\item the model structure on $\mathsf{dSet}^G$ from 
\cite[Thm. 2.1]{Per18} is replaced with the model structure
$\mathsf{dSet}^G_{\mathcal{F}}$ (on the \textit{same} underlying category) from 
\cite[Thm. 2.2]{Per18},
whose cofibrations are the $\mathcal{F}$-normal monomorphisms
and whose fibrant objects are the $\mathcal{F}$-$\infty$-operads;
\item genuine dendroidal sets $\mathsf{dSet}_G = \mathsf{Set}^{\Omega_G^{op}}$ are replaced with
$\mathcal{F}$-dendroidal sets
	$\mathsf{dSet}_{\mathcal{F}} = \mathsf{Set}^{\Omega_{\mathcal{F}}^{op}}$.
\end{itemize}
We briefly outline the main reasons why these substitutions do not affect our proofs.

Firstly, the characteristic edge lemma, Lemma \ref{CHAREDGE LEM}, extends automatically. Indeed, if $T \in \Omega_{\mathcal{F}}$ the sieve condition for $\Omega_{\mathcal{F}}$ implies 
that the filtrations produced by the original 
Lemma \ref{CHAREDGE LEM} must necessarily use only 
$\mathcal{F}$-inner horn inclusions.
Therefore, all results in \S \ref{HYPERSAT SEC}, 
most notably Proposition \ref{HYPER PROP} concerning hypersaturations,
extend to a general weak indexing system $\Omega_{\mathcal{F}}$
via the same proof.

Next, for \S \ref{CEDSS SEC}, one can again consider two different Reedy model structures on 
$\mathsf{sdSet}^G$.
Firstly, using the fact that $\Delta$ is Reedy and the model structure $\mathsf{dSet}^G_{\mathcal{F}}$, one obtains a
\textit{$\mathcal{F}$-simplicial Reedy model structure} on
$\mathsf{sdSet}^G$.
Secondly, using the fact that $\Omega^{op} \times G$ is generalized Reedy such that the families 
$\{\mathcal{F}_U\}$ in Remark \ref{WHYF REM}
are Reedy admissible (see Example \ref{FGRAPHREEDY EX})
together with the Kan model structure on 
$\mathsf{sSet}$, Theorem \ref{REEDYADM THM} yields a 
\textit{$\mathcal{F}$-dendroidal Reedy model structure}
on $\mathsf{sdSet}^G$.
Thus, by applying Proposition \ref{COMBMODSTR PROP} to combine the two structures,
one obtains a $\mathcal{F}$-joint/$\mathcal{F}$-Rezk model structure, which we denote
$\mathsf{sdSet}^G_{\mathcal{F}}$.
The remaining discussion in \S \ref{CEDSS SEC} then follows through to yield the analogue of Theorem \ref{INC0AGJ THM}.

\begin{theorem}\label{FINC0AGJ THM}
	The constant/$0$-th level adjunction
	\[
	c_!\colon 
	\mathsf{dSet}^G_{\mathcal{F}} \rightleftarrows \mathsf{sdSet}^G_{\mathcal{F}}
	\colon (-)_0
	\]
	is a Quillen equivalence.
\end{theorem}
The modifications for \S \ref{PREOP SEC} are entirely straightforward, with the model structure 
$\mathsf{sdSet}^G_{\mathcal{F}}$
inducing a model structure
$\mathsf{PreOp}^G_{\mathcal{F}}$ via the obvious analogue of 
Theorem \ref{PREOPMOD THM}, and yielding the analogue of Theorem \ref{ANOQUEQUIV THM}.
\begin{theorem}\label{FANOQUEQUIV THM}
The adjunction
\[
	\gamma^{\**} \colon \mathsf{PreOp}^G_{\mathcal{F}}
\rightleftarrows
	\mathsf{sdSet}^G_{\mathcal{F}} \colon \gamma_{\**}
\]
is a Quillen equivalence.
\end{theorem}
For \S \ref{HMPTYGEN SEC}, $\mathcal{F}$-dendroidal Segal spaces
$X \in \mathsf{sdSet}^G$
are defined in the natural way by localizing the $\mathcal{F}$-dendroidal Reedy model structure against the Segal core inclusions
$Sc[T] \to \Omega[T], T \in \Omega_{\mathcal{F}}$.
The most notable difference is then that in Notation \ref{GCOR NOT} and afterwards one works only with 
$\mathcal{F}$-corollas, i.e. $G$-corollas
$C \in \Omega_{\mathcal{F}}$,
and thus only with $\mathcal{F}$-profiles, thus obtaining a notion of $\mathcal{F}$-fully faithfullness and of $\mathcal{F}$-DK-equivalence 
(essential surjectivity needs not be changed due to condition (ii'') in Remark \ref{SEGCOMB REM} implying that all the stick $G$-trees $G/H \cdot \eta$ are in $\Omega_{\mathcal{F}}$).
Thus, noting that the Segal condition 
(ii') in Remark \ref{SEGCOMB REM}
ensures that the grafted $G$-trees 
$T=C \amalg_{Ge} D$ in \eqref{HOMOTCIRC EQ}
are in $\Omega_{\mathcal{F}}$
whenever $C,D$ are $\mathcal{F}$-corollas,
the remaining discussion in 
\S \ref{HMPTYGEN SEC}, \S \ref{REZKCOMP SEC}
generalizes to yield the analogue of 
Theorem \ref{COMPIFFDK THM}.
\begin{theorem}\label{FCOMPIFFDK THM}
A map of $X \to Y$ of $\mathcal{F}$-dendroidal Segal spaces is a $\mathcal{F}$-complete equivalence iff it is a $\mathcal{F}$-DK-equivalence.
\end{theorem}

\appendix

\section{Equivariant Reedy model structures}\label{EQREED AP}

In \cite{BM11} Berger and Moerdijk extend the notion of Reedy category so as to allow for categories $\mathbb{R}$
 with non-trivial automorphism groups 
 $\mathsf{Aut}(r)$ for $r \in \mathbb{R}$.
For such $\mathbb{R}$ and suitable model category $\mathcal{C}$ they then show that there is a 
\textit{Reedy model structure}
on $\mathcal{C}^{\mathbb{R}}$
defined by modifying the usual characterizations of
Reedy cofibrations, weak equivalences and fibrations
(see \cite[Thm. 1.6]{BM11} or
Theorem \ref{REEDYADM THM} below)
to be determined by the $\mathsf{Aut}(r)$-projective model structures
on $\mathcal{C}^{\mathsf{Aut}(r)}$
for each $r \in \mathbb{R}$.

The purpose of this appendix is to show that,
under suitable conditions, this can also be done by replacing
the $\mathsf{Aut}(r)$-projective model structures
on $\mathcal{C}^{\mathsf{Aut}(r)}$
with the more general 
$\mathcal{C}^{\mathsf{Aut}(r)}_{\mathcal{F}_r}$
model structures for 
$\{\mathcal{F}_r\}_{r \in \mathbb{R}}$
a nice collection of families of subgroups of each 
$\mathsf{Aut}(r)$.

To do so, we first need some key notation.
For each map $r \to r'$ in the category $\mathbb{R}$ we will write
$\mathsf{Aut}(r \to r')$ for its automorphism group in the arrow category and write
\begin{equation}\label{PIDEFR EQ}
\begin{tikzcd}
\mathsf{Aut}(r) &
\mathsf{Aut}(r \to r') \ar{r}{\pi_{r'}} \ar{l}[swap]{\pi_{r}} &
\mathsf{Aut}(r')
\end{tikzcd}
\end{equation}
for the obvious projections. We now introduce our equivariant generalization of
the ``generalized Reedy categories''
of \cite[Def. 1.1]{BM11}, 
the novelty of which is in axiom (iv).

\begin{definition}\label{GENRED DEF}
A \textit{generalized Reedy category structure} on a
small category $\mathbb{R}$ consists of
wide subcategories 
$\mathbb{R}^+$, $\mathbb{R}^-$
and a degree function $|\minus| \colon ob(\mathbb{R}) \to \mathbb{N}$ such that:
\begin{itemize}
	\item[(i)] non-invertible maps in $\mathbb{R}^+$ (resp. $\mathbb{R}^-$) raise (lower) degree; isomorphisms preserve degree;
	\item[(ii)] $\mathbb{R}^+ \cap \mathbb{R}^- = \mathsf{Iso}(\mathbb{R})$;
	\item[(iii)] every map $f$ in $\mathbb{R}$ factors as
	$f = f^{+} \circ f^{-}$ with $f^{+} \in \mathbb{R}^+$, $f^{-} \in \mathbb{R}^-$, and this factorization is unique up to isomorphism
	(meaning that if 
	$f = \tilde{f}^{+} \circ \tilde{f}^{-}$ is another such factorization, then
	$f^+ = \tilde{f}^+ \circ \phi$, $\phi \circ f^- = \tilde{f}^-$ for some isomorphism $\phi$).
\end{itemize}
Let $\{\mathcal{F}_r\}_{r \in \mathbb{R}}$
be a collection of families of subgroups of the groups $\mathsf{Aut}(r)$.
The collection $\{\mathcal{F}_r\}$ is called 
\textit{Reedy-admissible} if:
\begin{itemize}
	\item[(iv)] for all maps
	$r \to r'$ in $\mathbb{R}^-$ one has
	$\pi_{r'}\left( \pi_r^{-1} (H) \right) \in \mathcal{F}_{r'}$
	for all $H \in \mathcal{F}_r$.
\end{itemize}
\end{definition}

We note that condition (iv) above should be thought of as a constraint on the pair 
$(\mathbb{R},\{\mathcal{F}_r\})$.
The original setup of \cite{BM11} then deals with the case
where $\{ \mathcal{F}_r \} =
 \left\{ \left\{ e \right\} \right\}$
is the collection of trivial families. Indeed, our setup recovers
the setup in \cite{BM11}, as follows.

\begin{example}\label{RECAXIV EX}
	When $\{ \mathcal{F}_r \} =
 \left\{ \left\{ e \right\} \right\}$, Reedy-admissibility coincides with axiom (iv) in \cite[Def. 1.1]{BM11},
stating that if $\theta \circ f^{-} = f^{-}$
for some $f^- \in \mathbb{R}^{-}$ and 
$\theta \in \mathsf{Iso}(\mathbb{R})$ then $\theta$ is an identity.
\end{example}

\begin{example}
For any generalized Reedy category $\mathbb{R}$
(i.e. if $\mathbb{R}$ satisfies (i),(ii),(iii)),
the collection $\{\mathcal{F}_{\text{all}}\}$
of the families of all subgroups of $\mathsf{Aut}(r)$
is Reedy-admissible.
\end{example}

\begin{example}
	Let $G$ be a group and set $\mathbb{R} = G \times (0 \to 1)$ with $\mathbb{R} = \mathbb{R}^+$
	(and thus necessarily $\mathbb{R}^- = \mathsf{Iso}(\mathbb{R})$).
	Then any pair 
	$\{\mathcal{F}_0,\mathcal{F}_1\}$
	of families of subgroups of $G$ is Reedy-admissible.
	
	Similarly, set $\mathbb{R} = G \times (0 \leftarrow 1)$
	with $\mathbb{R} = \mathbb{R}^-$. Then a pair
	$\{\mathcal{F}_0,\mathcal{F}_1\}$
	of families of subgroups of $G$ is Reedy-admissible
	iff $\mathcal{F}_0 \supseteq \mathcal{F}_1$.
\end{example}

\begin{example}\label{GGRAPHREEDY EX}
	Letting $\mathbb{S}$ denote any generalized Reedy category in the sense of \cite[Def. 1.1]{BM11} (cf. Example \ref{RECAXIV EX}) and $G$ a group,
	we set $\mathbb{R} = G \times \mathbb{S}$
	with $\mathbb{R}^+ = G \times \mathbb{S}^+$ and 
	$\mathbb{R}^- = G \times \mathbb{S}^-$.
	Further, for each $s \in \mathbb{S}$ we write
	$\mathcal{F}_s^{\Gamma}$ for the family of 
	$G$-graph subgroups of $G \times \mathsf{Aut}_{\mathbb{S}}(s)$, i.e., those subgroups 
	$\Gamma \leq G \times \mathsf{Aut}_{\mathbb{S}}(s)$ 
	which are graphs of partial homomorphisms
	$G \geq H \to \mathsf{Aut}_{\mathbb{S}}(s)$. We note that $G$-graph subgroups are also characterized by the condition $\Gamma \cap \mathsf{Aut}_{\mathbb{S}}(s) = \{e\}$.
	
	Reedy admissibility of the collection $\left\{\mathcal{F}_s^{\Gamma}\right\}$ then follows since for every map 
	$s \to s'$ in $\mathbb{S}^-$ one has that the homomorphism
	$\pi_s \colon \mathsf{Aut}_{\mathbb{S}}(s \to s')
	\to \mathsf{Aut}_{\mathbb{S}}(s)$ is injective
	(we note that this is a restatement of axiom (iv) in \cite[Def. 1.1]{BM11} for $\mathbb{S}$; see Example \ref{RECAXIV EX}).
\end{example}

Our primary example of interest is obtained by setting
$\mathbb{S} = \Omega^{op}$ in the previous example.
Moreover, in this case we are also interested 
in certain subfamilies
$\{\mathcal{F}_U\}_{U \in \Omega}
\subseteq
\left\{\mathcal{F}_U^{\Gamma}\right\}_{U \in \Omega}$.

\begin{example}\label{FGRAPHREEDY EX}
	Let $\mathbb{R} = G \times \Omega^{op}$ and let
	$\{\mathcal{F}_U\}_{U \in \Omega}$ be the family of graph subgroups determined by a weak indexing system $\mathcal{F}$ (see Remark \ref{WHYF REM}).
	Then $\{\mathcal{F}_U\}$ is Reedy-admissible.
	To see this, recall first that each $\Gamma \in \mathcal{F}_U$ encodes 
	a $H$-action on $U \in \Omega$ for some $H \leq G$
	so that $G \cdot_H U$ is a $\mathcal{F}$-tree.
	Given a face map $\varphi \colon U' \hookrightarrow U$, 
	the subgroup $\pi^{-1}_U(\Gamma)$ is then determined by the largest subgroup $\bar{H}\leq H$ such that 
	$U'$ inherits the $\bar{H}$-action from $U$ along $\varphi$ (thus making $\varphi$ a $\bar{H}$-map), 
	so that $\pi_{U'}\left(\pi^{-1}_U(\Gamma)\right)$ is the graph subgroup encoding the $\bar{H}$-action on $U'$.
	Thus, we see that Reedy-admissibility is simply the sieve condition for the induced map of $G$-trees
	$G \cdot_{\bar{H}} U' \to G \cdot_H U$.
\end{example}

We now state the main result.
We will assume throughout that $\mathcal{C}$ is a model category such that for any group $G$ and family of subgroups $\mathcal{F}$,
the category $\mathcal{C}^G$ admits the
$\mathcal{F}$-model structure
with weak equivalences/fibrations detected by the fixed points
$X^H$ for $H \in \mathcal{F}$
(for example, this is the case whenever $\C$ is a cofibrantly generated cellular model category in the sense of \cite{Ste16}).

The definitions of $L_r$ and $M_r$ are recalled below the result.

\begin{theorem}\label{REEDYADM THM}
Let $\mathbb{R}$ be generalized Reedy and 
$\{\mathcal{F}_r\}_{r \in \mathbb{R}}$ a Reedy-admissible collection of families. 
Then there is a \textbf{$\{\mathcal{F}_r\}$-Reedy model structure} on
$\mathcal{C}^{\mathbb{R}}$ such that a map $A \to B$ is
\begin{itemize}
  \item a (trivial) cofibration if $A_r \underset{L_r A}{\amalg}L_r B \to B_r$ is a (trivial) $\mathcal{F}_r$-cofibration in $\mathcal{C}^{\mathsf{Aut}(r)}$, $\forall r \in \mathbb{R}$;
	\item a weak equivalence if $A_r \to B_r$ is a $\mathcal{F}_r$-weak equivalence in $\mathcal{C}^{\mathsf{Aut}(r)}$, $\forall r \in \mathbb{R}$;
	\item a (trivial) fibration if $A_r \to B_r \underset{M_r B}{\times }M_r A $ is a (trivial) $\mathcal{F}_r$-fibration in $\mathcal{C}^{\mathsf{Aut}(r)}$, $\forall r \in \mathbb{R}$.
\end{itemize}
\end{theorem}

The proof of Theorem \ref{REEDYADM THM}
is given near the end of the appendix after establishing some routine generalizations of the key lemmas in \cite{BM11}.
We note that the work in \cite{BM11} has two main components: a formal analysis of the 
\textit{latching} and \textit{matching objects}
$L_r A$ and $M_r A$, which depends only on axioms 
\cite[Def. 1.1]{BM11}(i),(ii),(iii), 
and a model category analysis,
which depends on the extra axiom 
\cite[Def. 1.1]{BM11}(iv).

Since axioms (i),(ii),(iii) in Definition \ref{GENRED DEF} repeat \cite[Def. 1.1]{BM11}(i),(ii),(iii), we will only briefly recall the definitions of latching and matching objects.
Writing $\iota_n \colon \mathbb{R}_{\leq n} \to \mathbb{R}$ for the inclusion of the full subcategory of those $r \in \mathbb{R}$ with $|r|\leq n$, we have adjunctions
\begin{equation}\label{SKELADJ EQ}
\begin{tikzcd}[column sep =5em]
	\mathcal{C}^{\mathbb{R}} \ar{r}[swap]{\iota_n^{\**}} 
	&
	\mathcal{C}^{\mathbb{R}_{\leq n}}
	\ar[bend right]{l}[swap,midway]{\iota_{n,!}}
	\ar[bend left]{l}{\iota_{n,\**}}
\end{tikzcd}
\end{equation}
One then defines \textit{$n$-skeleta}
by $\mathsf{sk}_n A = \iota_{n,!} \iota_n^{\**} A$
and \textit{$n$-coskeleta}
by $\mathsf{csk}_n A = \iota_{n,\**} \iota_n^{\**} A$
as well as 
\textit{$r$-latching objects} by
$L_r A = \left(\mathsf{sk}_{|r|-1} A\right)_r$
and 
\textit{$r$-matching objects} by
$M_r A = \left(\mathsf{csk}_{|r|-1} A\right)_r$.
Axioms (i),(ii),(iii) then imply that
$\mathsf{sk}_n A$ (resp. $\mathsf{csk}_n A$)
depends only on the restriction to 
$\mathbb{R}_{\leq n}^{+}$ (resp. $\mathbb{R}_{\leq n}^{-}$). We refer the reader to \cite[\S 4,\S 6]{BM11} for a detailed discussion.

We now turn to the model categorical analysis, which depends on the Reedy-admissibility condition (iv)
in Definition \ref{GENRED DEF}, and is the actual novelty of this appendix.
We first recall the following, cf. \cite[Props. 6.5 and 6.6]{BP17}.
\begin{proposition}
      \label{FGTRL_PROP}
Let $\phi \colon G \to \bar{G}$ be a homomorphism and
$\mathcal{F}$, $\bar{\mathcal{F}}$ families of subgroups of
$G, \bar{G}$. Then the leftmost (resp. rightmost) adjunction below
is a Quillen adjunction 
\[
	\bar{G} \cdot_G (\minus)
	\colon \C^G_{\mathcal{F}}
		\rightleftarrows
	\C^{\bar{G}}_{\bar{\mathcal{F}}} \colon
	\mathsf{res}^{\bar{G}}_G
\qquad
	\mathsf{res}^{\bar{G}}_G
	\colon	\C^{\bar{G}}_{\bar{\mathcal{F}}}
		\rightleftarrows
	\C^G_{\mathcal{F}} \colon
	\mathsf{Hom}_G(\bar{G},\minus)
\]
provided that for $H \in \mathcal{F}$ it is
$\phi(H) \in \bar{\mathcal{F}}$
(resp. for $\bar{H} \in \bar{\mathcal{F}}$ it is
$\phi^{-1}(H) \in \mathcal{F}$).
\end{proposition}

\begin{proof}
For $H \leq G$, $\bar{H} \leq \bar{G}$, the fixed point formulas
$\left(\mathsf{res}^{\bar{G}}_G Y \right)^H = Y^{\phi(H)}$
and
$\left(\mathsf{Hom}_{G}(\bar{G},X)\right)^{\bar{H}}=
\prod_{[\bar{g}] \in \phi(G) \backslash \bar{G}/\bar{H}}
X^{\phi^{-1}(\bar{g}\bar{H}\bar{g}^{-1})}$
show that the right adjoints preserve (trivial) fibrations.
\end{proof}

For $\mathcal{F}_{\text{all}}$
the family of all subgroups of $G$
the model structure
$\C^G_{\mathcal{F}_{\text{all}}}$ is called the 
\textit{fine/genuine model structure}.
We regard this as the default model structure, 
and hence denote it simply as $\C^G$.

\begin{corollary}\label{RESGEN COR}
For any homomorphism $\phi \colon G \to \bar{G}$, the functor
$\mathsf{res}^{\bar{G}}_G \colon 
\C^{\bar{G}} \to \C^G$
preserves all four genuine classes of 
cofibrations, trivial cofibrations, fibrations and trivial fibrations.
\end{corollary}

\begin{corollary}\label{FTRIVALL COR}
For any $\mathcal{F}$, a $\mathcal{F}$-(trivial) cofibration is also a genuine
$G$-(trivial) cofibration.	
\end{corollary}

\begin{remark}\label{COFNOTRES REM}
	It is clear
	from the definitions
	that 
	$X \to Y$ is a $\mathcal{F}$-weak equivalence/$\mathcal{F}$-fibration iff it is a genuine
	$H$-weak equivalence/$H$-fibration
	for each $H \in \mathcal{F}$.
	However, while a $\mathcal{F}$-cofibration is necessarily a genuine $H$-cofibration for each $H \in \mathcal{F}$, the converse is rarely true.
	
For example, consider $\mathcal{C} = \mathsf{sSet}$ and
	$\mathcal{F}=\{\{e\}\}$ the family consisting only of the trivial subgroup. Then the $\mathcal{F}$-cofibrant objects
	turn out to be the $G$-free objects of $\mathsf{sSet}^G$, but all objects are
	genuine $\{e\}$-cofibrant, since this just means that their restriction to
	$\mathsf{sSet}^{\{e\}}=\mathsf{sSet}$ is cofibrant.
%
\end{remark}

Lemmas \ref{BLALIFT LEM} and \ref{BLAFACT LEM} below formalize straightforward arguments implicit in the proofs of \cite[Lemma 5.2]{BM11} and \cite[Thm 1.6]{BM11}.

\begin{definition}
Consider a commutative diagram
\begin{equation}\label{BLA EQ}
	\begin{tikzcd}
		A \ar{r} \ar{d} & X \ar{d}
	\\
		B \ar{r} \ar[dashed]{ru} & Y
	\end{tikzcd}
\end{equation}
in $\C^{\mathbb{R}}$. A collection of maps 
$f_r \colon B_r \to X_r$ for $|r|\leq n$ 
that induce a lift of the restriction to $\C^{\mathbb{R}_{\leq n}}$ is called a \textit{$n$-partial lift}.

Similarly, given a map $f\colon X \to Y$ in $\mathcal{C}^{\mathbb{R}}$, 
a factorization 
$\iota_{n}^{\**}X \to A \to \iota_{n}^{\**}Y$
of $\iota_{n}^{\**} f$ in 
$\mathcal{C}^{\mathbb{R}_{\leq n}}$
is called a \textit{$n$-partial factorization}.
\end{definition}

\begin{lemma}\label{BLALIFT LEM}
	Let $\C$ be any bicomplete category, and consider a commutative diagram as in \eqref{BLA EQ}. Then any $(n-1)$-partial lift uniquely induces commutative diagrams
\begin{equation}\label{BLALIFT EQ}
	\begin{tikzcd}
		A_r \amalg_{L_r A} L_r B \ar{r} \ar{d} & X_r \ar{d}
	\\
		B_r \ar{r} \ar[dashed]{ru} & Y_r \times_{M_r Y} M_r X
	\end{tikzcd}
\end{equation}
in $\mathcal{C}^{\mathsf{Aut}(r)}$
for each $r \in \mathbb{R}$ such that $|r|=n$. Furthermore, extensions of the 
$(n-1)$-partial lift to a $n$-partial lift are in bijection with choices of $\mathsf{Aut}(r)$-equivariant lifts in the diagrams \eqref{BLALIFT EQ} for $r$ ranging over representatives of the isomorphism classes of $r \in \mathbb{R}$ with $|r|=n$.
\end{lemma}

\begin{lemma}\label{BLAFACT LEM}
	Let $\C$ be any bicomplete category.
Then extensions
\[
\begin{tikzcd}
	\mathbb{R}_{\leq n-1} \ar{r}{A} \ar[hookrightarrow]{d}&
	\mathcal{C}
\\
	\mathbb{R}_{\leq n} \ar[dashed]{ru}[swap]{\tilde{A}}
\end{tikzcd}
\]
are determined uniquely up to unique isomorphism (in $\mathcal{C}^{\mathbb{R}_{\leq n}}$)
by choices of $\mathsf{Aut}(r)$-equivariant factorizations
\[
\left(\iota_{n-1,!}A \right)_r \dashedrightarrow
\tilde{A}_r \dashedrightarrow
\left(\iota_{n-1,\**}A\right)_r
\]
for $r$ ranging over representatives of the isomorphism classes of $r \in \mathbb{R}$ with $|r|=n$.
	
More generally, given a map 
$X \to Y$ in $\mathcal{C}^{\mathbb{R}}$,
extensions of a $(n-1)$-partial factorization
$\iota_{\leq n-1}^{\**}X \to A \to \iota_{\leq n-1}^{\**} Y$ in $\mathcal{C}^{\mathbb{R}_{\leq n-1}}$
to a $n$-partial factorization
$\iota_{\leq n}^{\**}X \to \tilde{A} \to \iota_{\leq n}^{\**} Y$ in $\mathcal{C}^{\mathbb{R}_{\leq n}}$
are determined uniquely up to unique isomorphism 
by choices of $\mathsf{Aut}(r)$-equivariant factorizations
\[
	X_r \amalg_{L_r X} \left(\iota_{n-1,!}A \right)_r \dashedrightarrow
	\tilde{A}_r \dashedrightarrow
	Y_r \times_{M_r Y} \left(\iota_{n-1,\**}A\right)_r
\]
for $r$ ranging over representatives of the isomorphism classes of $r \in \mathbb{R}$ with $|r|=n$.
\end{lemma}

In the next result, by $\{\mathcal{F}_r\}$-cofibration/trivial cofibration/fibration/trivial fibration 
we mean a map as described in 
Theorem \ref{REEDYADM THM}, regardless of whether such a model structure exists.

\begin{corollary}\label{BLALIFT COR}
Let $\mathbb{R}$ be generalized Reedy and 
$\{\mathcal{F}_r\}$ an arbitrary family of subgroups of $\mathsf{Aut}(r)$, $r \in \mathbb{R}$.
Then a map in $\mathcal{C}^{\mathbb{R}}$ 
is a $\{\mathcal{F}_r\}$-cofibration (resp. trivial cofibration) iff it has the left lifting property 
with respect to all 
$\{\mathcal{F}_r\}$-trivial fibrations (resp. fibrations),
and vice-versa for the right lifting property.
\end{corollary}

\begin{lemma}\label{GINJ LEM}
Let $\mathbb{S}$ be generalized Reedy with $\mathbb{S}=\mathbb{S}^+$, $K$ a group, and $\pi \colon \mathbb{S} \to K$ a functor.

Consider a map $A \to B$ in $\C^{\mathbb{S}}$ such that for all 
$s \in \mathbb{S}$
the maps 
$
  A_s \amalg_{L_s A} L_s B \to B_s
$	
are $\mathsf{Aut}(s)$-genuine (resp. trivial) cofibrations. 
Then $\mathsf{Lan}_{\pi\colon \mathbb{S} \to K}(A \to B)$
is a $K$-genuine (trivial) cofibration.
\end{lemma}

\begin{proof}
By adjunction, one needs only show that for any 
$K$-fibration $X \to Y$ in $\mathcal{C}^K$,
the map $\pi^{\**}(X \to Y)$
has the right lifting property against all maps $A \to B$ in $\C^{\mathbb{S}}$ as in the statement.
By Corollary \ref{BLALIFT COR}, it thus suffices to check
that the maps
\[
	(\pi^{\**} X)_s \to 
	(\pi^{\**} Y)_s \times_{M_s \pi^{\**} Y} M_s \pi^{\**} X
\]
are $\mathsf{Aut}(s)$-fibrations. But since $M_s Z = \**$ 
(recall $\mathbb{S}=\mathbb{S}^+$)
this map is just $X \to Y$ with the $\mathsf{Aut}(s)$-action induced by
$\pi \colon \mathsf{Aut}(s) \to K$, hence 
Corollary \ref{RESGEN COR} finishes the proof.
\end{proof}

\begin{lemma}\label{GINJMIN LEM}
Let $\mathbb{S}$ be generalized Reedy with $\mathbb{S}=\mathbb{S}^-$, $K$ a group, and $\pi \colon \mathbb{S} \to K$ a functor.

Consider a map $X \to Y$ in $\C^{\mathbb{S}}$ such that for all 
$s \in \mathbb{S}$
the maps 
$
	X_s \to Y_s \times_{M_s Y} M_s X
$	
are $\mathsf{Aut}(s)$-genuine (resp. trivial) fibrations.
Then 
$\mathsf{Ran}_{\pi\colon \mathbb{S} \to K}(A \to B)$
is a $K$-genuine (trivial) fibration.
\end{lemma}

\begin{proof}
This follows dually to the previous proof.
\end{proof}

\begin{remark}
      \label{GINJ REM}
Lemmas \ref{GINJ LEM} and \ref{GINJMIN LEM} generalize key parts of the proofs of \cite[Lemmas 5.3 and 5.5]{BM11}.  
The duality of their proofs reflects the duality in 
Corollary \ref{RESGEN COR}.
\end{remark}

\begin{notation}\label{KLTIMES NOT}
Given a subgroup $K \leq \mathsf{Aut}_{\Cat}(\mathbb{R})$ of the group of automorphisms of $\mathbb R$ as a category,
we write $K \ltimes \mathbb{R} \to K$ 
for the category obtained from $\mathbb{R}$ by formally adding action arrows
$r \to kr$ for $r \in \mathbb{R},k \in K$,
which satisfy three natural relations:
the maps $r \xrightarrow{e} e r$
are the identity $id_r$, and the two diagrams below commute
for all $k,\bar{k} \in K$ and maps $f: r \to r'$ in $\mathbb R$.
%
\[
      \begin{tikzcd}
            r \arrow[rr, "k"] \arrow[dr, "\bar{k}k"']
            &&
            k r \arrow[dl, "\bar{k}"]
            & 
            r \arrow[r, "k"] \arrow[d, "f"']
            &
            k r \arrow[d, "k f"]
            \\
            &
            \bar{k} k r
            &
            & 
            r' \arrow[r, "k"]
            &
            k r'
      \end{tikzcd}
\]
Note that 
if $\mathbb{R}=G$ is a group (regarding groups as one object categories) this recovers the usual $K \ltimes G$ construction and
if $K$ acts trivially on $\mathbb{R}$ this recovers the product of categories $K \times \mathbb{R}$.

Alternatively, $K \ltimes \mathbb{R}$
admits a succinct formal definition: 
if one regards $\mathbb{R}$ with its $K$-action as a
functor $ 
K \xrightarrow{\mathbb{R}} \mathsf{Cat}$, 
 the associated Grothendieck fibration is the natural functor $K \ltimes \mathbb{R} \to K$.



We note that for a functor
$A\colon K \ltimes \mathbb{R} \to \mathcal{C}$
the Kan extension
$\mathsf{Ran}_{K \ltimes \mathbb{R} \to K} A$
(resp. $\mathsf{Lan}_{K \ltimes \mathbb{R} \to K} A$)
is simply the (co)limit
$\lim_{\mathbb{R}} A$
(resp. $\colim_{\mathbb{R}} A$)
together with the naturally induced $K$-action.
\end{notation}

\begin{notation}\label{RPLUSR NOT}
We borrow the $\mathbb{R}^+(r)$ notation introduced after \cite[Rem. 1.5]{BM11} for the full subcategory of the overcategory $\mathbb{R} \downarrow r$
with objects the arrows 
$r' \xrightarrow{+} r$ in $\mathbb{R}^+$
with $|r'|<|r|$.
One has a \emph{domain} 
functor $d\colon \mathbb{R}^+(r) \to \mathbb{R}$ given by 
$(r'\to r) \mapsto r'$ such that
(cf. the discussion after \cite[Rem. 1.5]{BM11} and 
\cite[Lemma 4.4(i)]{BM11}):
\begin{enumerate}[label=(\roman*)]
\item by Definition \ref{GENRED DEF}(iii)
arrows in $\mathbb{R}^+(r)$ forget under $d$ to arrows in $\mathbb{R}^+$, so that
$\mathbb{R}^+(r)$ becomes generalized Reedy with
$|r' \to r| = |r'|$,
$\left(\mathbb{R}^+(r)\right)^+ = \mathbb{R}^+(r)$,
$\left(\mathbb{R}^+(r)\right)^- = \mathsf{Iso}\left(\mathbb{R}^+(r)\right)$;
\item for $(r' \to r) \in \mathbb{R}^+(r)$ the functor $d$ induces isomorphisms
$\left(\mathbb{R}^+(r)\right)^+(r'\to r) \xrightarrow{\simeq} \mathbb{R}^+(r')$;
\item for $A \in \mathcal{C}^{\mathbb{R}}$, $(r' \to r) \in \mathbb{R}^+(r)$
there are natural isomorphisms 
$L_{r' \to r} \left(d^{\**} A \right) \simeq L_{r'} A$;
\item for $A \in \mathcal{C}^{\mathbb{R}}$ there are natural isomorphisms  $L_r A \simeq \mathop{\colim}_{\mathbb{R}^+(r)} 
\left(d^{\**}A \right)$.
\end{enumerate}
Moreover, writing $\mathbb{R}^-(r)$ for the full subcategory of $r\downarrow \mathbb{R}$ with objects the 
$r \xrightarrow{-} r'$ in $\mathbb{R}^-$ with $|r'|<|r|$, 
one has a \emph{target} functor $t \colon \mathbb{R}^-(r) \to \mathbb{R}$ satisfying the dual properties of (i),(ii),(iii),(iv).
\end{notation}

\begin{remark}
	Lemma \ref{GINJ LEM} will be applied when
	$K \leq \mathsf{Aut}_{\mathbb{R}}(r)$ and
	$\mathbb{S} = K \ltimes \mathbb{R}^+(r)$ for $\mathbb{R}$ a given generalized Reedy category and $r \in \mathbb{R}$.
	Similarly, Lemma \ref{GINJMIN LEM} will be applied when
	$\mathbb{S} = K \ltimes \mathbb{R}^-(r)$.
	It is straightforward to check that in the $\mathbb{R}^+$ (resp. $\mathbb{R}^-$) case
	maps in $\mathbb{S}$ are identified with squares as on the left (right)
\begin{equation}\label{ARRINS EQ}
	\begin{tikzcd}
		r'' \ar{r}{+} \ar{d}[swap]{+} & r \ar{d}{\simeq}
	& &
		r \ar{r}{-} \ar{d}[swap]{\simeq} & r' \ar{d}{-}
	\\
		r' \ar{r}[swap]{+} & r
	& &
		r \ar{r}[swap]{-} & r''
	\end{tikzcd}
\end{equation}
where maps labelled $+$ are in $\mathbb{R}^+$,
maps labelled $-$ are in $\mathbb{R}^-$,
the horizontal maps are non-invertible, and the maps labelled $\simeq$ are automorphisms in $K$. 
Using the projection $\pi \colon \mathbb{S} \to K$ and the domain (resp. target) functor
$d \colon \mathbb{S} \to \mathbb{R}$ 
($t \colon \mathbb{S} \to \mathbb{R}$), 
Lemma \ref{GINJ LEM} (\ref{GINJMIN LEM}) will be applied to maps  
$d^{\**}A \to d^{\**} B$ ($t^{\**}A \to t^{\**} B$) in $\mathcal{C}^{\mathbb{S}}$
induced from maps
$A \to B$ in $\C^{\mathbb{R}}$ so that Notations \ref{KLTIMES NOT} and \ref{RPLUSR NOT}(iv) give the natural identifications below, compatibly with the $K$-actions.
\[
\mathsf{Lan}_{\pi} d^{\**} (A \to B) \simeq
(L_r A \to L_r B),
	\qquad
\mathsf{Ran}_{\pi} t^{\**} (A \to B) \simeq
(M_r A \to M_r B).
\]
Lastly, note that (in the $\mathbb{R}^+$ case) the natural inclusion
$\left(\mathbb{R}^+(r)\right)^+(r'\to r) \xrightarrow{\iota} \mathbb{S}^+(r \to r')$ is an equivalence of categories (indeed, by adding a $\searrow$ arrow to the left square in \eqref{ARRINS EQ}, the bottom left triangle becomes an object in the source of $\iota$ and the top right triangle an isomorphism in the target of $\iota$). 
Hence Notation \ref{RPLUSR NOT}(iii) (and its dual) also apply to $\mathbb{S}$.
\end{remark}

We can now 
prove the following lemmas, which are the key to Theorem \ref{REEDYADM THM}
(to simplify our proofs, 
the lemmas are stated 
as dually as possible, but care is needed; see Remarks \ref{COFNOTRES REM}, \ref{NOTDUAL REM}).

\begin{lemma}\label{REEDYTRCOF LEM}
Let $\mathbb{R}$ be generalized Reedy and 
$\{\mathcal{F}_r\}_{r \in \mathbb{R}}$ a Reedy-admissible family.

Suppose further that $A \to B$ is a $\{\mathcal{F}_r\}$-Reedy cofibration
in $\mathcal{C}^{\mathbb{R}}$. Then:
\begin{itemize}
	\item[(i)] for all $r \in \mathbb{R}$ and $H \in \mathcal{F}_r$
	the maps $L_r A \to L_r B$ and $A_r \to B_r$ are
	$H$-genuine cofibrations;
	\item[(ii)] the maps $A_r \to B_r$ are all $\{\mathcal{F}_r\}$-weak equivalences iff so are the maps $A_r \amalg_{L_r A} L_r B \to B_r$.
\end{itemize}
\end{lemma}

\begin{proof}
We start with (ii), arguing by induction on $|r|$
that the analogue claim
restricted to
$|r|\leq n$ holds.
The $n=0$ case is obvious. 
Otherwise, if suffices to show that, for each $r$ 
with $|r|=n$
and $H \in \mathcal{F}_r$,
the map
$A_r \to B_r$ is a $H$-genuine weak equivalence iff 
so is $A_r \amalg_{L_r A} L_r B \to B_r$.

We apply Lemma \ref{GINJ LEM} with 
$K = H$ and 
$\mathbb{S} = H \ltimes \mathbb{R}^+(r)$
to the map $d^{\**}A \to d^{\**}B$. 
By Notation \ref{RPLUSR NOT}(iii), the hypothesis to be checked is that, for each $(r' \to r) \in \mathbb{S}$,
the map $A_{r'} \amalg_{L_{r'}A} L_{r'}B \to B_{r'}$ is a genuine trivial cofibration for the group
$\mathsf{Aut}_{\mathbb{S}}(r' \to r)$.
But by induction this map is a $\mathcal{F}_{r'}$-trivial cofibration, so this holds by Corollaries \ref{FTRIVALL COR} and \ref{RESGEN COR}.
%
%
%
It now follows that
the maps labelled $\sim$ in
\begin{equation}\label{WHATEV2 EQ}
\begin{tikzcd}[row sep=10]
   L_r A \ar[>->]{r}{\sim} \ar[d]  & 
   L_r B \ar[d] & 
\\
   A_r \ar[>->]{r}{\sim}&  A_r \amalg_{L_r A} L_r B 
    \ar[>->]{r} &
   B_r
\end{tikzcd}
\end{equation}
are $H$-genuine trivial cofibrations, finishing the proof of (ii).

Lastly, (i) follows from a streamlined version of the argument above (with no induction needed), 
with the same instance of Lemma \ref{GINJ LEM} showing that the maps 
$L_r A \to L_r B$ are $H$-genuine cofibrations and thus,
by again considering \eqref{WHATEV2 EQ},
that so are the maps $A_r \to B_r$. 
\end{proof}

\begin{lemma}\label{REEDYTRFIB LEM}
Let $\mathbb{R}$ be generalized Reedy and 
$\{\mathcal{F}_r\}_{r \in \mathbb{R}}$ a Reedy-admissible family.

Suppose further that $X \to Y$ is a $\{\mathcal{F}_r\}$-Reedy fibration
in $\mathcal{C}^{\mathbb{R}}$. Then:
\begin{itemize}
\item[(i)] for all $r \in \mathbb{R}$ and $H \in \mathcal{F}_r$
	the maps $M_r X \to M_r Y$ and $X_r \to Y_r$ are
	$H$-genuine fibrations;
\item[(ii)] the maps $X_r \to Y_r$ are all $\{\mathcal{F}_r\}$-weak equivalences iff so are the maps $X_r \to Y_r \times_{M_r Y} M_r X$.
\end{itemize}
\end{lemma}

\begin{proof}
For (ii) we again argue by induction on $|r|$. We need to
show that, for $r$ with $|r|=n$ and $H \in \mathcal{F}_r$, the map
$X_r \to Y_r$ is a $H$-genuine weak equivalence iff 
so is $X_r \to Y_r \times_{M_r Y} M_r X$.

We apply Lemma \ref{GINJMIN LEM} with 
$K = H$ and 
$\mathbb{S} = H \ltimes \mathbb{R}^-(r)$
to the map $t^{\**}A \to t^{\**}B$. 
By the dual of Notation \ref{RPLUSR NOT}(iii), the hypothesis to be checked is that, for each $(r \to r') \in \mathbb{S}$, 
the map $X_{r'} \to Y_{r'} \times_{M_{r'}Y} M_{r'} X$ is a genuine trivial fibration for the group
$\mathsf{Aut}_{\mathbb{S}}(r \to r') \simeq \pi_r^{-1}(H)$,
where $\pi_r$ is as in \eqref{PIDEFR EQ}.
But since $\{\mathcal{F}_r\}$ is Reedy-admissible (cf. Definition \ref{GENRED DEF}(iv)), we know by induction that this map is a 
$\pi_{r'}\left(\pi_r^{-1}(H)\right)$-trivial fibration, and this suffices by Corollary \ref{RESGEN COR}.
%
%
It now follows that the maps labelled $\sim$ in
\[
\begin{tikzcd}[row sep=10]
	X_r \ar[->>]{r}&
	Y_r \times_{M_r Y} M_r X \ar[d] \ar[->>]{r}{\sim} & 
	Y_r \ar[d]
\\
	&
	M_r X \ar[->>]{r}{\sim} &
	M_r Y
\end{tikzcd}
\]
are $H$-genuine trivial fibrations, finishing the proof of (ii).
(i) again follows similarly.
\end{proof}


\begin{remark}\label{NOTDUAL REM}
The proofs of Lemmas \ref{REEDYTRCOF LEM} and \ref{REEDYTRFIB LEM}
are similar, but not dual, since
Lemma \ref{REEDYTRFIB LEM} uses Reedy-admissibility 
while Lemma \ref{REEDYTRCOF LEM} does not.
This reflects the difference in the proofs of 
\cite[Lemmas 5.3 and 5.5]{BM11} as discussed in 
\cite[Remark 5.6]{BM11}, albeit with a caveat.

Setting $K=\{e\}$ in Lemma \ref{GINJ LEM} yields that
$\mathsf{lim}_{\mathbb{S}} (A \to B)$ is a cofibration provided that $A \to B$ is a genuine Reedy cofibration, i.e. a Reedy cofibration for $\{\mathcal{F}_{\text{all}}\}$ the families of all subgroups. 
On the other hand, the proof of \cite[Lemma 5.3]{BM11} argues that 
$\mathsf{lim}_{\mathbb{S}} (A \to B)$ is a cofibration provided that $A \to B$ is a projective Reedy cofibration, i.e. a Reedy cofibration for $\{\{e\}\}$ the collection of trivial families 
(note that all projective cofibrations are genuine cofibrations, so that our claim is more general).
Since the cofibration half of the 
analogue of Corollary \ref{RESGEN COR}
for projective model structures only holds if 
$\phi \colon G \to \bar{G}$ is injective, 
the proof of \cite[Lemma 5.3]{BM11} also requires an injectivity check that is not needed in our proof of Lemma \ref{REEDYTRCOF LEM}.
\end{remark}

\begin{proof}[proof of Theorem \ref{REEDYADM THM}]
Lemmas \ref{REEDYTRCOF LEM}(ii) and \ref{REEDYTRFIB LEM}(ii) say that the characterizations of trivial cofibrations (resp. trivial fibrations) in the statement of Theorem \ref{REEDYADM THM} are correct, i.e. that they describe the maps that are both cofibrations (resp. fibrations) and weak equivalences.	

	We refer to the model category axioms in \cite[Def. 1.1.3]{Hov99}. 	
	Both 2-out-of-3 and the retract axioms are immediate
(note that retracts commute with Kan extensions).	
	The lifting axiom follows from Corollary \ref{BLALIFT COR}
	while the task of building factorizations $X \to A \to Y$ of a given map $X \to Y$ follows by a similar standard argument from Lemma \ref{BLAFACT LEM} by iteratively factoring the maps
\[
	X_r \amalg_{L_r X} L_r A \to Y_r \times_{M_r Y} M_r A
\]
in $\mathcal{C}^{\mathsf{Aut}(r)}$, 
thus building both $A$ and the factorization inductively
(recall that $L_r A$, $M_r A$ depend only on 
the restriction $\iota_{|r|-1}^{\**} A$, and are thus well defined in each inductive step).
\end{proof}

\begin{remark}\label{ITERREEDY REM}
If $\mathbb{R}$, $\mathbb{S}$ are generalized Reedy categories,
then 
$\mathbb{R} \times \mathbb{S}$ is again generalized Reedy 
with $|(r,s)| = |r|+|s|$,  
$\left(\mathbb{R} \times \mathbb{S}\right)^+ =
\mathbb{R}^+ \times \mathbb{S}^+$,
$\left(\mathbb{R} \times \mathbb{S}\right)^- =
\mathbb{R}^- \times \mathbb{S}^-$.
Moreover, given corresponding Reedy admissible families
$\{\mathcal{F}_r\}_{r \in \mathbb{R}}$,
$\{\mathcal{F}_s\}_{s \in \mathbb{S}}$,
one has induced admissible families 
$\{\mathcal{F}_{(r,s)}\}_{(r,s) \in \mathbb{R} \times \mathbb{S}}$
where 
$\mathcal{F}_{(r,s)} =
\{H \times \bar{H} \colon H \in \mathcal{F}_r,\bar{H} \in \mathcal{F}_s\}$.
It is then straightforward to check that the $\{\mathcal{F}_{(r,s)}\}$-Reedy model structure on 
$\mathcal{C}^{\mathbb{R} \times \mathbb{S}}
\simeq (\mathcal{C}^{\mathbb{R}} )^{\mathbb{S}}$
can also be obtained iteratively as the
$\{\mathcal{F}_{s}\}$-Reedy model structure over the 
$\{\mathcal{F}_{r}\}$-Reedy model structure over the model structure on $\mathcal{C}$ (and vice-versa).

We note that, given a map $f \colon A \to B$ in $\mathcal{C}$
and writing $l_{(-)}f$
for the relative latching map
$A_{(-)} \amalg_{L_{(-)}A} L_{(-)}B
\to B_{(-)}$,
one has $l_{(r,s)} f \simeq l_r l_s f \simeq l_s l_r f$
and dually for relative matching maps.
\end{remark}

Lastly, we briefly discuss the cofibrant generation of equivariant Reedy model structures.

For $K\in \mathsf{Set}, X \in \mathcal{C}$ we write
$K \cdot X = \coprod_K X$ and
$\{K,X\}=X^{\times K}$ for the standard tensoring and cotensoring of $\mathcal{C}$ over $\mathsf{Set}$. One then has a two-variable adjunction
\[
	\mathsf{Set}^{\mathbb{R}} \times \mathcal{C}
	\xrightarrow{(-)\cdot(-)}
	\mathcal{C}^{\mathbb{R}}
\qquad
	\mathcal{C}^{op} \times  \mathcal{C}^{\mathbb{R}}
	\xrightarrow{\mathcal{C}(-,-)}
	\mathsf{Set}^{\mathbb{R}}
\qquad
	\left(\mathsf{Set}^{\mathbb{R}}\right)^{op} \times \mathcal{C}^{\mathbb{R}}
	\xrightarrow{\{-,-\}_{\mathbb{R}}}
	\mathcal{C}
\]
where 
$\{K_{\bullet},X_{\bullet}\}_{\mathbb{R}} =
 \int_{r \in \mathbb{R}} \{K_r,X_r\}$
 is the end (sometimes called the weighted limit \cite[\S 7.1]{Ri14}).

The following is essentially \cite[Lemma 3.5]{RV14}, rewritten in our notation.

\begin{lemma} For $X \in \mathcal{C}^{\mathbb{R}}, r \in \mathbb{R}$ there is a natural identification
$
M_r X \simeq \{\mathsf{sk}_{|r|-1} \mathbb{R}(r,-),X \}_{\mathbb{R}}
$.
\end{lemma}

\begin{proof}
Abbreviating $\mathbb{R}_{\leq n-1}$ as $\mathbb{R}_{<n}$, this follows from the calculation
\begin{align}
	M_r X 
\simeq &
	\int_{\bar{r} \in \mathbb{R}_{<|r|}} 
\left\{\mathbb{R}(r,\bar{r}),X_{\bar{r}}\right\}
\simeq
	\int_{\bar{r} \in \mathbb{R}_{<|r|}} 
\left\{\mathbb{R}(r,\bar{r}),
	\int_{\hat{r} \in \mathbb{R}} \left\{\mathbb{R}(\bar{r},\hat{r}) ,X_{\hat{r}} \right\}\right\}\simeq
\\
\simeq&
	\int_{\bar{r} \in \mathbb{R}_{<|r|}} \int_{\hat{r} \in \mathbb{R}}
\left\{\mathbb{R}(r,\bar{r}),
	\left\{\mathbb{R}(\bar{r},\hat{r}) ,X_{\hat{r}} \right\}\right\}
\simeq
	\int_{\hat{r} \in \mathbb{R}} \int_{\bar{r} \in \mathbb{R}_{<|r|}}
\left\{\mathbb{R}(r,\bar{r})\times \mathbb{R}(\bar{r},\hat{r}) ,X_{\hat{r}} \right\}\simeq
\\
\simeq &
	\int_{\hat{r} \in \mathbb{R}} 
\left\{\int^{\bar{r} \in \mathbb{R}_{<|r|}}\mathbb{R}(r,\bar{r})\times \mathbb{R}(\bar{r},\hat{r}) ,X_{\hat{r}} \right\}
\simeq
	\int_{\hat{r} \in \mathbb{R}} 
\left\{\left(\mathsf{sk}_{|r|-1} \mathbb{R}(r,-)\right)(\hat{r}) ,X_{\hat{r}} \right\} =
\{\mathsf{sk}_{|r|-1} \mathbb{R}(r,-),X \}_{\mathbb{R}} 
\end{align}
where the second step is the Yoneda lemma, the fourth step is Fubini together with adjointness for $\{-,-\}$, and the sixth step is the coYoneda lemma \cite[Ex. 1.4.6]{Ri14}.
\end{proof}

Combining this lemma with the observation that, if a group $H$ acts on 
$K \in \mathsf{Set}^{\mathbb{R}}$,
then $\{K/H,X\}_{\mathbb{R}} = \left(\{K,X\}_{\mathbb{R}}\right)^H$ now yields the following.

\begin{proposition}\label{REEDYCOFGEN PROP}
Suppose the model category $\mathcal{C}$ is cofibrantly generated,
$\mathbb{R}$ is generalized Reedy and 
$\{\mathcal{F}_r\}_{r \in \mathbb{R}}$
is a Reedy-admissible collection of families.

Then the $\{\mathcal{F}_r\}$-Reedy model structure on 
$\mathcal{C}^{\mathbb{R}}$ is also cofibrantly generated.

More explicitly, letting 
$\mathcal{I}$ (resp. $\mathcal{J}$)
denote a set of generating (resp. trivial) cofibrations of $\mathcal{C}$,
the generating cofibrations for $\mathcal{C}^{\mathbb{R}}$
are the maps of the form
\[
\left(\left(\mathsf{sk}_{|r|-1} \mathbb{R}(r,-) \to
 \mathbb{R}(r,-)\right)/H \right)
\square
i,
\qquad
r \in \mathbb{R}, H \in \mathcal{F}_r, i \in \mathcal{I}
\]
while the generating trivial cofibrations are the maps of the form
\[
\left(\left(\mathsf{sk}_{|r|-1} \mathbb{R}(r,-) \to
 \mathbb{R}(r,-)\right)/H \right)
\square
j,
\qquad
r \in \mathbb{R}, H \in \mathcal{F}_r, j \in \mathcal{J}.
\]
\end{proposition}

\begin{example}\label{REEDYCOFGEN EX}
For $\mathbb{R} = \Delta^{op}$, it is well known that
$\mathsf{sk}_{n-1} \Delta^{op}(n,-) = \mathsf{sk}_{n-1} \Delta[n] = \partial \Delta[n]$.

Similarly, letting $\mathbb{R} = G \times \Omega^{op}$ as in 
Example \ref{FGRAPHREEDY EX} and letting 
$\Gamma \in \mathcal{F}_U$ be a graph subgroup encoding a $H$-action on $U$, then
\begin{align}
	\left( \mathsf{sk}_{|U|-1} \left( G \times \Omega^{op} \right)(U,-) \right)/\Gamma
= &
	\left( \mathsf{sk}_{|U|-1} \left( G \cdot \Omega[U] \right) \right)/\Gamma
=
	\left( G \cdot \mathsf{sk}_{|U|-1}  \Omega[U] \right)/\Gamma =
\\
= &
	\left( G \cdot \partial \Omega[U] \right)/\Gamma
=
	G \cdot_H \partial \Omega[U] = \partial \Omega[G \cdot_H U]
\end{align}
where the third equality is \cite{BM11} (see also \cite[Cor 5.63]{Per18})
and the last equality is \eqref{T_DECOMP_EQ}.
\end{example}

\bibliography{biblio}{}

\bibliographystyle{alpha}

\end{document}